\newif\ifimanumstyle
\imanumstylefalse

\ifimanumstyle
  \documentclass{imanum}
\else
  \documentclass{amsart}
\fi

\usepackage{graphicx}    
\usepackage{multicol}    
\usepackage{enumitem}
\usepackage{url}
\usepackage{hyperref}
\usepackage[capitalize]{cleveref}
\usepackage{verbatim}
\graphicspath{{gfx/}}

\renewcommand{\PackageWarningNoLine}[2]{}
\usepackage{amsmath}
\usepackage{amssymb}
\usepackage{amsthm}
\usepackage{amsfonts}
\usepackage{mathrsfs}
\usepackage{esint}
\usepackage{colonequals}                
\usepackage{color}
\usepackage{tikz}
\usetikzlibrary{arrows,shapes,positioning}
\usetikzlibrary{decorations.markings}

\usepackage{bm}
\usepackage{float}

\newcommand{\R}{\mathbb{R}}

\newcommand{\cJ}{\mathcal J}

\newcommand{\cM}{\mathcal M}

\newcommand{\vertiii}[1]{{\left\vert\kern-0.25ex\left\vert\kern-0.25ex\left\vert #1 
    \right\vert\kern-0.25ex\right\vert\kern-0.25ex\right\vert}}

\providecommand{\st}{\mathrel\vert}

\ifimanumstyle
\else
    \newtheorem{theorem}{Theorem}[section]
    \newtheorem{lemma}[theorem]{Lemma}
    \newtheorem{corollary}[theorem]{Corollary}
    \newtheorem{proposition}[theorem]{Proposition}

    \newtheorem{problem}[theorem]{Problem}

\fi
\providecommand{\pref}[1]{Problem~\ref{#1}}

\numberwithin{equation}{section} 
\numberwithin{figure}{section} 

\makeatletter
\def\ifempty#1{\def\@temp{#1}\ifx\@temp\@empty}
\makeatother

\newcounter{assumptionCounter}

\newcommand{\assitem}[1][]{%
    \ifempty{#1}%
        \renewcommand{\theassumptionCounter}{A\arabic{assumptionCounter}}%
        \refstepcounter{assumptionCounter}
    \else%
        \renewcommand{\theassumptionCounter}{#1}%
        \addtocounter{assumptionCounter}{-1}
        \refstepcounter{assumptionCounter}
    \fi%
    \item[(\theassumptionCounter)]
}

\newenvironment{assumptions}{
    \begin{list}{}
    {
        \setlength{\leftmargin}{13mm}
        \setlength{\itemindent}{0mm}
        \setlength{\labelwidth}{13mm}
    }
}{\end{list}}

\newcounter{todocounter}
\newlength{\todowidthinner}
\usepackage{xifthen}
\usepackage{marginnote}
\DeclareRobustCommand{\MyChange}[3][\empty]{%
  {\color{#2}#3}
  \ifthenelse{\isempty{#1}}{}
  {%
    \addtocounter{todocounter}{1}%
    \ifmmode%
        {\color{#2}\text{$^{\framebox{\arabic{todocounter}}}$}}%
    \else%
        {\color{#2}\text{$^{\arabic{todocounter}}$}}%
    \fi%
    \marginpar{\textcolor{#2}{$^{\arabic{todocounter}}$\textnormal{#1}}}%
  }%
}
\definecolor{lightgray}{rgb}{0.7,0.7,0.7}

\ifimanumstyle
    \usepackage{ulem}
\else
    \usepackage[normalem]{ulem}
\fi

\definecolor{cgcol}{rgb}{0,0.3,0.9}

\newcommand{\snote}[1]{ { \color{gray} #1 } }
\newcommand{\xnote}[1]{ }

\newcommand{\eps}{\varepsilon}
\newcommand{\norm}[1]{ \left\Vert #1 \right\Vert }
\newcommand{\abs}[1]{ \left\vert #1 \right\vert }

\newcommand{\wt}[1]{ \widetilde{#1} }

\begin{document}

\ifimanumstyle

    \title{Discretization Error Estimates for Penalty Formulations of a Linearized Canham--Helfrich Type Energy}

    \author{%
    {\sc
    Carsten Gr\"aser\thanks{Corresponding author. Email: graeser@mi.fu-berlin.de},
    and
    Tobias Kies\thanks{Email: tobias.kies@fu-berlin.de}} \\[2pt]
    Freie Universit\"at Berlin, Institut f\"ur Mathematik,\\
    Arnimallee~6, D-14195~Berlin, Germany}

    \shorttitle{Error Estimates for Penalty Formulations}

    \shortauthorlist{C. Gr\"aser and T. Kies}

    \begin{abstract}
    {%
\else

    \title[Error Estimates for Penalty Formulations]{Discretization Error Estimates for Penalty Formulations of a Linearized Canham--Helfrich Type Energy}

    \author[Gr\"aser]{Carsten Gr\"aser}
    \address{Carsten Gr\"aser\\
    Freie Universit\"at Berlin\\
    Institut f\"ur Mathematik\\
    Arnimallee~6\\
    D-14195~Berlin\\
    Germany}
    \email{graeser@mi.fu-berlin.de}

    \author[Kies]{Tobias Kies}
    \address{Tobias Kies\\
    Freie Universit\"at Berlin\\
    Institut f\"ur Mathematik\\
    Arnimallee~9\\
    {D-14195}~Berlin\\
    Germany}
    \email{tobias.kies@fu-berlin.de}

    \begin{abstract}
\fi%
This paper is concerned with minimization of a fourth-order linearized Canham--Helfrich
energy subject to Dirichlet boundary conditions on curves inside the domain.
Such problems arise in the modeling of the mechanical interaction of
biomembranes with embedded particles. There, the curve conditions
result from the imposed particle--membrane coupling.
We prove almost-$H^{\frac{5}{2}}$ regularity of the solution and
then consider two possible penalty formulations.
For the combination of these penalty formulations with a Bogner--Fox--Schmit
finite element discretization we prove discretization error estimates
which are optimal in view of the solution's reduced regularity.
The error estimates are based on a general estimate for linear penalty problems
in Hilbert spaces.
Finally, we illustrate the theoretical results by numerical computations.
An important feature of the presented discretization is that it does not
require to resolve the particle boundary. This is crucial in order to avoid
re-meshing if the presented problem arises as subproblem in a model where
particles are allowed to move or rotate.
\ifimanumstyle
    }
    {biomembrane model, Bogner-Fox-Schmit finite element, discretization error estimate, penalty method}
    \end{abstract}
\else
    \end{abstract}
\fi

\maketitle

\tableofcontents

\section{Introduction}\label{sec:introduction}


A standard model for the behavior of biomembranes on a macroscale is the
Canham--Helfrich model which describes a biological membrane mathematically
as a hypersurface $\cM\subseteq \R^3$ minimizing the
Canham--Helfrich energy
\begin{align*}
    \cJ_{CHS}(\cM) = \int_{\cM} \frac{1}{2} \kappa H^2 + \kappa_G K + \sigma  \, \mathrm{d}\mathcal{H}^2\text{.}
\end{align*}
Here $H$ and $K$ are the mean and Gaussian curvature of $\cM$,
the coefficients $\kappa>0$ and $\kappa_G\geq 0$ are the corresponding
bending rigidities, and $\sigma \geq 0$ is the membrane's surface tension.
Under the assumption that the membrane is ``almost~flat'' one can justify
a geometric linearization of this functional for a membrane patch,
leading to the so called  Monge--gauge approximation
\begin{align*}
    \cJ_{\Omega}(u) = \int_\Omega \frac{1}{2} \kappa (\Delta u)^2 + \frac{1}{2} \sigma \vert\nabla u\vert^2 \, \mathrm{d}x
\end{align*}
of the Canham--Helfrich energy. Here, the membrane patch is considered to be the
graph $\cM = \{(x,u(x) \st x \in \Omega \}$ of a function $u\colon \Omega \to \R$
over some reference domain $\Omega\subseteq \R^2$.

A variety of hybrid models for the coupling of embedded particles to the
membrane have been considered. These hybrid models are based on a continuous
surface description of the membrane while particles are described by discrete
entities
(see, e.g., \cite{BahramiEtAl2014, DomFou99, DomFou02, HelfrichJakobsson90, WeiKozHel98}).
For an overview on hybrid models we refer to \cite{ElGrHoKoWo15} and the references cited therein.
In the present paper we consider coupling conditions imposed on the particle
boundaries and follow the notation introduced in \cite{ElGrHoKoWo15}.
There the reference domain $\Omega$ is split into the membrane's domain
$\Omega_B \subseteq \Omega$ and the particles' domain $B = \Omega \setminus \Omega_B$.
The model introduces membrane--particle interactions by functions which
prescribe the membrane's height profile and slope on the interface
$\Gamma := \partial \Omega_B \cap \partial B$.
Altogether this yields an energy minimization problem subject to Dirichlet
boundary conditions on $\Gamma$.

The present paper introduces and analyzes a discretization based on a penalty
formulation for the corresponding boundary value problem that avoids the
resolution of particle boundaries.
The reason for considering penalized boundary conditions is the following:
Since a biological membrane behaves like a fluid in tangential directions,
particles can in principle move and rotate in plane. Any model
simulating moving particles and any method computing optimal particle
positions will thus have to solve multiple problems with varying particle
positions. If the particle boundaries would have to be resolved this
would require mesh-deformation or re-meshing which can be computationally quite expensive.
An alternative is to replace the strict boundary conditions by adequate
penalty terms that can be formulated without having to resolve the boundary.
One such penalty approach has been introduced in \cite{ElGrHoKoWo15} and
is called \emph{soft curve formulation}.
A novel second formulation in this paper will be the \emph{soft bulk formulation}.
Both penalty problems will be defined later on.

Our goal is to derive discretization error estimates for those penalty problems.
For second order equations such estimates are well known, see for example
\cite{Babushka73}, \cite{BarEll86} and \cite{Nitsche71}.
It turns out that the fourth order problems that we are interested in can be
treated sufficiently well using a simple general framework for penalty problems
in Hilbert spaces.
This is mostly due to the fact that the regularity of our solutions is rather
limited in first place and so we can use rather simple estimates to still
obtain optimal rates of convergence.
Thanks to the abstract formulations we get as a byproduct a general error
theory for finite element penalty problems with low regularity.

The paper is structured as follows:
In \cref{sec:notationAndProblems} we introduce the notations and problem
formulations that are used throughout this paper.
\cref{sec:abstractError} is devoted to an abstract error result for linear penalty approximations on Hilbert spaces.
As a foundation for the application of this result to the problem at hand
we then discuss regularity of solutions in \cref{sec:regularity}.
There we prove that a solution of the original problem lies in
$H^{\frac{5}{2}-\delta}(\Omega)$ for all $\delta>0$. 
In \cref{sec:discreteError} we combine the regularity with the
developed abstract results to show the optimal (in view of the restricted regularity)
convergence rate $O(h^{1/2-\delta})$ for a discretization with
Bogner--Fox--Schmit finite elements.
Finally, \cref{sec:computations} illustrates our results by numerical examples
that reproduce our theoretical findings.


\section{Notation and Problem Formulations}\label{sec:notationAndProblems}
We consider a membrane with $k$ embedded particles. To this end
let $\Omega \subset \R^2$ be a bounded reference domain with Lipschitz
boundary $\partial \Omega$ and $B_i \subset \Omega$, $i=1,\dots, k$
the area occupied by the $i$-th particle. For simplicity we assume
that the $B_i$ are closed, nonempty, connected, and pairwise disjoint.
Furthermore we require that each particle $B_i$ has a $C^{1,1}$-boundary.
By $B = \bigcup_{i=1}^k B_i$ and $\Omega_B = \Omega \setminus B$
we denote the area occupied by the particles and the membrane, respectively.
Since the $B_i$ are closed and disjoint the total membrane--particle interface
is given by
\begin{align*}
    \Gamma
        = \bigcup_{i=1}^k \partial B_i
        = \partial \Omega_B \setminus \partial \Omega
        = \partial \Omega_B \cap B.
\end{align*}
For simplicity we consider the space
\begin{align*}
    H = \{v \in H^2(\Omega_B) \st v|_{\partial \Omega} = \partial_\nu v|_{\partial\Omega} = 0 \}
\end{align*}
with homogeneous Dirichlet boundary conditions on the boundary of $\partial \Omega$.
Notice that we can treat Navier and periodic boundary conditions on $\partial \Omega$
using the same techniques (cf. \cite{ElGrHoKoWo15}).

We assume that the membrane--particle interaction is governed by
boundary conditions on the interface. More precisely, each particle
$B_i$ enforces a height profile given by
$f^i_1\colon \partial B_i \to \R$ and a slope given by
$f^i_2\colon \partial B_i \to \R$ to the membrane on its
boundary $\partial B_i$.
I.\,e. we consider boundary values
\begin{align}\label{eq:var_height_bc}
    u|_{\partial B_i} &= f^i_1 + \gamma^i, &
    \partial_\nu u|_{\partial B_i} &= f^i_2
\end{align}
where $\nu$ is the unit outward normal to $\partial \Omega_B$ and the parameter
$\gamma^i \in \R$ is allowed to vary freely in order to factor out the average
height on $\partial B_i$.
This is necessary because we only want to prescribe the height profile,
while the absolute or average height is not fixed.
For the following we collect all such boundary data in a function
$f  = (f_1,f_2) = \sum_{i=1}^{k} f^i\colon \Gamma \to \R^2$ where
$f^i = (f^i_1,f^i_2)\colon \partial B_i \to \R^2$ is extended by zero to the whole of $\Gamma$.
Using this notation we consider the
minimization problem:
\begin{problem}
    \label{prob:optimBase}
    Find $u\in H$ and $\gamma \in \R^k$ minimizing $\cJ_{\Omega_B}(u)$
    subject to 
    \begin{align}\label{eq:variational_particle_bc}
        u|_{\Gamma} = f_1 + \sum_{i=1}^k \gamma_{i} \eta_{i}, \qquad
        \partial_\nu u|_{\Gamma} = f_2\text{ .}
    \end{align}
\end{problem}
Notice that the effect of the free parameter $\gamma_i$ is localized to $\partial B_i$
via the use of the indicator function
$\eta_i := \chi_{\partial B_i} \in L^2(\Gamma)$
which is one on $\partial B_i$ but vanishes on all $\partial B_j, j \neq i$.
In view of the trace theorem we from now on assume that
$f = (f_1,f_2) \in H^{\frac{3}{2}}(\Gamma) \times H^{\frac{1}{2}}(\Gamma)$.
Under this assumption it is known that \pref{prob:optimBase} admits a unique solution
by application of Lax--Milgram's theorem \cite{ElGrHoKoWo15}.

It is in fact possible to simplify this problem formulation.
For this purpose we make use of the trace operator $T_{\Gamma} = (T_{\Gamma}^1, T_{\Gamma}^2)$
\begin{alignat*}{3}
    T_{\Gamma}\colon H^2(\Omega) & \rightarrow H^{\frac{3}{2}}(\Gamma) \times H^{\frac{1}{2}}(\Gamma),
    & \qquad T_{\Gamma}(v) = (v|_{\Gamma}, \ \partial_\nu v|_{\Gamma} ).
\end{alignat*}
This operator is well-defined, continuous, surjective and admits a
continuous right-inverse (see e.\,g. \cite[Theorem 1.5.1.2]{Grisvard85}).
Note that we can also view $T_{\Gamma}$ as a trace operator for $H^2(\Omega_B)$ due
to our regularity assumptions on $\Gamma$.
In order to simplify the notation for boundary values up to the average height
in \eqref{eq:variational_particle_bc} we furthermore introduce the projection operator
\begin{align*}
    P_{\Gamma} \colon H^{\frac{3}{2}}(\Gamma) \times H^{\frac{1}{2}}(\Gamma) &\longrightarrow H^{\frac{3}{2}}(\Gamma) \times H^{\frac{1}{2}}(\Gamma)
\end{align*} 
defined by $P_\Gamma(v_1,v_2) = (P_\Gamma^1(v_1),P_\Gamma^2(v_2))$ where
\begin{align*}
    P_\Gamma^1(v)
        = v - \sum_{i=1}^k |\partial B_i|^{-1}\int_{\partial B_i} v \,\mathrm{d}\sigma
        = v - \sum_{i=1}^k \frac{(v, \eta_i)_{L^2(\Gamma)}}{(\eta_i,\eta_i)_{L^2(\Gamma)}} \eta_i, \qquad
    P_\Gamma^2(v) = v.
\end{align*}
Using this operator we can formulate the boundary conditions \eqref{eq:variational_particle_bc} as
$P_\Gamma (T_\Gamma u - f ) = 0$ which gives rise to
the space of admissible functions defined by
\begin{align*}
    V_f := \left\{ v \in H_0^2(\Omega) \mid P_\Gamma (T_\Gamma v - f ) = 0 \right \}
\end{align*}
and the corresponding minimization problem:
\begin{problem}
    \label{prob:optimHardCurve}
    Find $u\in V_f$ minimizing $\cJ_{\Omega}(u)$.
\end{problem}
Notice that \pref{prob:optimHardCurve} differs from \pref{prob:optimBase}
not only due to the different notation for the boundary conditions
on $\Gamma$, but also because it minimizes $J_\Omega$ for functions
defined on the whole of $\Omega$ and not just on $\Omega_B$.
However, the following result from \cite{ElGrHoKoWo15} shows that a solution of
\pref{prob:optimHardCurve} also immediately yields a solution to
\pref{prob:optimBase}.
\begin{proposition}
    Let $u \in H_0^2(\Omega)$ be the solution of \pref{prob:optimHardCurve}.
    Define
    $\gamma_i = \vert \partial B_i\vert^{-1} (u-f_1, \eta_i )_{L^2(\Gamma)}$ for $i = 1,\dots,k$.
    Then $(u|_{\Omega_B},\gamma)$ is the solution of \pref{prob:optimBase}.
\end{proposition}

While Problem~\ref{prob:optimHardCurve} allows for variable height of particles,
the full problem considered in \cite{ElGrHoKoWo15} has additional degrees of freedom.
This is due to the fact that particles can move and rotate in the plane of the fluid
membrane.
To avoid mesh-deformation or re-meshing whenever particle positions change we
will in the following drop the hard
constraints at the particle boundaries in favor of a penalized approach.
Replacing the hard curve constraints by penalty terms in the energy functional
leads to:
\begin{problem}
    \label{prob:optimSoftCurve}
    Find $u_\eps \in H^2_0(\Omega)$ minimizing
    \begin{align*}
        \cJ_{\Omega}(u_\eps ) + \sum_{i=1}^2 \frac{1}{\eps_i} \Vert P_\Gamma^i ( T_{\Gamma}^i u_\eps  - f_i ) \Vert_{L^2(\Gamma)}^2\text{.}
    \end{align*}
\end{problem}
This formulation is more favorable than \pref{prob:optimHardCurve} in so far as
it admits a straightforward conforming finite element discretization without re-meshing
in case of variable particle positions.

At this point we want to mention an alternative penalty formulation which is
based on the idea that, for known shapes $B_i$, the solution $u|_B$ could
be computed a-priori up to a constant per component $B_i$.
We define the restriction operator
\begin{alignat*}{3}
    T_B\colon H^2(\Omega) & \longrightarrow H^{2}(B), & \qquad & T_B(v)= v|_B.
\end{alignat*}
Analogously to the curve constraint formulation we introduce the
associated projection operator
\begin{align*}
    P_B\colon H^2(B) & \longrightarrow H^2(B), &
    P_B(v) &= v - \sum_{i=1}^k \frac{(v,\psi_i)_{H^s(B)}}{(\psi_i,\psi_i)_{H^s(B)}} \psi_i,
\end{align*}
for $\psi_i := \chi_{B_i}$ and some fixed $s \in [0,2]$.
Using this notation the bulk constrained problem reads:
\begin{problem}
    \label{prob:optimHardArea}
    Find $u\in H_0^2(\Omega)$ minimizing $\cJ_{\Omega}(u)$ subject to
    \begin{align*}
        P_B T_B u = P_B u|_B \text{.}
    \end{align*}
\end{problem}
One quickly verifies that the solutions of \pref{prob:optimHardCurve} and \pref{prob:optimHardArea} coincide.
Analogously to Problem~\ref{prob:optimSoftCurve} a penalty formulation of Problem~\ref{prob:optimHardArea}
is given by:
\begin{problem}
    \label{prob:optimSoftArea}
    Find $u_\eps \in H_0^2(\Omega)$ and minimizing
    \begin{align*}
        \cJ_{\Omega}(u_\eps) + \frac{1}{\eps} \Vert P_B(T_Bu_\eps - u|_B) \Vert_{H^s(B)}^2\text{.}
    \end{align*}
\end{problem}

While the penalized formulations are more flexible in the sense that
the domain $\Omega_B$ does not have to be resolved for discretization,
one will be faced with the problem of balancing the penalty parameter
and discretization errors. To this end we will first develop an abstract
error estimate for penalized problems
and then
analyze the regularity of the solution of the hard constrained \pref{prob:optimHardCurve}.
Using a suitable regularity result later allows to derive optimal $h$-dependent
values of $\eps_i$ for a discretization with mesh size $h$.



\section{An Error Estimate for Linear Penalty Problems on Hilbert Spaces}\label{sec:abstractError}

In this section we derive an abstract energy error estimate for penalized discretizations
of linearly constrained problems in a Hilbert space setting.
Let $H$ be a Hilbert space and $U_0 \subset H$ a closed subspace.
We consider the affine closed subspace $U = U_0 + u_0 \subseteq H$ of $H$
for some given $u_0 \in H$.
In addition, let $a\colon H \times H \rightarrow \mathbb{R}$ be a
symmetric positive semi-definite bounded bilinear form and
$\ell \colon H \rightarrow \mathbb{R}$ a bounded linear form on $H$.
We furthermore assume that $a(\cdot,\cdot)$ is coercive on $U_0$.
In this setting we consider the affine constrained minimization
problem:

\begin{problem}\label{prob:abstractConstrainedMinimization}
    Find $u \in U$ minimizing
    \begin{align*}
        \frac{1}{2} a(u,u) - \ell(u)\text{.}
    \end{align*}
\end{problem}

Application of Lax--Milgram's theorem yields the existence of a unique solution $u \in U$
of \pref{prob:abstractConstrainedMinimization}
characterized by the variational equation
\begin{align}\label{eq:abstractConstrainedVariation}
    a(u, v) = \ell(v)
    \qquad \forall{v \in U_0}.
\end{align}
From now on $u \in U$ denotes this unique solution.


Now let us assume that we want to approximate this problem by a penalty
formulation over some closed linear subspace $X \subseteq H$ with
$m\in\mathbb{N}$ penalty terms.
Those penalty terms shall be given by symmetric positive semi-definite
bounded bilinear forms $b_i\colon H
\times H \rightarrow \mathbb{R}$ and penalty parameters $\eps_i >0$.
Denoting by $\Vert v \Vert_{c} := \sqrt{ c(v,v) }$ the
semi-norm induced by a symmetric positive semi-definite bilinear form
$c(\cdot,\cdot)$ on $H$ the penalized problem reads:

\begin{problem}\label{prob:abstractPenaltyMinimization}
    Find $u^X_\varepsilon \in X$ minimizing
    \begin{align*}
        \frac{1}{2} a(u^X_\varepsilon,u^X_\varepsilon) - \ell(u^X_\varepsilon)
        + \sum_{i=1}^m \frac{1}{2\eps_i} \Vert u^X_\varepsilon-u\Vert_{b_i}^2.
    \end{align*}
\end{problem}

For the sake of convenience we write
\begin{align}
    \label{eq:penalizedForms}
    a_\eps(\cdot,\cdot) = a(\cdot,\cdot) + \sum_{i=1}^m \frac{1}{\eps_i} b_i(\cdot,\cdot),
    \qquad \ell_\eps(\cdot) = \ell(\cdot) + \sum_{i=1}^m \frac{1}{\eps_i} b_i(u,\cdot)
\end{align}
for $\eps = (\eps_i)_{i=1,\dots,m} \in \mathbb{R}_{+}^m = (0,\infty)^m$.
We require that $a_1(\cdot,\cdot)$ is coercive on $X$,
such that $a_\eps(\cdot,\cdot)$ is also coercive for all $\eps \in (0,1]^m$.
Under these assumptions Lax--Milgram's theorem implies existence of a unique solution
$u^X_\varepsilon \in X$ for any $\eps \in (0,1]^m$, characterized by
the variational equation
\begin{align}\label{eq:abstractPenaltyVariation}
    a_\eps(u^X_\varepsilon, v) = \ell_\eps(v)
    \qquad \forall v \in X.
\end{align}

The following result states a C\'ea-type estimate for
the error $u - u^X_\varepsilon$ resulting from penalization and
discretization in $X$.

\begin{theorem}\label{thm:AbstractError}
    For the fixed $u \in U$ used in the definition
    of $\ell_\eps$ 
    suppose there exist constants $c_i >0$ for $i=1,\dots,m$
    such that
    \begin{align}\label{eq:captureCondition}
        \vert a(u, v) - \ell(v) \vert \leq \sum_{i=1}^m c_i \Vert v\Vert_{b_i}
        \qquad \forall{v\in X}.
    \end{align}
    Then the error
    for the solution
    $u^X_\eps$ of \pref{prob:abstractPenaltyMinimization} can be bounded by
    \begin{align}\label{eq:abstractErrorEstimate}
        \Vert u- u^X_\eps \Vert_{a_\eps}^2 
            \leq 3 \inf_{v\in X} \left( \Vert u-v \Vert_{a_\eps}^2
                    + \sum_{i=1}^m \varepsilon_i c_i^2 \right)\text{.}
    \end{align}
\end{theorem}

\begin{proof}
    Let $e = u - u^X_\varepsilon$. Then we get
    \begin{align}\label{eq:helper4}
        \Vert e \Vert_{a_\eps}^2 = a_\eps(e,e) = a_\eps(e,u-v) + a_\eps(e,v-u^X_\varepsilon)
    \end{align}
    for all $v\in X$.
    Using Young's inequality we can bound the first term by
    \begin{align}\label{eq:helper2}
        \begin{aligned}
            a_\eps(e,u-v)
                &\leq \Vert e\Vert_{a_\eps} \, \Vert u - v\Vert_{a_\eps}
            \leq \frac{1}{4} \Vert e\Vert_{a_\eps}^2 + \Vert u - v\Vert_{a_\eps}^2.
        \end{aligned}
    \end{align}
    The definition \eqref{eq:penalizedForms} of
    the $u$-dependent penalized functional $\ell_\eps$
    yields
    \begin{align*}
        a_\eps(u,w) - a(u,w) & = \ell_\eps(w) - \ell(w)
        \qquad \forall{w \in X}.
    \end{align*}
    Combining this identity with
    \eqref{eq:abstractPenaltyVariation} we get
    \begin{align}\label{eq:helper1}
        a_\eps(e,w) = a(u,w) - \ell(w)
        \qquad \forall{w \in X}.
    \end{align}
    Because of $v - u^X_\varepsilon \in X$ this implies
    \begin{align}\label{eq:helper3}
        \begin{aligned}
            a_\eps(e,v-u^X_\varepsilon) & = a(u,v-u^X_\varepsilon) - \ell(v-u^X_\varepsilon)
            \\ & \leq \sum_{i=1}^m c_i \Vert v - u^X_\varepsilon\Vert_{b_i}
            \\ & \leq \sum_{i=1}^m c_i \left( \Vert e\Vert_{b_i} + \Vert u - v \Vert_{b_i} \right)
            \\ & \leq \sum_{i=1}^m \Bigl( \frac{1}{4\varepsilon_i} \Vert e\Vert_{b_i}^2 + \frac{1}{2\varepsilon_i} \Vert u - v\Vert_{b_i}^2 + \frac{3}{2}\varepsilon_i c_i^2\Bigr)
            \\ & \leq \frac{1}{4}\Vert e\Vert_{a_\eps}^2 + \frac{1}{2}\Vert u-v\Vert_{a_\eps}^2 + \sum_{i=1}^m \frac{3}{2}\varepsilon_i c_i^2
        \end{aligned}
    \end{align}
    where the second to last inequality follows from Young's inequality.

    Inserting the estimates \eqref{eq:helper2} and \eqref{eq:helper3} into \eqref{eq:helper4}
    we obtain
    \begin{align*}
        \Vert e\Vert_{a_\eps}^2
        \leq \frac{1}{2}\Vert e\Vert_{a_\eps}^2 + \frac{3}{2}\Vert u-v\Vert_{a_\eps}^2 + \sum_{i=1}^m \frac{3}{2}\varepsilon_i c_i^2
    \end{align*}
    which proves the assertion.
\end{proof}



Next we consider the case where the evaluation of
$a(\cdot,\cdot)$, $\ell(\cdot)$, and $b_i(\cdot,\cdot)$
is not performed exactly but approximated by some $\wt{a}$,
$\wt{\ell}$ and $\wt{b}_i$, respectively.
We use a notation analogous to $a_\eps$ and $\ell_\eps$:
\begin{align*}
    \wt{a}_\eps(\cdot,\cdot) = \wt{a}(\cdot,\cdot) + \sum_{i=1}^m \frac{1}{\eps_i} \wt{b}_i(\cdot,\cdot)
    , \qquad \wt{\ell}_\eps(\cdot) = \wt{\ell}(\cdot) + \sum_{i=1}^m \frac{1}{\eps_i} \wt{b}_i(u,\cdot)
\end{align*}
Instead of solving \eqref{eq:abstractPenaltyVariation} for $u^X_\eps$ directly one
now computes a $\wt{u}^X_\eps$ by solving
\begin{align}\label{eq:abstractInexactPenaltyVariation}
    \wt{a}_\eps(\wt{u}^X_\eps, v ) = \wt{\ell}_\eps(v)
    \qquad \forall{v \in X} \text{.}
\end{align}
In this setting we can prove the following Strang-type result:

\begin{proposition}\label{prop:strang}
    Let $\norm{a}$ be the continuity constant of $a$ with respect to the $H$-norm and let
    $a_1 = a+\sum_{i=1}^m b_i$ and $\wt{a}_1 = \wt{a} +\sum_{i=1}^m \wt{b}_i$ be coercive
    with respect to the constants $\alpha$ and $\wt{\alpha}$, respectively.
    Additionally to the assumptions of Theorem~\ref{thm:AbstractError}
    suppose that for all $i\in\{1,\dots,m\}$ there exists $c_i \in \mathbb{R}$ such that for all $v \in X$
    \begin{align}\label{eq:helper1611061548}
        \abs{ b_i(v,v) - \wt{b}_i(v,v) } \leq c_i \Vert v \Vert^2\text{.}
    \end{align}
    Then there exists a constant $C = C(\alpha,\wt{\alpha},\norm{a}) >0$ such that
    \begin{align*}
        \norm{u - \wt{u}^X_\eps}_{a_\eps}
        & \leq C \left( 1 + \sum_{i=1}^m \frac{c_i}{\eps_i} \right) \inf_{v\in X} \sup_{w \in X} \left[  \norm{u - v}_{a_\eps} + \frac{1}{\norm{w}} \abs{(a-\wt{a})(v,w)}   \phantom{\left(\frac{1}{\norm{w}}{\sum_{i=1}^m\abs{(b_i-\wt{b}_i)}}\right)}\right.
        \\ & \left. \qquad \qquad + \frac{1}{\norm{w}}\sum_{i=1}^m \frac{1}{\eps_i}\abs{(b_i-\wt{b}_i)(u-v,w)} + \frac{1}{\norm{w}}\abs{(\ell-\wt{\ell})(w)} \right ].
    \end{align*}
\end{proposition}

\begin{proof}
    See appendix
\end{proof}

\newcommand{\remove}[1]{}
\remove{
\snote{
In the special case that exact evaluation of the penalty terms is not possible and one instead approximates the solution of \eqref{eq:abstractPenaltyVariation} by solving
\begin{align}\label{eq:abstractInexactPenaltyVariation}
    \forall{v \in X\colon} \ a( \wt{u}^X_\eps, v ) + \sum_{i=1}^m \frac{1}{\eps_i} \wt{b}_i( \wt{u}^X_\eps - u^\ast, v ) = \ell(v)\text{.}
\end{align}
For this situation we can prove a simple error estimate.

\begin{proposition}
    Let $\wt{u}^X_\eps$ be the solution of \eqref{eq:abstractInexactPenaltyVariation}.
    If there exists $c \in \mathbb{R}_{>0}$ such that for all $i \in \{1,\dots,m\}$ and $v \in X$
    \begin{align*}
        (\wt{b}_i-b)(u^X_{\eps} - u^\ast, v ) \leq c \Vert u^X_\eps - u^\ast\Vert_{b_i} \, \Vert v \Vert_{\wt{b}_i}
    \end{align*}
    holds, then
    \begin{align*}
        \Vert u^X_{\eps} - \wt{u}^X_{\eps}\Vert_{a} \leq c \Vert u^\ast - u^X_\eps \Vert_{a'}
    \end{align*}
    where $a' = a + \sum_{i=1}^m \frac{1}{\eps_i}b_i$.
\end{proposition}

\begin{proof}
    Let $\wt{a} = a + \sum_{i=1}^m \frac{1}{\eps_i} \wt{b}_i$.
    For brevity let $u := u^X_\eps$ and $\wt{u} := \wt{u}^X_\eps$.
    Using the variational equalities \eqref{eq:abstractPenaltyVariation}, \eqref{eq:abstractInexactPenaltyVariation} and the assumptions on the $\wt{b}_i$ we conclude
    \begin{align*}
        \wt{a}\left( u - \wt{u}, u - \wt{u} \right)
        & = a( u, u - \wt{u} ) + \sum_{i=1}^m \frac{1}{\eps_i} b_i( u - u^\ast, u - \wt{u} )
        \\ & \qquad - a( \wt{u}, u - \wt{u} ) - \sum_{i=1}^m \frac{1}{\eps_i} \wt{b}_i ( \wt{u} - u^\ast, u - \wt{u} )
            \\ & \qquad + \sum_{i=1}^m \frac{1}{\eps_i} (\wt{b}_i - b )( u - u^\ast, u - \wt{u} )
        \\ & = \sum_{i=1}^m (\wt{b}_i - b )( u - u^\ast, u - \wt{u} )
        \\ & \leq c \sum_{i=1}^m \Vert u - u^\ast\Vert_{b_i} \, \Vert u - \wt{u} \Vert_{\wt{b}_i}
        \\ & \leq \sum_{i=1}^m \frac{c^2}{4\eps_i} \Vert u - u^\ast\Vert_{b_i}^2
                + \sum_{i=1}^m \frac{1}{\eps_i} \Vert u - \wt{u} \Vert_{\wt{b}_i}\text{.}
    \end{align*}
    Substracting the second summand from the inequality immediately yields
    \begin{align*}
        \Vert u^X_{\eps} - \wt{u}^X_{\eps} \Vert_a \leq \frac{c}{2} \sum_{i=1}^m \frac{1}{\sqrt{\eps_i}} \Vert u^\ast - u^X_\eps \Vert_{b_i}
        \leq c \Vert u^\ast - u^X_\eps \Vert_{a'}\text{.}
    \end{align*}
\end{proof}
}
}

\section{Regularity of the Hard Curve Constraint Problem}\label{sec:regularity}
In order to apply the results of the previous section it will be crucial to
prove \eqref{eq:captureCondition}. For this purpose and in order to prove
convergence rates by bounding the best approximation errors in Theorem~\ref{thm:AbstractError}
we will now derive regularity results for the solution of \pref{prob:optimHardCurve}.

Let $u \in V_f$ be the solution of the hard curve minimization
formulation, \pref{prob:optimHardCurve}.
Knowing the regularity of $u$ is central to proving discretization errors
since the maximal regularity of the solution immediately reveals the optimal
rates of convergence that one would expect for the corresponding discretization
errors.

Our strategy is to rewrite the hard curve minimization problem as a system of
elliptic partial differential equations to which we apply standard regularity
theory for elliptic interface problems.
To this end we define the fourth order elliptic differential operator
$L := \kappa \Delta^2 - \sigma \Delta$ associated with the energy functional $\cJ$.

\begin{proposition}\label{prop:weakHardCurve}
    Suppose $(u_1,u_2) \in H^2(\Omega_B) \times H^2(B)$ is a weak solution of the system of PDEs
    \begin{align}\label{eq:pdeHardCurve}
        \begin{aligned}
            Lu_1 & = 0 \text{ on $\Omega_B$,}
            & T_\Gamma u_1 & = T_\Gamma u \text{ on $\Gamma$,}
            & u_1 = \partial_\nu u_1 = 0 \text{ on $\partial\Omega$,}
            \\ Lu_2 & = 0 \text{ on $B$,}
            & T_\Gamma u_2 & = T_\Gamma u \text{ on $\Gamma$.}
        \end{aligned}
    \end{align}
    Then $u_1 = u|_{\Omega_B}$ and $u_2 = u|_B$ where $u$ is the solution of \pref{prob:optimHardCurve}.
    Conversely, if $u$ is the solution of \pref{prob:optimHardCurve},
    then $(u|_{\Omega_B}, u|_B)$ solves \eqref{eq:pdeHardCurve}.
\end{proposition}

\begin{proof}
    A proof for the equivalence of \pref{prob:optimHardCurve} with a weak formulation of
    \eqref{eq:pdeHardCurve} is given in \cite[Proposition 4.1]{ElGrHoKoWo15}.
\end{proof}


\begin{lemma}
    \label{lem:piecewiseRegularity}
    Let $\Omega$ be a piecewise polygonal domain, let $\Gamma$ be smooth and suppose $(f_1,f_2) \in H^{\frac{7}{2}}(\Gamma)\times H^{\frac{5}{2}}(\Gamma)$.
    If all corners of $\Omega$ have an inner angle $\omega$ with $\omega \leq 126^\circ$,
    then $(u|_{\Omega_B},u|_B) \in H^{4}(\Omega_B)\times H^{4}(B)$.
\end{lemma}

\begin{proof}
    Let $g := \frac{\sigma}{\kappa} \Delta u|_{\Omega_B} \in L^2(\Omega_B)$.
    From Proposition~\ref{prop:weakHardCurve} we know that $u|_{\Omega_B}$ is a weak solution of
    \begin{align*}
        \Delta^2 u = g \text{ on $\Omega_B$},
        \qquad T_{\Gamma}u = T_{\Gamma} u|_{\Omega_B} \text{ on $\Gamma$},
        \qquad u = \partial_\nu u = 0 \text{ on $\partial\Omega$.}
    \end{align*}
    Since $g \in L^2(\Omega_B)$ and
    $T_{\Gamma}^1 u|_{\Omega_B} = f_1 + \sum_{i=1}^k \vert \partial B_i\vert^{-1} ( u|_{\Omega_B} - f_1, \eta_i)_{L^2(\Gamma)}\eta_i \in H^{\frac{7}{2}}(\Gamma)$
    as well as $T_{\Gamma}^2 u|_{\Omega_B} = f_2 \in H^{\frac{5}{2}}(\Gamma)$
    it follows from \cite[Theorem 7.2.2.3]{Grisvard85} and the computations in
    \cite{BlumRannacher80} on the associated characteristic equation that
    $u|_{\Omega_B} \in H^4(\Omega_B)$.
    The analogue argumentation on $B$ yields $u|_B \in H^4(B)$ and thus proves the assertion.
\end{proof}


Having the intrinsic regularity $u \in H^2(\Omega)$ and the piecewise regularity
from Lemma~\ref{lem:piecewiseRegularity} allows to show an improved global regularity result.
The key ingredient is the following technical Lemma.

\begin{lemma}\label{lem:gluing}
    Let $v \in H^2(\Omega)$ such that $v|_{\Omega_B} \in H^4(\Omega_B)$ and $v|_{B} \in H^4(B)$.
    Then $v \in H^{2+\frac{1}{2}-\delta}(\Omega)$ for all $\delta > 0$.
\end{lemma}

\begin{proof}
    See appendix.
\end{proof}

Combining Lemma~\ref{lem:piecewiseRegularity} and Lemma~\ref{lem:gluing} now gives:

\begin{corollary}\label{cor:regularity}
    Let the assumptions from Lemma~\ref{lem:piecewiseRegularity} hold.
    Then $u \in H^{\frac{5}{2}-\delta}(\Omega)$ for every $\delta >0$.
\end{corollary}


\section{Discretization Errors}\label{sec:discreteError}
In this section we apply the results from the previous sections to derive
discretization errors for penalized finite element approximations of the form
\pref{prob:optimSoftCurve} and \pref{prob:optimSoftArea} but with $h$-dependent
penalty parameters.

We suppose from now on that $\Omega$ is a rectangular domain.
In view of Lemma~\ref{lem:piecewiseRegularity} this implies that
$u|_{\Omega_B} \in H^4(\Omega_B)$, $u|_B \in H^4(B)$ and
$u \in H^{\frac{5}{2}-\delta}(\Omega)$ for any $\delta >0$.
On the domain $\Omega$ we establish a quadrilateral grid equipped with
Bogner--Fox--Schmit finite elements (see e.\,g. \cite{Ciarlet78}).
Given the set of grid nodes $N$ and the set of multi-indices
$\Lambda = \{(0,0), (1,0), (0,1), (1,1)\}$,
the resulting finite element space is spanned by a basis
$(\psi_{p,\alpha})_{p\in N, \alpha \in \Lambda}$ of piecewise bi-cubic
polynomials such that each $\psi_{p,\alpha}$ satisfies
\begin{align*}
    \partial^\beta
    \psi_{p,\alpha}(q)
        = \delta_{\alpha,\beta} \delta_{p,q}
    \qquad \forall q \in N, \beta \in \Lambda
\end{align*}
where $\partial^\beta$ is the partial derivative associated with the multi-index $\beta$
and $\delta_{a,b}$ is the Kronecker delta. I.\,e., the degrees of freedom are the values,
first order partial derivatives and mixed second order partial derivatives at the vertices
(see Figure~\ref{fig:bfsdofs}).

\begin{figure}
    \begin{center}
        \begin{tikzpicture}[
            node distance=10mm,
            mixed/.style={
                preaction={decorate,decoration={markings, mark=at position 0 with {\arrowreversed[scale=2]{stealth}}}}
            }]
            \coordinate (p1) at (0,0);
            \coordinate (p2) at (2,0);
            \coordinate (p3) at (2,2);
            \coordinate (p4) at (0,2);
            \draw (p1) node[below left] {}
               -- (p2) node[below right] {}
               -- (p3) node[above right] {}
               -- (p4) node[above left] {}
               --cycle;
            \draw (p1) circle (.2);
            \fill (p1) circle (.1);
            \draw (p1)+(0.5,0.5) edge[mixed] (p1);
            \draw (p2) circle (.2);
            \fill (p2) circle (.1);
            \draw (p2)+(0.5,0.5) edge[mixed] (p2);
            \draw (p3) circle (.2);
            \fill (p3) circle (.1);
            \draw (p3)+(0.5,0.5) edge[mixed] (p3);
            \draw (p4) circle (.2);
            \fill (p4) circle (.1);
            \draw (p4)+(0.5,0.5) edge[mixed] (p4);

            \fill (4,2) circle (.1);
            \draw (4.5,2) node[right] {value};
            \draw (4,1) circle (.2);
            \draw (4.5,1) node[right] {1st order partial derivatives};
            \draw (3.8,-0.2)+(0.5,0.5) edge[mixed] (3.8,-0.2);
            \draw (4.5,0) node[right] {mixed 2nd order partial derivatives};
        \end{tikzpicture}
    \end{center}
    \caption{Degrees of freedom for the Bogner--Fox--Schmit element}
    \label{fig:bfsdofs}
\end{figure}
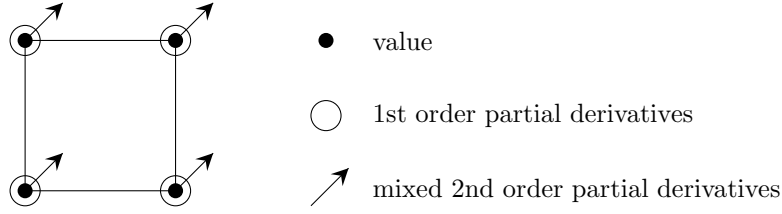

By setting the degrees of freedom on $\partial\Omega$ to zero this yields a
conforming subspace of $C^1(\overline{\Omega}) \cap H_0^2(\Omega)$.
Given a family $(S_h)_{h\in I}$ of such finite element discretizations over
$\Omega$ based on quasi-uniform grids with mesh size $h$
we will use the following well known result:
There exists a constant $c>0$ such that for all $K \in [0,4]$ the approximation estimate
\begin{align}\label{eq:interpolationEstimate}
    \forall{v \in H^K(\Omega) \cap H_0^2(\Omega)}
    \ \exists{\overline{v} \in S_h}
    \ \forall{k\in [0,\min\{2,K\}]}
    \colon \ 
    \Vert v - \overline{v} \Vert_{H^k(\Omega)} \leq c h^{K-k} \Vert v \Vert_{H^K(\Omega)}
\end{align}
holds true for all $h \in I$.
(This is a direct consequence of classical approximation estimates, see e.\,g.
\cite[Theorem 3.1.4]{Ciarlet78}, and interpolation theory of function spaces,
see e.\,g. \cite[Theorem 1.1.6]{Lunardi99}.)

In order to put our minimization problems into the variational framework used
in \cref{sec:abstractError} we introduce the bilinear form
\begin{align*}
    a \colon H^2(\Omega) \times H^2(\Omega) &\longrightarrow \mathbb{R}, &
    a(w,v) &= \kappa ( \Delta w, \Delta v )_{L^2(\Omega)} + \sigma ( \nabla w, \nabla v )_{L^2(\Omega)}.
\end{align*}
together with the curve penalty terms
\begin{align*}
    b_{\Gamma}^1 \colon H^2(\Omega) \times H^2(\Omega) & \longrightarrow \mathbb{R}, &
        b_{\Gamma}^1(w,v) &= ( P_\Gamma^1T_\Gamma^1 w, P_\Gamma^1T_\Gamma^1 v )_{L^2(\Gamma)},\\
    b_{\Gamma}^2 \colon H^2(\Omega) \times H^2(\Omega) & \longrightarrow \mathbb{R}, &
        b_{\Gamma}^2(w,v) &= ( P_\Gamma^2T_\Gamma^2 w, P_\Gamma^2T_\Gamma^2 v )_{L^2(\Gamma)}
\end{align*}
and the bulk penalty term
\begin{align*}
    b_{B}\colon H^2(\Omega) \times H^2(\Omega) & \longrightarrow \mathbb{R}, &
        b_{B}(w,v) &= ( P_BT_B w, P_BT_B v )_{H^s(B)}\text{.}
\end{align*}
Then the
solution $u \in V_f \subset H_0^2(\Omega)$ of \pref{prob:optimHardCurve} satisfies
the variational equation
\begin{align}\label{eq:varBase}
    a(u,v) = 0 \qquad
    \forall{v \in V_0}.
\end{align}

Discretization of the penalized curve constrained \pref{prob:optimSoftCurve}
in $S_h$ leads to the \emph{soft curve problem}:
    Find $u_\eps^h \in S_h$ minimizing
    \begin{align}\label{eq:optimSoftCurveDiscrete}
        \cJ_{\Omega}(u_\eps^h) + \sum_{i=1}^2 \frac{1}{\eps_i} \Vert P_\Gamma^i (T_{\Gamma}^i u_\eps^h - f_i) \Vert_{L^2(\Gamma)}^2\text{.}
    \end{align}
which is equivalently characterized by the variational equation
\begin{align}\label{eq:varDiscreteSoftCurve}
    u_\eps^h \in S_h: \qquad
    a(u^h_\eps,v) + \sum_{i=1}^2 \frac{1}{\eps_i} b_\Gamma^i(u^h_\eps-u,v)  = 0
    \qquad \forall{v \in S_h}.
\end{align}

Analogously, discretization of the penalized bulk constrained \pref{prob:optimSoftArea}
in $S_h$ leads to the \emph{soft bulk problem}:
    Find $u_\eps^h \in S_h$ minimizing
    \begin{align}\label{eq:optimSoftAreaDiscrete}
        \cJ_{\Omega}(u_\eps^h) + \frac{1}{\eps} \Vert P_BT_B(u_\eps^h - u) \Vert_{H^s(B)}^2
    \end{align}
with the equivalent variational equation
\begin{align}\label{eq:varDiscreteSoftArea}
    u_\eps^h \in S_h: \qquad
    a(u^h_\eps,v) + \frac{1}{\eps} b_B(u^h_\eps-u,v)  = 0
    \qquad \forall{v \in S_h}.
\end{align}

At this point we want to emphasize that the solutions of
\eqref{eq:varDiscreteSoftCurve} and \eqref{eq:varDiscreteSoftArea} are in
general different.
We nevertheless refer in a slight abuse of notation to both solutions as $u^h_\eps$.
In the following it will, however, be clear from the context whether $u^h_\eps$
denotes the solution of the soft curve problem or the soft bulk problem.

In the next result we show that the soft curve formulation meets the requirement from Theorem~\ref{thm:AbstractError}, which we will use afterwards to get an asymptotic error estimate.
\begin{lemma}\label{lem:captureConditionSoftCurve}
    For all $v \in H_0^2(\Omega)$ the solution $u$ of \pref{prob:optimHardCurve} satisfies
    \begin{align*}
        a(u,v) \leq \kappa \left(\Vert [\Delta u] \Vert_{L^2(\Gamma)} + \Vert [\partial_\nu \Delta u]\Vert_{L^2(\Gamma)} \right) \left( \Vert v\Vert_{b^1_\Gamma} + \Vert v \Vert_{b^2_\Gamma} \right)\text{.}
    \end{align*}
    Here $[w] = w|_{\Omega_B} - w|_B$ denotes the jump of the function $w$ on $\Gamma$.
\end{lemma}

\begin{proof}
    By integration by parts and as of $u \in H^2(\Omega)$, $u|_{\Omega_B} \in H^4(\Omega_B)$, $u|_B \in H^4(B)$ and $Lu = 0$ we are able to state for all $v \in H_0^2(\Omega)$
    \begin{align*}
        \frac{1}{\kappa} a(u,v)
        & = \int_{\Omega} \Delta u \Delta v + \frac{\sigma}{\kappa} \nabla u \cdot \nabla v
        \\ & = \int_{\Omega_B} \frac{1}{\kappa} Lu|_{\Omega_B} \, v
                + \int_B \frac{1}{\kappa} Lu|_B \, v
                + \int_{\Gamma} \frac{\sigma}{\kappa} \left( \partial_\nu u|_{\Omega_B} - \partial_\nu u|_B \right) v
        \\ & \qquad     + \int_{\Gamma} \left( - \partial_\nu \Delta u|_{\Omega_B} + \partial_\nu \Delta u|_B \right) v
                + \int_{\Gamma} \left( \Delta u|_{\Omega_B} - \Delta u|_B \right ) \partial_\nu v
        \\ & = ( -[\partial_\nu \Delta u], T_\Gamma^1 v)_{L^2(\Gamma)} + ( [\Delta u], T_\Gamma^2 v)_{L^2(\Gamma)}
        \\ & = ( -[\partial_\nu \Delta u], P_\Gamma^1 T_\Gamma^1 v)_{L^2(\Gamma)} + ( [\Delta u], P_\Gamma^2 T_\Gamma^2 v)_{L^2(\Gamma)}
        \\ & \qquad + ( -[\partial_\nu \Delta u], (\operatorname{id}_{L^2(\Gamma)}-P_\Gamma^1)T_\Gamma^1 v)_{L^2(\Gamma)} + ( [\Delta u], (\operatorname{id}_{L^2(\Gamma)}-P_\Gamma^2)T_\Gamma^2 v)_{L^2(\Gamma)}\text{.}
    \end{align*}
    Since $\widetilde{v}_i := (\operatorname{id}_{L^2(\Gamma)} - P_{\Gamma}^i) T_\Gamma^i v \in H^{\frac{5}{2}-i}(\Gamma)$ and because $\Gamma$ is sufficiently smooth there exists a function $w \in H_0^2(\Omega)$ such that $T_\Gamma^i w = \widetilde{v}_i$, \cite[Theorem 9.4]{LioMag72}.
    In particular, we have $P_\Gamma^i T_\Gamma^i w = 0$ which yields $w \in V_0$ and
    \begin{align*}
        0 & = \frac{1}{\kappa} a(u,w) = ( -[\partial_\nu \Delta u], T_\Gamma^1 w)_{L^2(\Gamma)} + ( [\Delta u], T_\Gamma^2 w)_{L^2(\Gamma)}
        \\ & = ( -[\partial_\nu \Delta u], (\operatorname{id}_{L^2(\Gamma)}-P_\Gamma^1)T_\Gamma^1 v)_{L^2(\Gamma)} + ( [\Delta u], (\operatorname{id}_{L^2(\Gamma)}-P_\Gamma^2)T_\Gamma^2 v)_{L^2(\Gamma)}
    \end{align*}
    by applying the variational equation \eqref{eq:varBase} for $u$.
    Combined with the Cauchy--Schwarz inequality this implies
    \begin{align*}
        a(u,v) \leq \kappa \left(\Vert [\Delta u] \Vert_{L^2(\Gamma)} + \Vert [\partial_\nu \Delta u]\Vert_{L^2(\Gamma)} \right) \left( \Vert P_\Gamma^1 T_\Gamma^1 v\Vert_{L^2(\Gamma)} + \Vert P_\Gamma^2 T_\Gamma^2 v\Vert_{L^2(\Gamma)} \right)\text{,}
    \end{align*}
    which was to be shown.
\end{proof}

Now we are in the situation to show the main results,
namely the discretization error estimate for the discretizations
\eqref{eq:optimSoftCurveDiscrete} and \eqref{eq:optimSoftAreaDiscrete}.

\begin{theorem}\label{thm:discretizationErrorSoftCurve}
    Let $u^h_\eps \in S_h$ be the solution of the discrete soft curve problem \eqref{eq:varDiscreteSoftCurve}
    and assume that $\eps_1 = c_1 h^{\lambda_1}$ and $\eps_2 = c_2 h^{\lambda_2}$.
    For any $\delta >0$ define
    \begin{align*}
        \gamma = \min_{i\in\{1,2\}}\left( \frac{1}{2}-\delta, 3-i-2\delta-\frac{\lambda_i}{2}, \frac{\lambda_i}{2} \right)\text{.}
    \end{align*}
    Then there exists a constant $c>0$ independent of $h$ such that
    \begin{align*}
        \Vert u - u^h_\eps \Vert_{H^2(\Omega)} \leq c h^{\gamma}\text{.}
    \end{align*}
    In particular, $\Vert u - u^h_\eps \Vert_{H^2(\Omega)} \in O(h^{\frac{1}{2}-\delta})$
    for $\lambda_1 \in [1-2\delta,3-2\delta]$ and $\lambda_2 = 1-2\delta$.
\end{theorem}

\begin{proof}
    In this proof we use the notation $A \lesssim B$ whenever $A \leq c B$ holds with a constant
    $c$ that is independent of $\varepsilon_1$, $\varepsilon_2$ and $h$.
    
    As of Lemma~\ref{lem:captureConditionSoftCurve} we can apply Theorem~\ref{thm:AbstractError}.
    Using this together with the coercivity of $a(\cdot,\cdot)$ on $H_0^2(\Omega)$ we conclude for all $v\in S_h$ 
    \begin{align*}
        \Vert u - u^h_\eps \Vert_{H^2(\Omega)}
        & \lesssim \Vert u - u^h_\eps \Vert_{a}
        \\ & \lesssim \Vert u - v \Vert_a + \sum_{i=1}^2 \frac{1}{\sqrt{\eps_i}} \Vert u - v \Vert_{b_i} + \sum_{i=1}^2 \sqrt{\eps_i}
    \end{align*}
    Note that $a(\cdot,\cdot)$ is continuous on $H^2(\Omega)$ and $b_i(\cdot,\cdot)$
    is continuous on $H^{i-\frac{1}{2}+\delta}(\Omega)$.
    Furthermore, $u \in H^{\frac{5}{2}-\delta}(\Omega)$ according to Corollary~\ref{cor:regularity}.
    Choosing the interpolant $v = \overline{u}$ and applying the
    interpolation estimate \eqref{eq:interpolationEstimate} yields
    \begin{align*}
        \Vert u - u^h_\eps \Vert_{H^2(\Omega)}
        & \lesssim  \Vert u - \overline{u} \Vert_{H^2(\Omega)} + \sum_{i=1}^2 \frac{1}{\sqrt{\eps_i}} \Vert u - \overline{u} \Vert_{H^{i-\frac{1}{2}+\delta}(\Omega)} + \sum_{i=1}^2 \sqrt{\eps_i}
        \\ & \lesssim h^{\frac{1}{2}-\delta} \Vert u \Vert_{H^{\frac{5}{2}-\delta}(\Omega)} + \sum_{i=1}^2 \frac{1}{\sqrt{\eps_i}} h^{3-i-2\delta} \Vert u \Vert_{H^{\frac{5}{2}-\delta}(\Omega)} + \sum_{i=1}^2 \sqrt{\eps_i}
        \\ & \lesssim h^{\frac{1}{2}-\delta} + \sum_{i=1}^2 h^{3-i-2\delta-\frac{\lambda_i}{2}} + \sum_{i=1}^2 h^{\frac{\lambda_i}{2}}
        \\ & \lesssim h^\gamma\text{.}
    \end{align*}
    This proves the assertion.
\end{proof}

We can proceed similarly to prove convergence rates for the soft bulk formulation.
\begin{lemma}\label{lem:captureConditionSoftArea}
    Suppose that there is a constant $c_0 >0$ such that for every $h \in I$, $K \in [0,2]$ and $k\in [0,K]$ the inverse estimate
    \begin{align}\label{eq:inverseEstimate}
        \forall{v \in S_h} \colon \ 
        \vert v \vert_{H^K(B)} \leq c_0 h^{k-K} \vert v \vert_{H^k(B)}
    \end{align}
    is fulfilled.
    Let $u$ be the solution of \pref{prob:optimHardCurve}
    and $\delta>0$. Then there exists a constant $c>0$ such that for all $h \in I$ and $v \in S_h$
    \begin{align*}
        a(u,v) \leq c h^{\min\left(0,s-\frac{3}{2}-\delta\right)} \Vert v \Vert_{b_B}\text{.}
    \end{align*}
\end{lemma}

\begin{proof}
    From the proof of Lemma~\ref{lem:captureConditionSoftCurve} we know for all $v \in S_h$
    \begin{align*}
        \frac{1}{\kappa} a(u,v)
        & = ( -[\partial_\nu \Delta u], v )_{L^2(\Gamma)}
            + ( [\Delta u],\partial_\nu v)_{L^2(\Gamma)}
        \\ & = ( -[\partial_\nu \Delta u], T_\Gamma^1 P_B T_Bv )_{L^2(\Gamma)}
            + ( [\Delta u], T_\Gamma^2 P_BT_Bv)_{L^2(\Gamma)}
        \\ & \quad + ( -[\partial_\nu \Delta u], T_\Gamma^1 (\operatorname{id}_B - P_B) T_B v )_{L^2(\Gamma)}
            + ( [\Delta u], T_\Gamma^2 (\operatorname{id}_B - P_B) T_Bv)_{L^2(\Gamma)}\text{.}
    \end{align*}
    And as in the proof of Lemma~\ref{lem:captureConditionSoftCurve} there exists a function $w \in V_0 \subseteq H_0^2(\Omega)$ such that $T_\Gamma^i w = T_\Gamma^i(\operatorname{id}_B - P_B) T_B v$ for $i\in\{1,2\}$.
    From this we conclude
    \begin{align*}
        0 = \frac{1}{\kappa} a(u,w) = ( -[\partial_\nu \Delta u], T_\Gamma^1 (\operatorname{id}_B - P_B) T_B v )_{L^2(\Gamma)}
            + ( [\Delta u], T_\Gamma^2 (\operatorname{id}_B - P_B) T_Bv)_{L^2(\Gamma)}
    \end{align*}
    as in the proof of Lemma~\ref{lem:captureConditionSoftCurve}.
    Furthermore note that the operators $T_\Gamma^i$ are continuous and linear over $H^{\frac{3}{2}+\delta}(B)$ and thus the corresponding operator norms $\vertiii{T_{\Gamma}^i}_\delta$ are bounded.
    It follows that
    \begin{align*}
        \frac{1}{\kappa} a(u,v)
        & = ( -[\partial_\nu \Delta u], T_\Gamma^1 P_B T_Bv )_{L^2(\Gamma)}
            + ( [\Delta u], T_\Gamma^2 P_BT_Bv)_{L^2(\Gamma)}
        \\ & \leq \Vert [\partial_\nu \Delta u] \Vert_{L^2(\Gamma)} \vertiii{T_\Gamma^1}_{\delta} \Vert P_B T_B v \Vert_{H^{\frac{3}{2}+\delta}(B)}
        \\ & \quad + \Vert [\Delta u] \Vert_{L^2(\Gamma)} \vertiii{T_\Gamma^2}_{\delta} \Vert P_B T_B v \Vert_{H^{\frac{3}{2}+\delta}(B)}\text{.}
    \end{align*}
    Together with \eqref{eq:inverseEstimate} this implies
    \begin{align*}
        a(u,v)
        & \leq c h^{\min(0,s-\frac{3}{2}-\delta)}\Vert P_BT_Bv \Vert_{H^s(B)}
    \end{align*}
    which proves the assertion.
\end{proof}

Note that \eqref{eq:inverseEstimate} is indeed an \emph{assumption}.
We can not apply the standard finite element inverse estimates here as the grid is in general not matched to the domain $B$.
A closer look to the proof of the standard inverse estimate reveals that the constant in \eqref{eq:inverseEstimate} would go to infinity as $\abs{E_h\cap B} / \abs{E_h} \to 0$ for $I \ni h \to 0$ and some element $E_h$ in the grid associated to $S_h$.
This means that the above inverse estimate over $B$ is only valid for grids that resolve $B$ sufficiently well.

\begin{theorem}\label{thm:discretizationErrorSoftArea}
    Let the assumptions from Lemma~\ref{lem:captureConditionSoftArea} hold,
    $u^h_\eps \in S_h$ the solution of the discrete soft bulk problem \eqref{eq:varDiscreteSoftArea},
    and $\eps = c_0 h^\lambda$ for some $\delta > 0$.
    Define
    \begin{align*}
        \gamma = \min\left( \frac{1}{2}-\delta, \frac{5}{2}-\delta-s-\frac{\lambda}{2}, \min(0,s-\frac{3}{2}-\delta)+\frac{\lambda}{2}\right)\text{.}
    \end{align*}
    Then there exists a constant $c>0$ independent of $h$ such that
    \begin{align*}
        \Vert u - u^h_\eps \Vert_{H^2(\Omega)}
        & \leq c h^\gamma\text{.}
    \end{align*}
    In particular, $\Vert u-u^h_\eps\Vert_{H^2(\Omega)} \in \mathcal{O}(h^{\frac{1}{2}-\delta})$ for $s \leq \frac{3}{2}+\delta$ and $\lambda = 4-2s$.
\end{theorem}

\begin{proof}
    As of Lemma~\ref{lem:captureConditionSoftArea} we can apply Theorem~\ref{thm:AbstractError}.
    Noting continuity of $b_B(\cdot,\cdot)$ on $H^s(\Omega)$ we obtain as in the proof
    of Theorem~\ref{thm:discretizationErrorSoftCurve} that
    \begin{align*}
        \Vert u - u^h_\eps \Vert_{H^2(\Omega)}
        & \lesssim \Vert u - v \Vert_{H^2(\Omega)}(\Omega) + \frac{1}{\sqrt{\eps}} \Vert u - v\Vert_{H^s(\Omega)} + h^{\min(0,s-\frac{3}{2}-\delta)}\sqrt{\eps} \text{.}
    \end{align*}
    Again choosing the interpolant $v = \overline{u}$ and applying the interpolation estimate \eqref{eq:interpolationEstimate} finally gives
    \begin{align*}
        \Vert u - u^h_\eps \Vert_{H^2(\Omega)}
        & \lesssim h^{\frac{1}{2}-\delta} + h^{\frac{5}{2}-\delta-s-\frac{\lambda}{2}} + h^{\min(0,s-\frac{3}{2}-\delta)+\frac{\lambda}{2}}
        \lesssim h^\gamma
    \end{align*}
    and thus proves the assertion.
\end{proof}


\section{Numerical Example}
\label{sec:computations}


\subsection{A symmetric example problem with known exact solution}
We
select a problem with a known,  analytically
computable solution.
We consider the circular domain $\hat{\Omega} = \{x \in \R^2 \st |x| < r_2\}$
with one embedded particle $\hat{B} = \{x \in \R^2 \st |x| < r_1\}$ for
$0 < r_1 < r_2 < 1$.
On the boundary of $\hat{\Omega}$ we consider homogeneous Dirichlet
boundary conditions whereas on the particle boundary
$\Gamma_1 = \partial B_1$ we impose variable-height
coupling conditions of the form \eqref{eq:var_height_bc}, i.\,e.,
\begin{align}
    u|_{\partial B_1} &= f^1_1 + \gamma^1, &
    \partial_\nu u|_{\partial B_1} &= f^1_2
\end{align}
for variable $\gamma^1 \in \R$ with the given functions
\begin{align*}
    f_1^1(re^{i\theta}) &= \cos(n \theta), &
    f_2^1(x) &= 0.
\end{align*}


We selected the parameters
\begin{align*}
    r_2 &= \tfrac23, &
    r_1 &= \tfrac13, &
    \kappa &= 1, &
    \sigma &= 0, &
    n &=4.
\end{align*}
Then the exact solution of Problem~\ref{prob:optimBase} on $\hat{\Omega}$
is the fourfold symmetric function $u \in H^2(\hat{\Omega})$ given by
\begin{align*}
    u(re^{i\theta}) =
        \begin{cases}
            \cos(4\theta) \, (c_1 r^4 + c_2 r^6) & \text{ for $r \in (0,r_1]$,}
            \\ \cos(4\theta) \, ( c_3 r^{-2} + c_4 r^{-4} + c_5 r^4 + c_6 r^6 ) & \text{ for $r \in [r_1,r_2]$,}
            \\ 0 & \text{ else}
        \end{cases}
\end{align*}
with the constants
\begin{alignat*}{3}
    c_1 & = 243,
    \qquad & c_2 & = -1458,
    \qquad & c_3 & =-6909/689,
    \\ c_4 & = -7936/502281,
    \qquad & c_5 & = 11502/689,
    \qquad & c_6 & = 44288/167427
    \text{.}
\end{alignat*}

In order to discretize the circular domain $\hat{\Omega}$ with the presented
Bogner--Fox--Schmit discretization, we embed $\hat{\Omega}$ into the larger
rectangular domain $\Omega = [-1,1]^2$ and treat the condition on
$\Gamma_2 = \partial \hat{\Omega}$ like a second particle boundary condition.
It is easily seen that the solution of Problem~\ref{prob:optimBase} on $\Omega$
is obtained by extending the solution given on $\hat{\Omega}$ by zero.
The solution of this problem is depicted in  \cref{fig:michellSolution}.

\begin{figure}
    \begin{center}
        \includegraphics[width=0.45\textwidth]{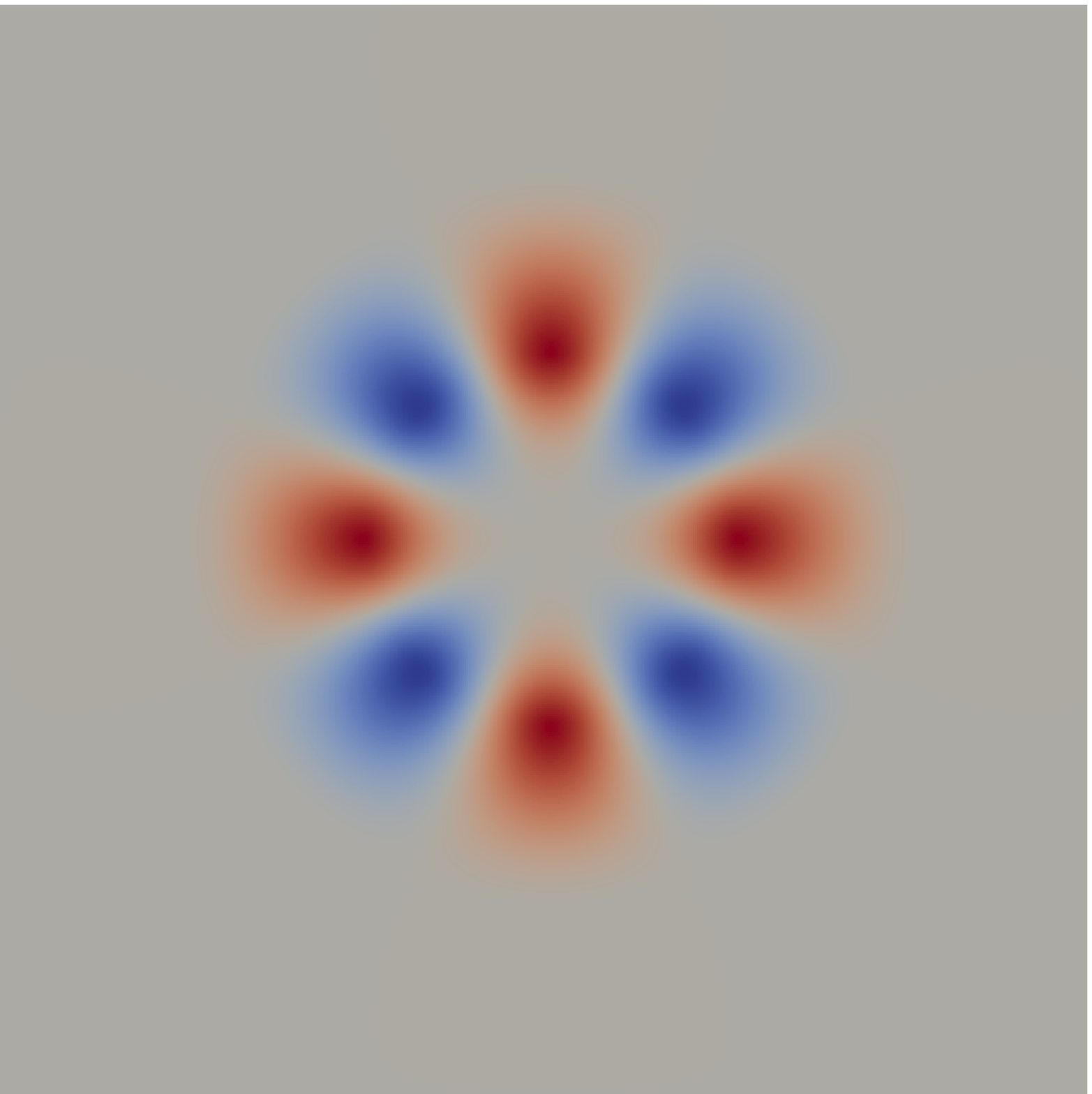}
        \includegraphics[width=0.45\textwidth]{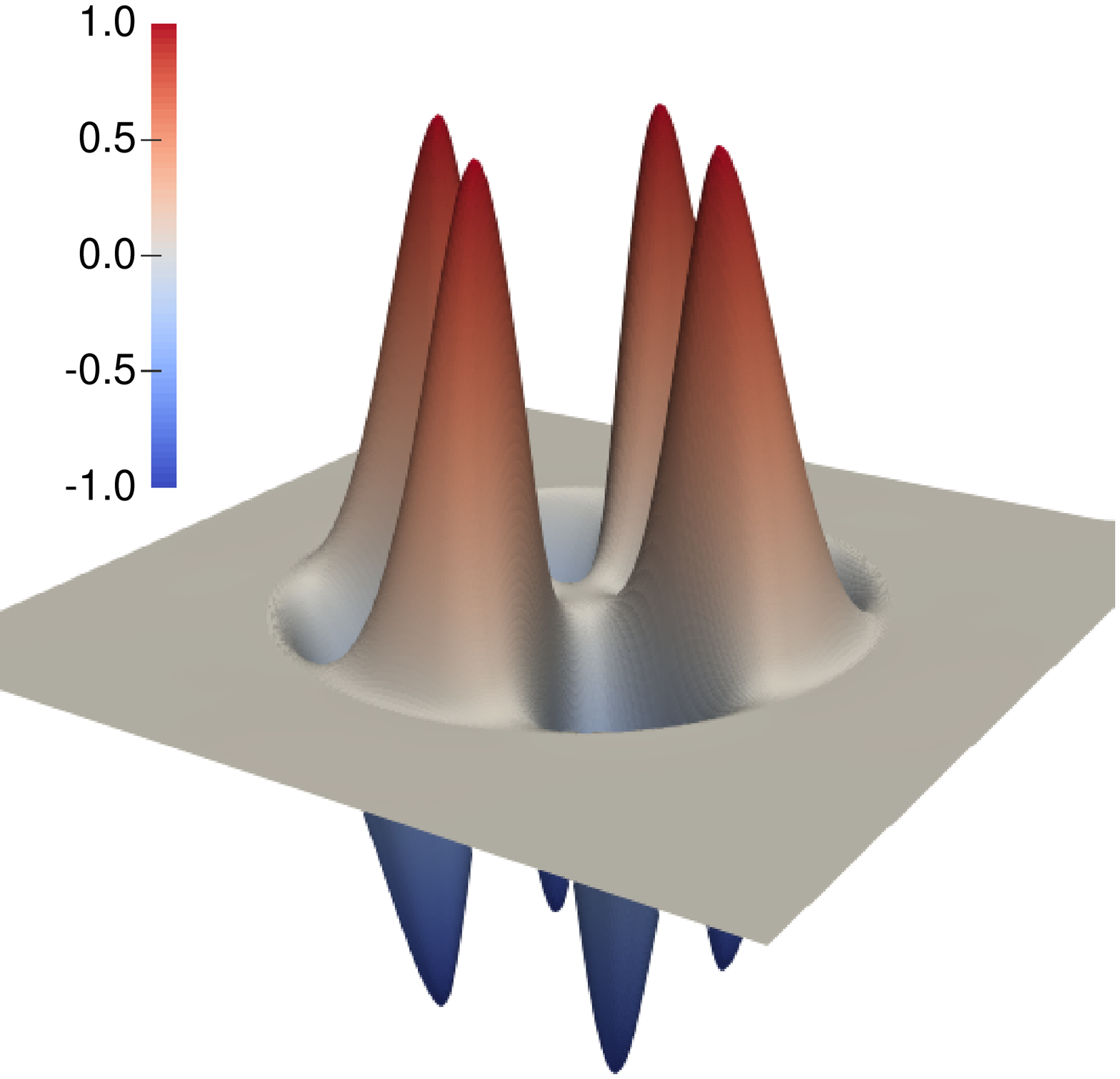}
    \end{center}
    \caption{Exact solution of the symmetric example problem. Left: Top view. Right: Rendered 3D view.}
    \label{fig:michellSolution}
\end{figure}

\subsection{Discretization, quadrature, and error measurement}
For the finite element discretization $S_h \subseteq H_0^2(\Omega)$
we divide $\Omega$ into uniform squares with edges of length $h$
and equip the resulting grid with a Bogner--Fox--Schmit finite element basis.
The boundary conditions over $\partial\Omega$ are enforced by setting the
corresponding degrees of freedom to zero while the boundary conditions over the
$\Gamma_i$ are replaced by penalized constraints.
The resulting soft curve problem with penalized curve constraints is then
given by \eqref{eq:varDiscreteSoftCurve} while the soft bulk problem with penalized
bulk constraints is given by \eqref{eq:varDiscreteSoftArea} where the
area of the virtual second particle is now given by $B_2 = \Omega \setminus \hat{\Omega}$.

We note that we drop the projection for the constraints on $\Gamma_2$ and $B_2$
because we do not allow variable height there.
While our analysis is only formulated in terms of the variable-height constraints
incorporated using projections, this is in fact no limitation, since all arguments
directly carry over to the case without these projections.


Since the assembly of the bilinear forms $b^i_\Gamma(\cdot,\cdot)$
and $b_B(\cdot,\cdot)$ involve integrals over the curves $\Gamma_i$
or the nontrivial domains $B_i$, respectively, we briefly mention
how these can be approximated.
Integrals over full grid cells are evaluated exactly using standard
tensor Gaussian quadrature rules.
For quadrature over the $\Gamma_i$ we use a trigonometric Gauss quadrature
as described in \cite{FieVia12} which in our case is both exact for
integration of finite element functions and for integration of $u$.
Integrals over the non-rectangular domains $B_i$ are approximated
using the local parameterization quadrature method introduced in
\cite{Olshanskii2016}, which is no longer exact.

The discretization error is measured in terms of the norm
$\| \Delta (u^h_\eps - u)\|_{L^2(\Omega)}$ on $H_0^2(\Omega)$.
It is important to note that the error cannot be computed accurately
by simple quadrature over the grid elements, because $u$ is
not smooth across the $\Gamma_i$.
Instead we split the integration into the subdomains
$B_1$, $B_2 = \Omega \setminus \hat{\Omega}$, and $\Omega \setminus (B_1 \cup B_2)$
where we apply the quadrature rules described above.

\subsection{Results for the soft curve formulation}
In case of the soft curve formulation we expect convergence with order $O(h^{\frac12})$
due to Theorem~\ref{thm:discretizationErrorSoftCurve} if the penalty parameters
are chosen suitably.
For the numerical examples we selected penalty parameters
$\eps = ( c h^{\lambda_1}, c h )$ with $\lambda_1 \in [1,3]$
according to Theorem~\ref{thm:discretizationErrorSoftCurve}.
The constant was fixed to be $c=10^{-3}$.

Figure~\ref{fig:errorSoftCurve} shows the behavior of the
$H^2$-discretization error over the mesh size $h$ for uniform grids
and penalty parameters with $\lambda_1 \in \{1, 2, 3\}$.
In accordance with Theorem~\ref{thm:discretizationErrorSoftCurve}
we observe convergence of at least order $O(h^{1/2})$.
While the observed rate seems to behave like $O(h^{1/2})$
for $\lambda_1 \in \{2,3\}$ it seems that the rate is slightly
better for $\lambda_1 = 1$. A possible explanation for this
observation is that the error contribution for the term $\|\cdot\|_{b_1}$
in the proof of Theorem~\ref{thm:discretizationErrorSoftCurve}
can be improved by smaller values of $\lambda_1$.
However, this does not lead to a better theoretical
bound due to other dominating terms.
\begin{figure}
    \begin{center}
        \includegraphics[width=0.31\textwidth]{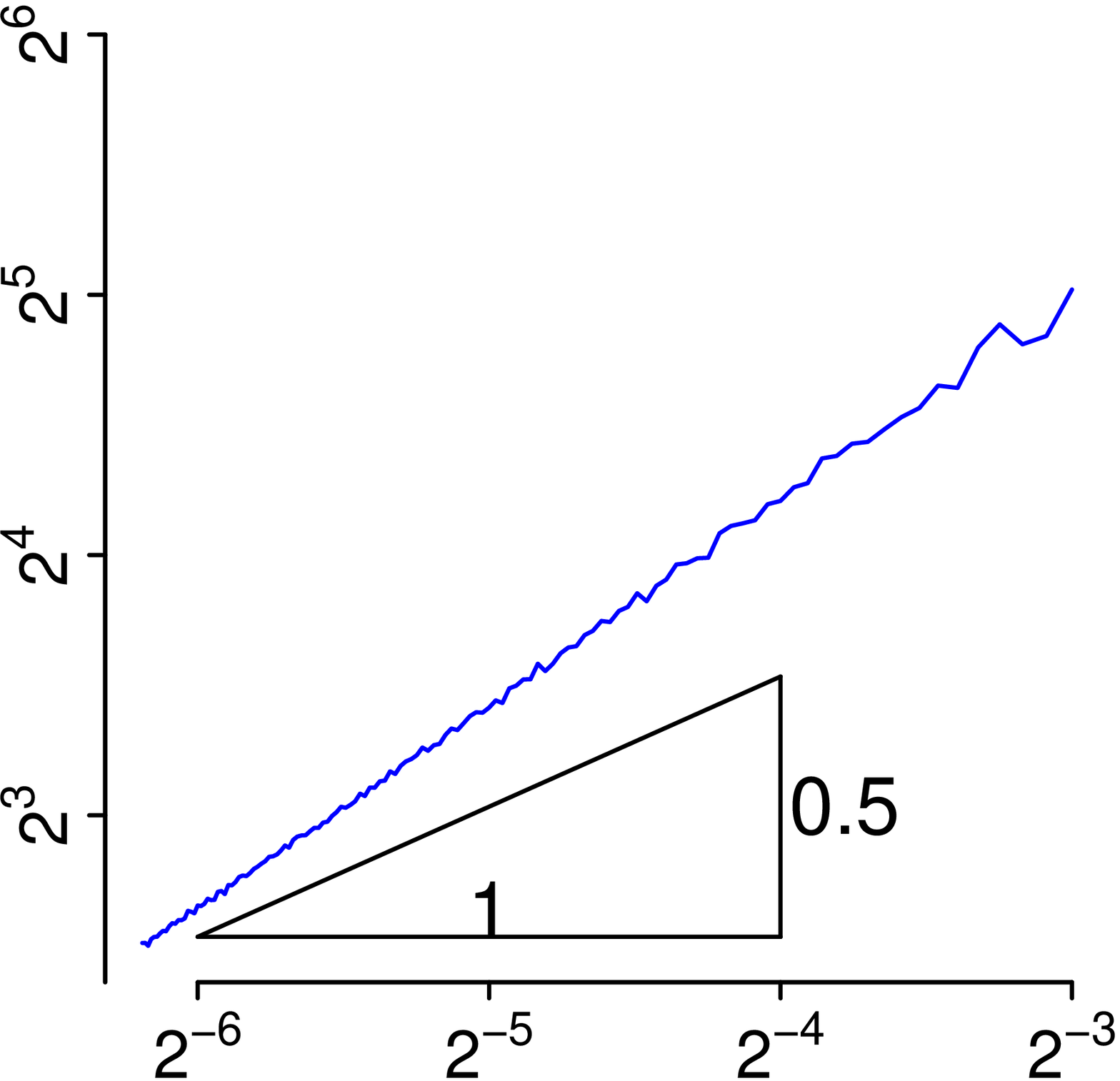}
        \includegraphics[width=0.31\textwidth]{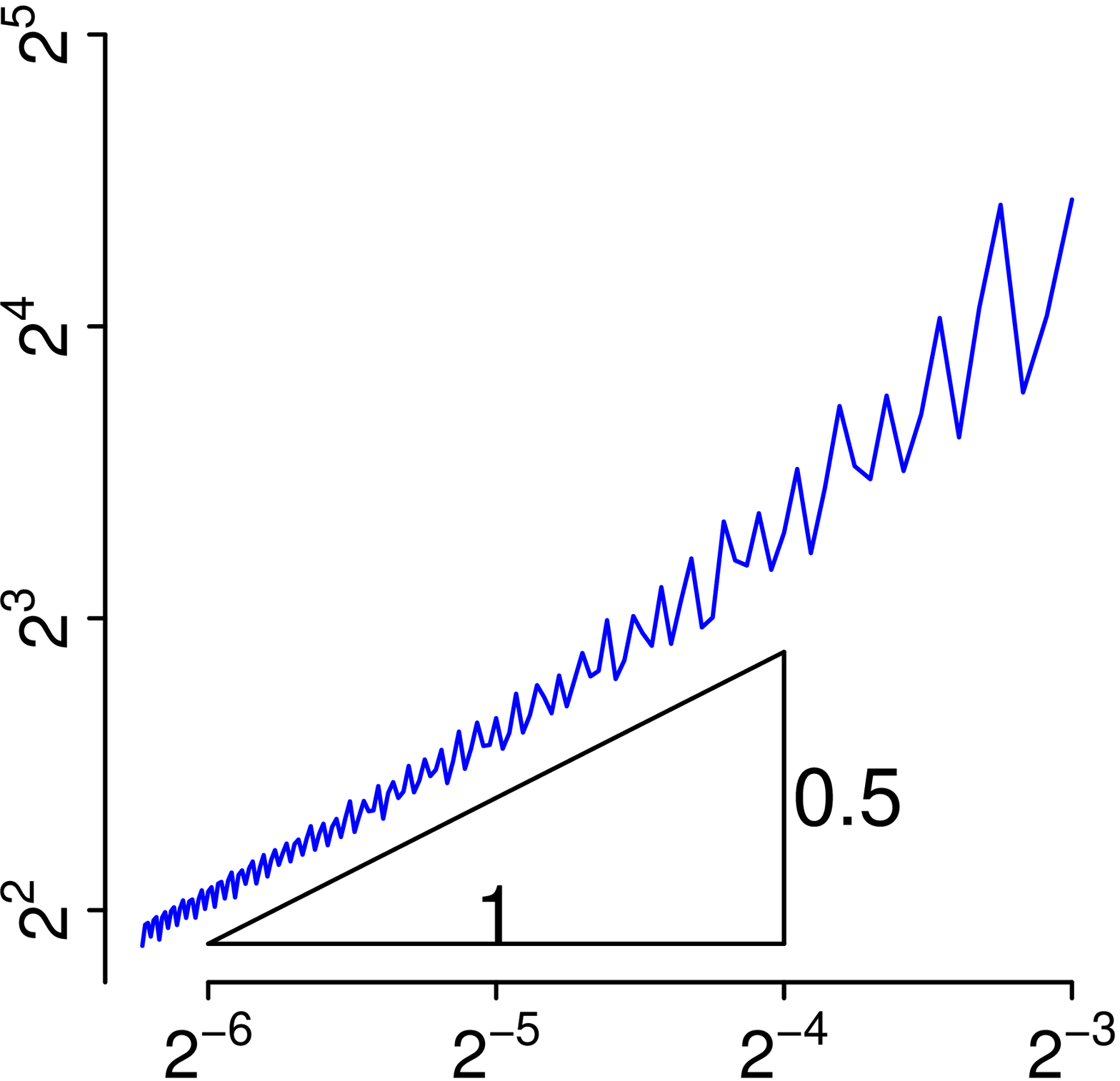}
        \includegraphics[width=0.31\textwidth]{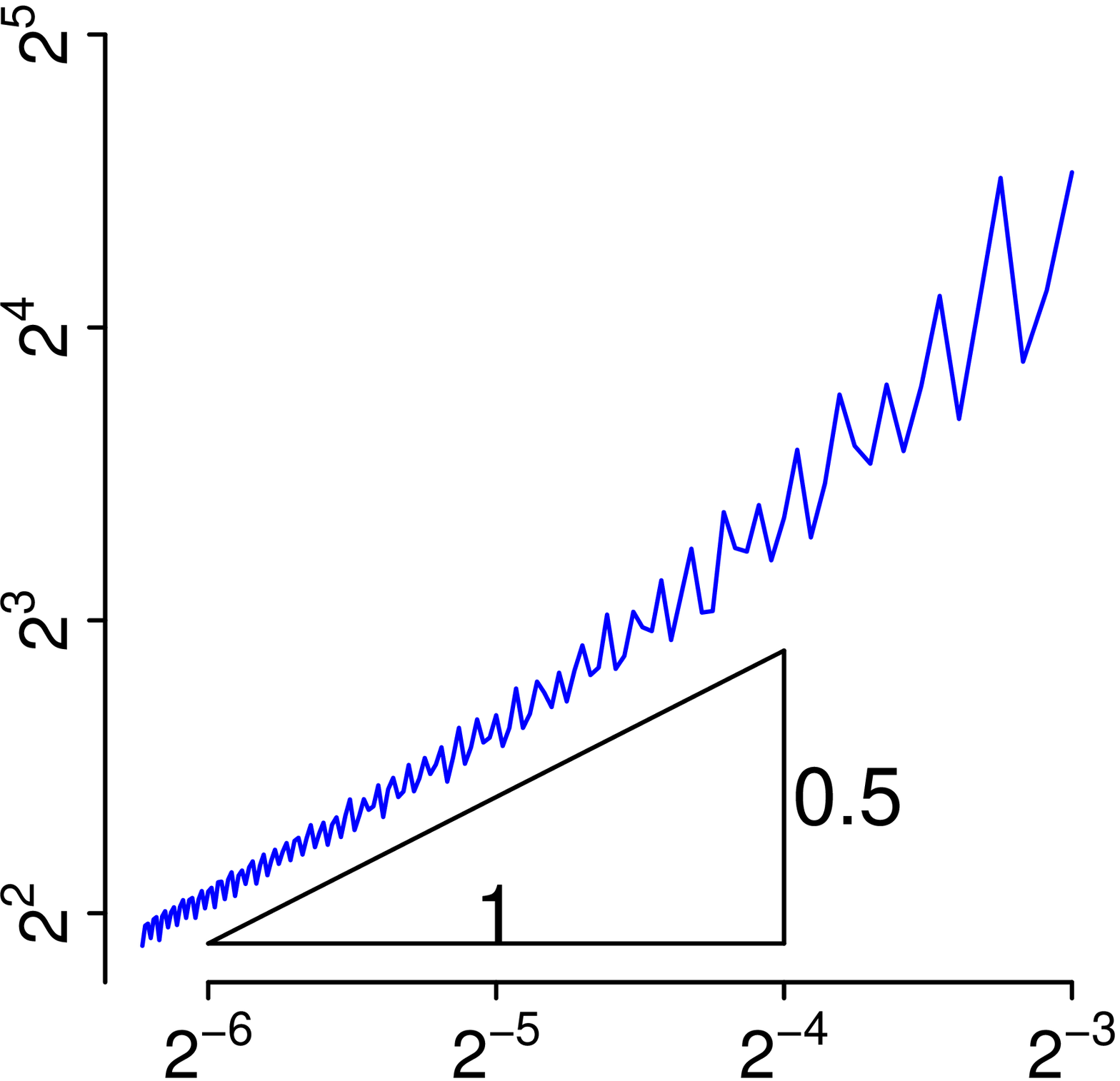}
    \end{center}
    \caption{%
        $H^2$-errors for the soft curve formulation for %
        $\lambda=1$, $\lambda=2$, and $\lambda=3$ (from left to right) %
        over the grid size $h = 1/N$ for $N = 8,\dots,75$.}
    \label{fig:errorSoftCurve}
\end{figure}

Figure~\ref{fig:H1errorSoftCurve} and Figure~\ref{fig:L2errorSoftCurve} show the $H^1$- and $L^2$-discretization errors,
respectively, for the same set of example problems.
In both cases we observe convergence with the order $O(h)$.
While improved convergence order for weaker norms is a well-known property
of many discretizations, we emphasize that this is not covered by the theory
presented here and may be considered in the future with a refined analysis.
\begin{figure}
    \begin{center}
        \includegraphics[width=0.31\textwidth]{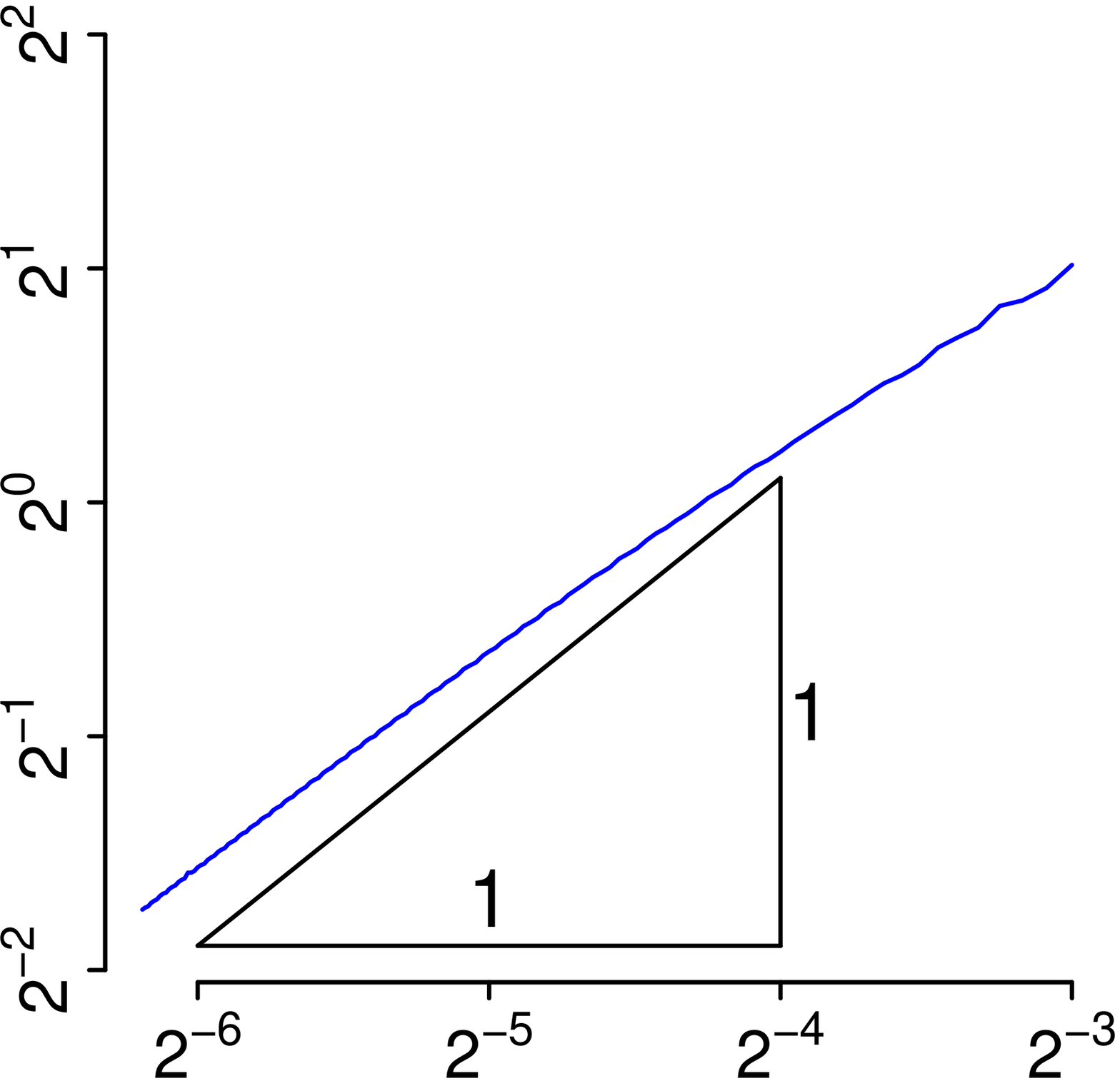}
        \includegraphics[width=0.31\textwidth]{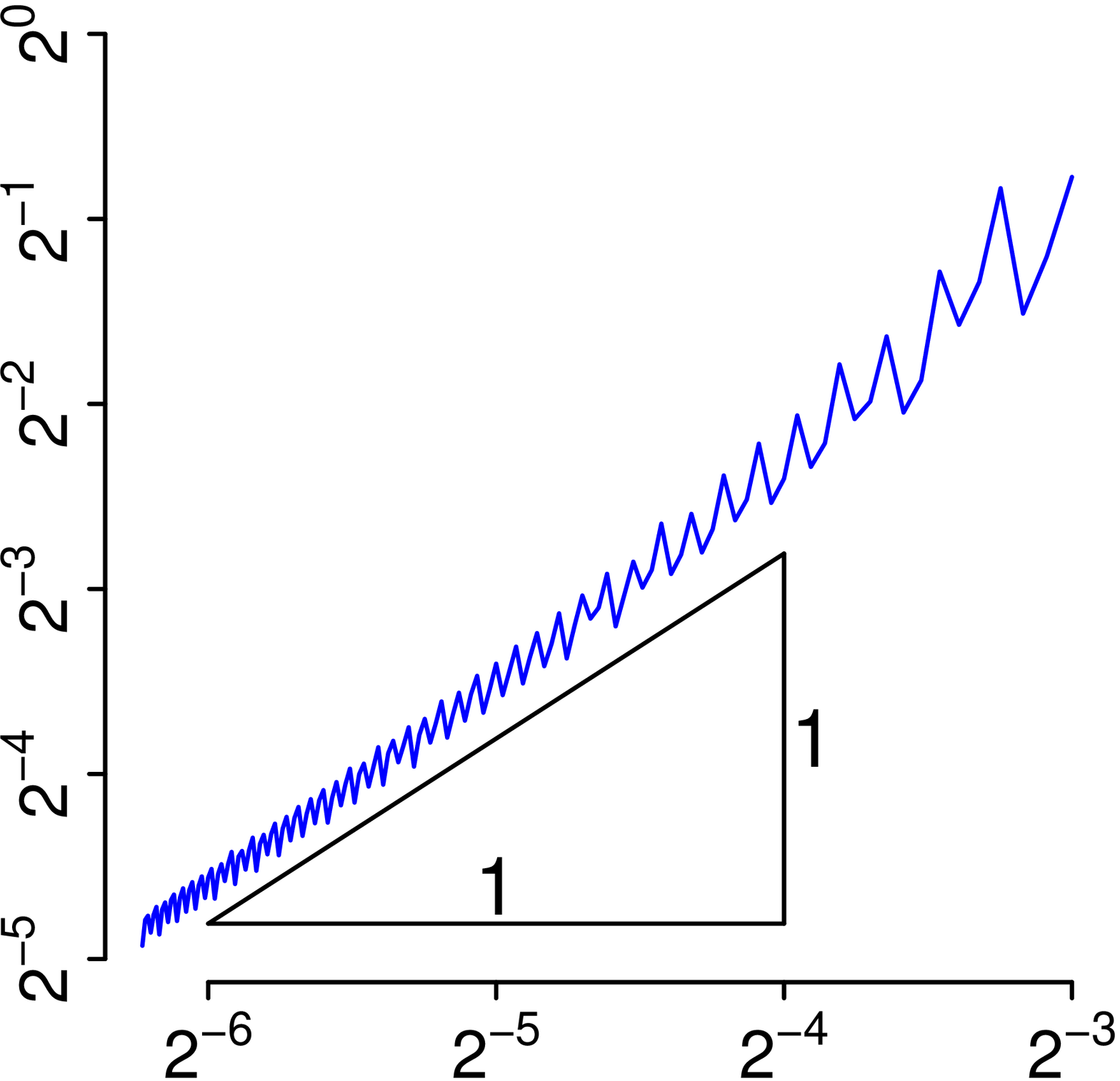}
        \includegraphics[width=0.31\textwidth]{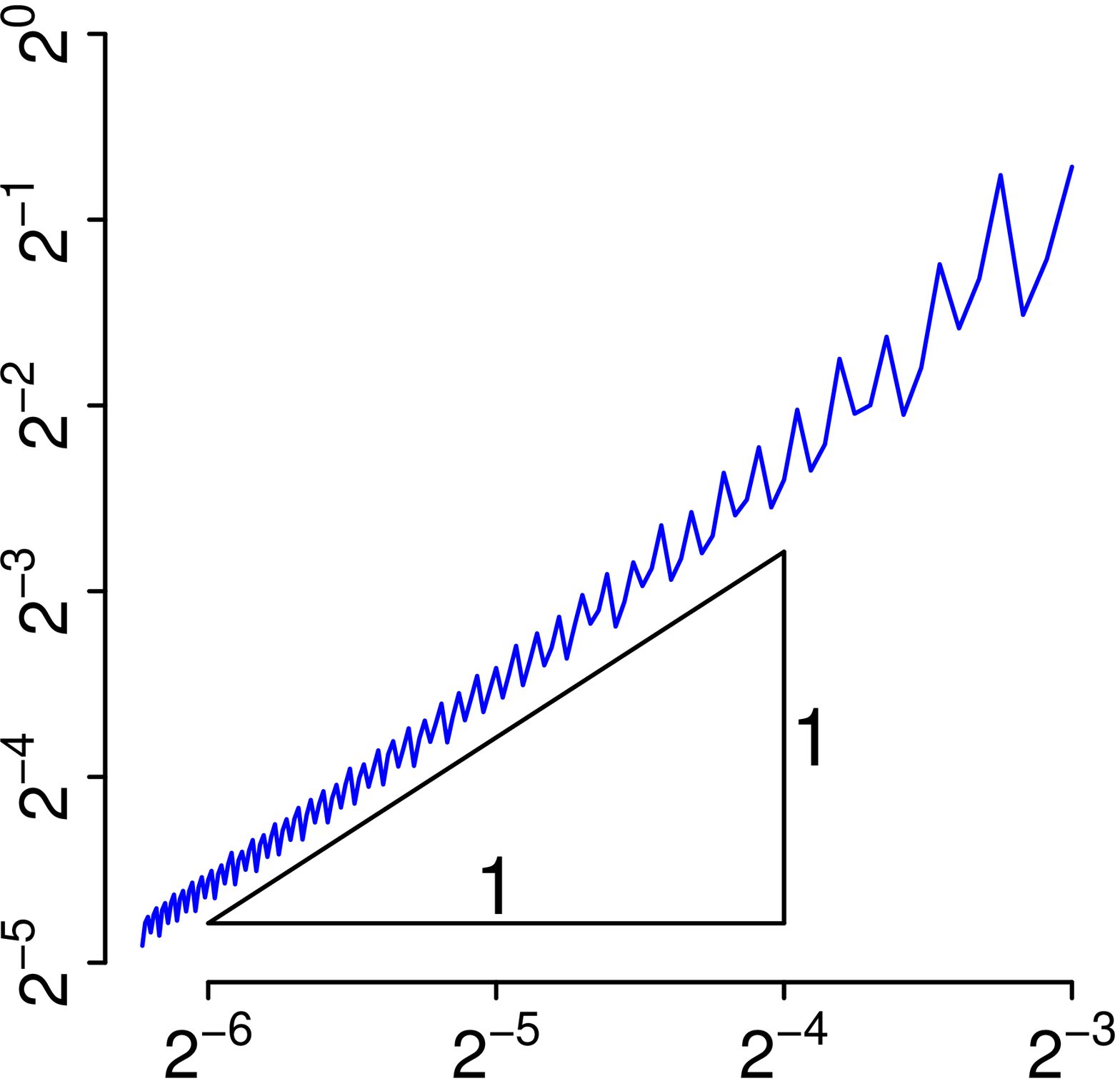}
    \end{center}
    \caption{%
        $H^1$-errors for the soft curve formulation for %
        $\lambda=1$, $\lambda=2$, and $\lambda=3$ (from left to right) %
        over the grid size $h = 1/N$ for $N = 8,\dots,75$.}
    \label{fig:H1errorSoftCurve}
\end{figure}
\begin{figure}
    \begin{center}
        \includegraphics[width=0.31\textwidth]{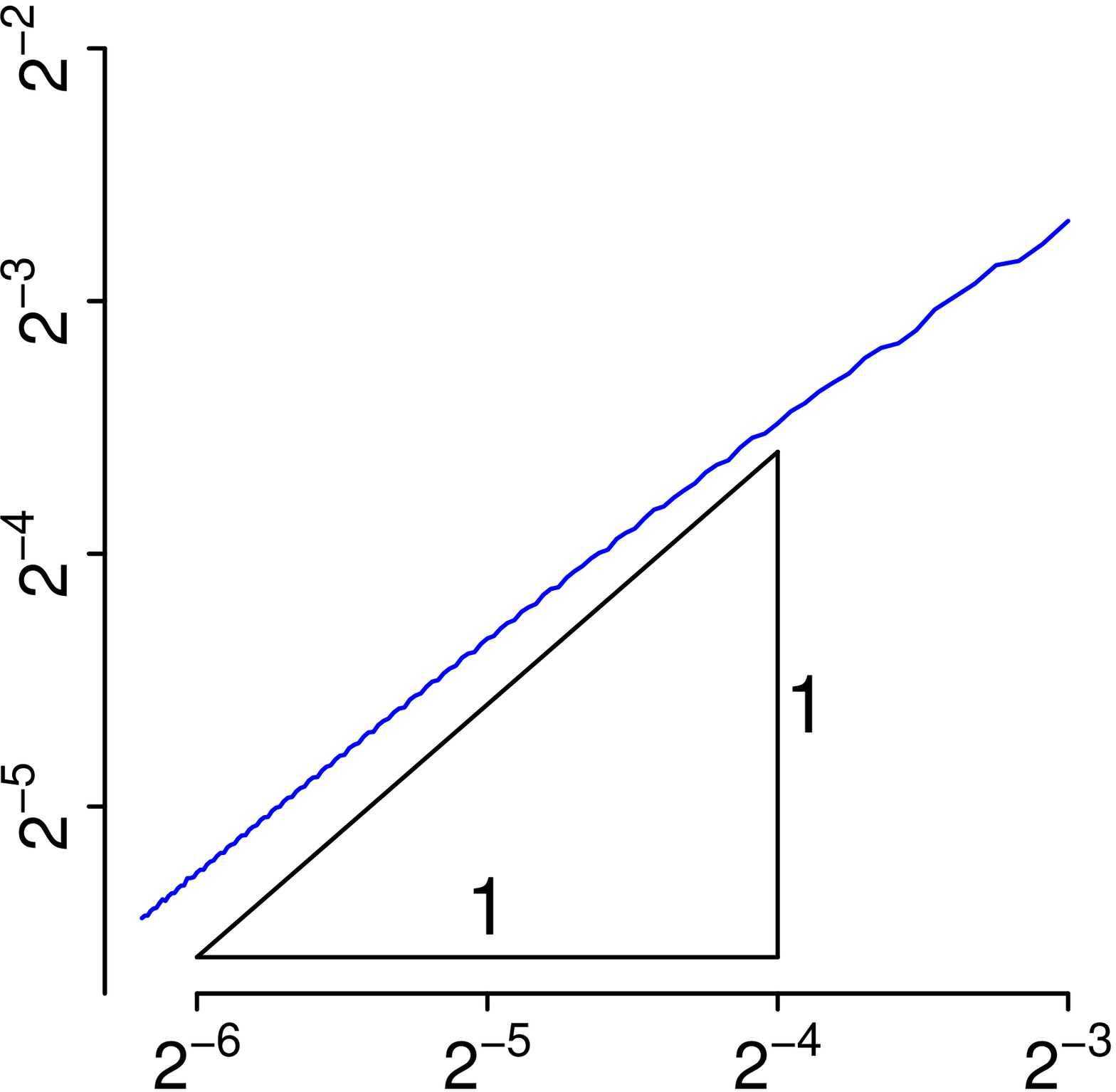}
        \includegraphics[width=0.31\textwidth]{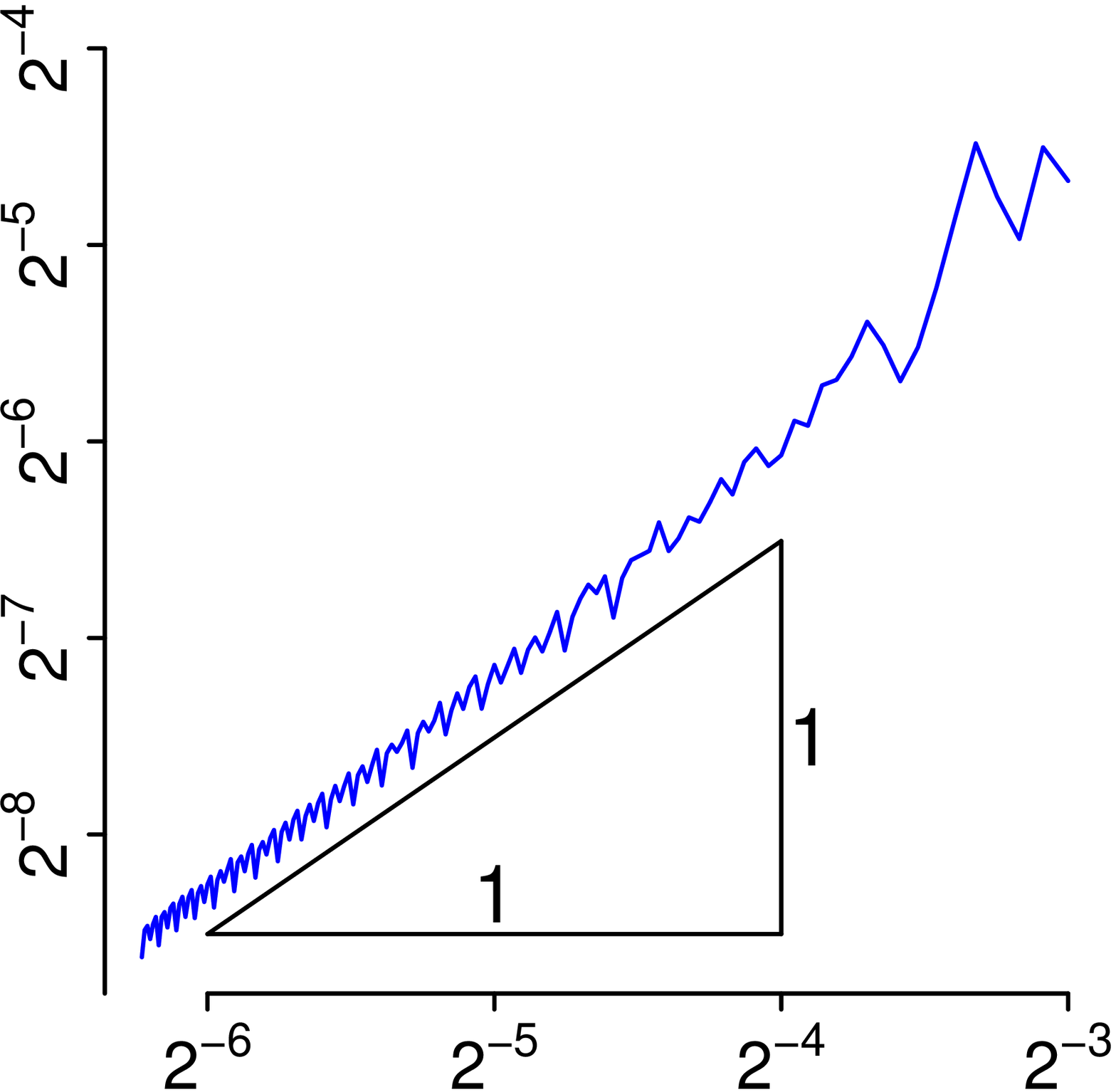}
        \includegraphics[width=0.31\textwidth]{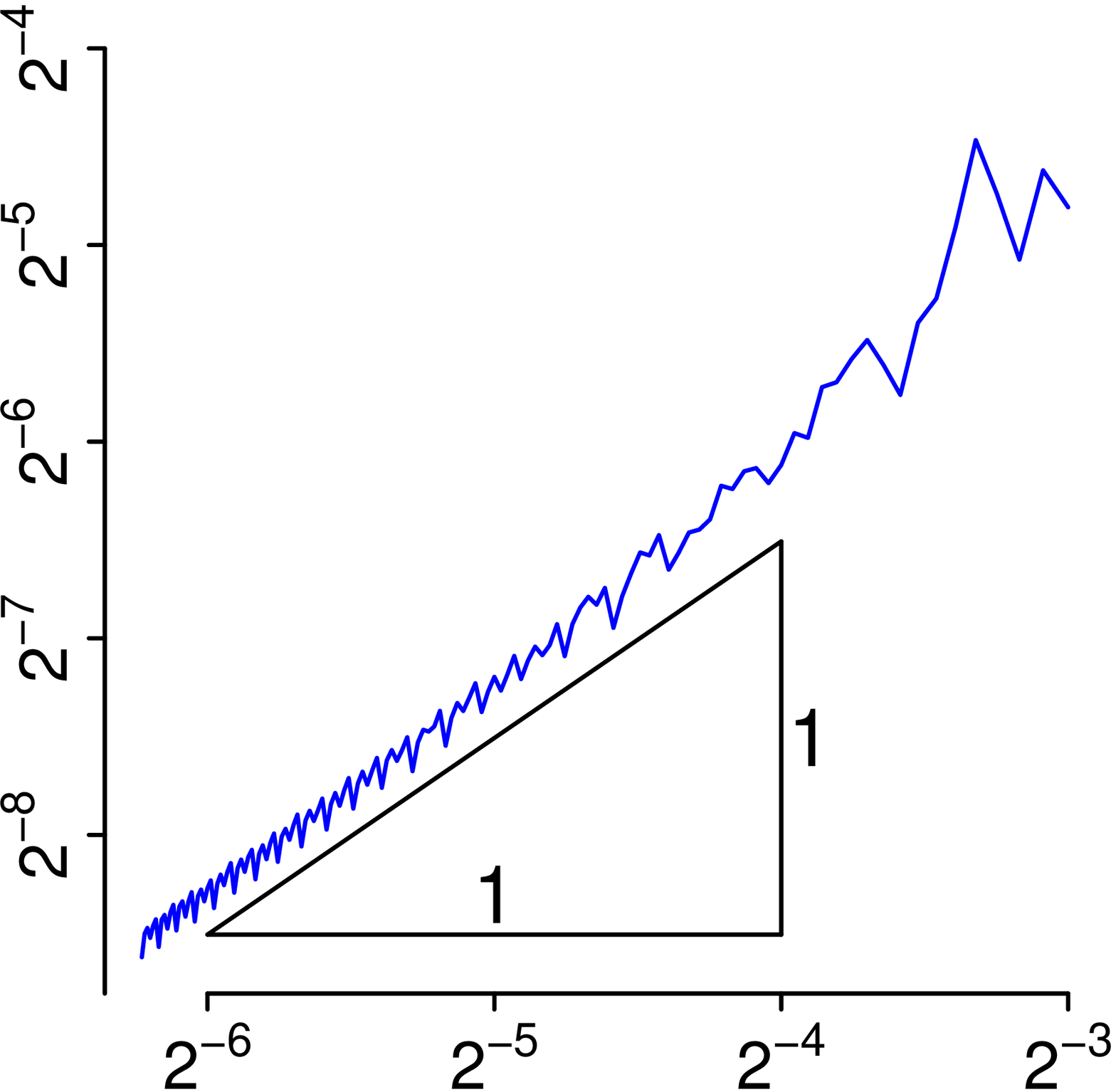}
    \end{center}
    \caption{%
        $L^2$-errors for the soft curve formulation for %
        $\lambda=1$, $\lambda=2$, and $\lambda=3$ (from left to right) %
        over the grid size $h = 1/N$ for $N = 8,\dots,75$.}
    \label{fig:L2errorSoftCurve}
\end{figure}

\subsection{Results for the soft bulk formulation}
Because we use uniform grids and do not refine with regard to the geometry
of the $\Gamma_i$ we lose control over the inverse estimate assumption
\eqref{eq:inverseEstimate}.
Thus we cannot prove condition \eqref{eq:captureCondition} in general.
Assuming that \eqref{eq:captureCondition} still holds true, we could
use Theorem~\ref{thm:discretizationErrorSoftArea} in case of exact integration.
However, as discussed above, we only approximate integrals using the
quadrature from \cite{Olshanskii2016}. If this is accurate enough,
we can in view of the Strang-type result in Proposition~\ref{prop:strang}
expect convergence of order $O(h^{1/2})$ when using
$\eps = c h^{4-2s}$ and the penalty norm $\|\cdot\|_{H^s(B_i)}$ for $s\in [0,1]$.
For our numerical experiments we considered $c = 10^{-3}$ and $s \in \{0,1\}$.

Figure~\ref{fig:errorSoftArea} shows the behaviour of the
discretization error over the mesh size $h$ for uniform grids
and $s\in \{0, 1\}$. Again the observed order is in accordance
with the expected order  $O(h^{1/2})$.
However, especially for $s=1$ this is perturbed by strong oscillations.
This is maybe due to the fact that the constant resulting from
\eqref{eq:captureCondition} is not uniformly bounded and can strongly
vary depending on how the curve $\Gamma$ intersects mesh elements.
Again the $H^1$- and $L^2$-errors depicted in Figure~\ref{fig:H1errorSoftArea}
and Figure~\ref{fig:L2errorSoftArea}, respectively, exhibit
a similar behavior but the improved convergence order $O(h)$.
\begin{figure}
    \begin{center}
        \includegraphics[width=0.31\textwidth]{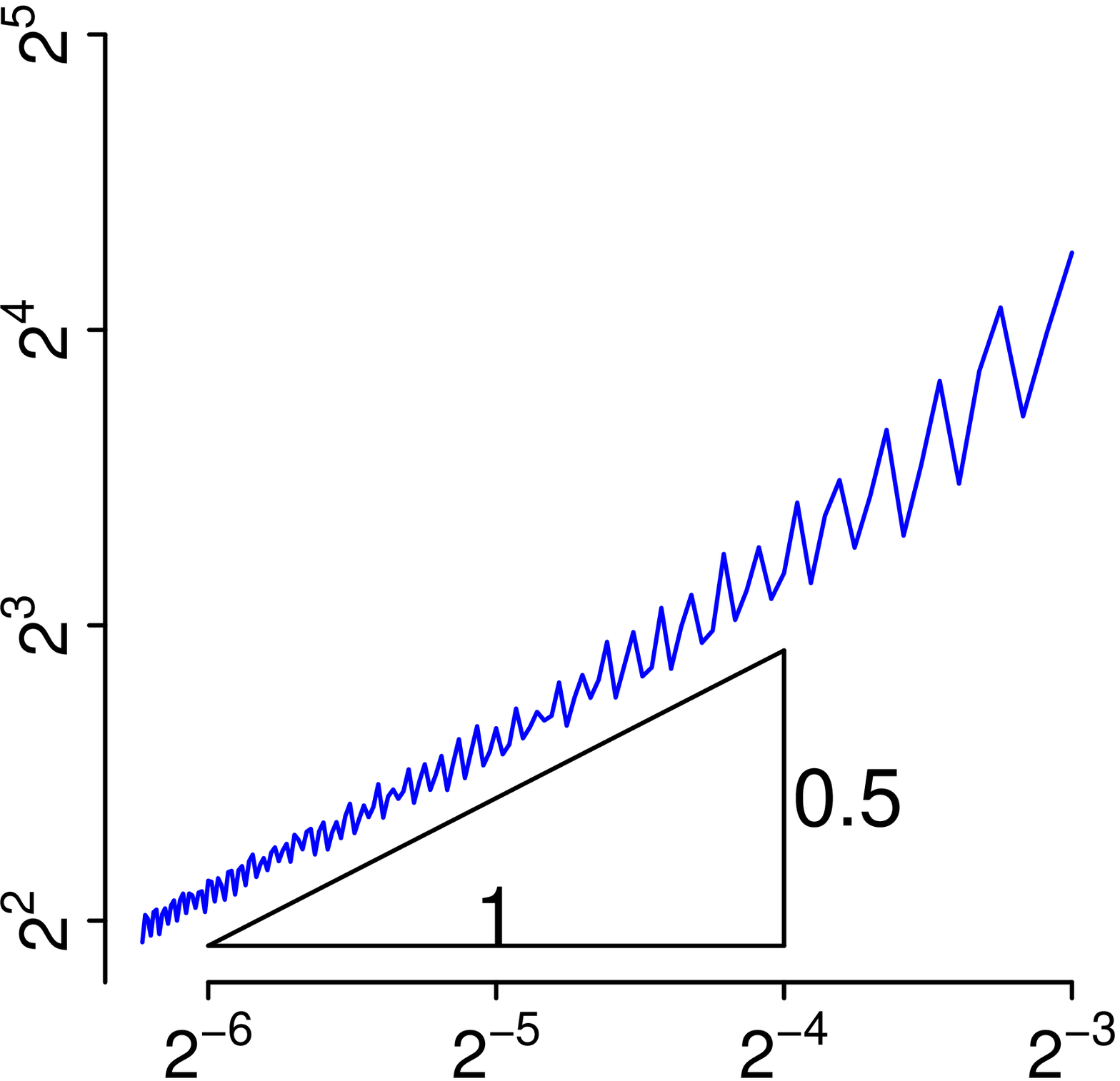}
        \hspace{0.7cm}
        \includegraphics[width=0.31\textwidth]{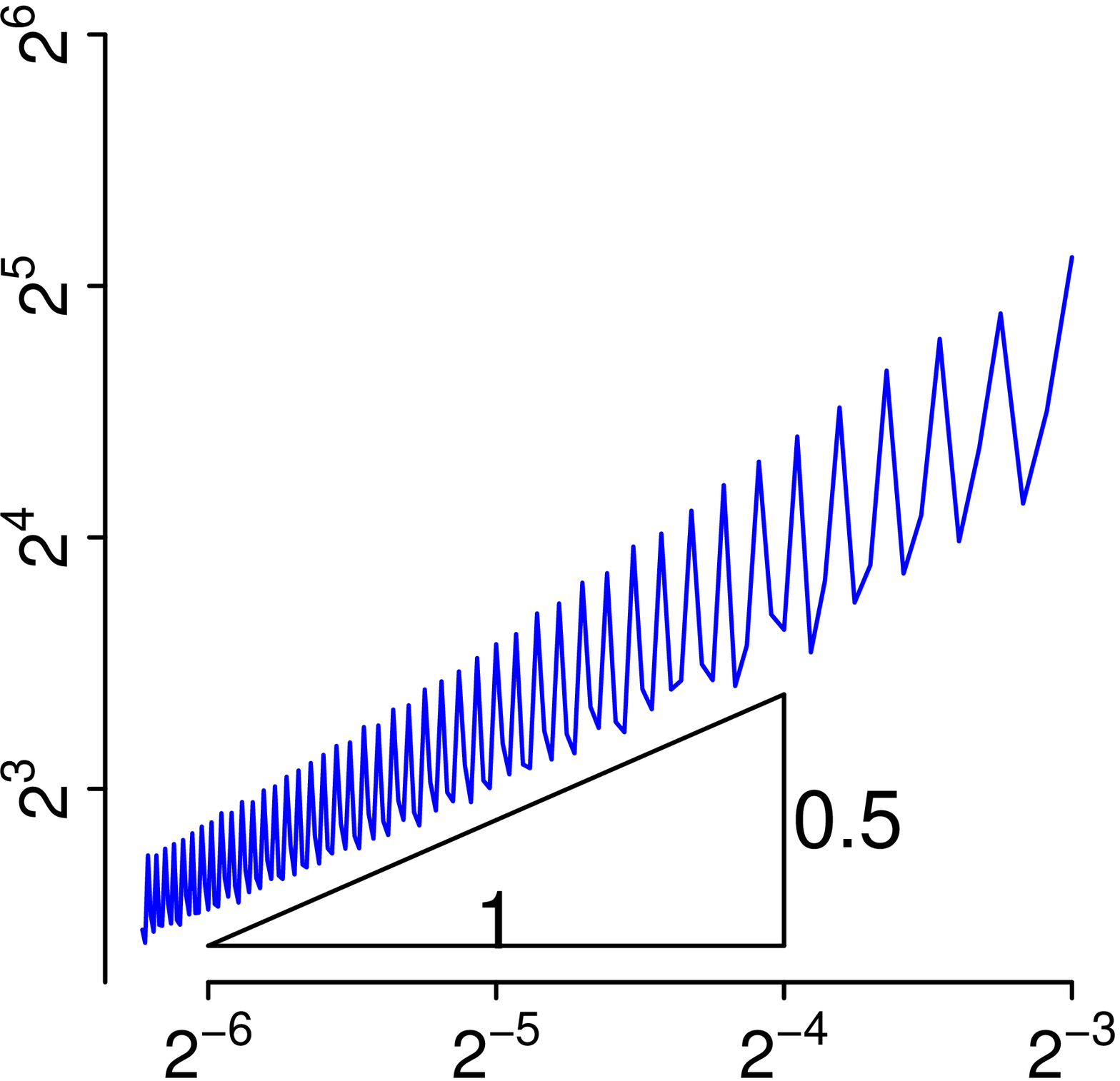}
    \end{center}
    \caption{%
        $H^2$-errors for the soft bulk formulation for %
        $s=0$ and $s=1$ (from left to right) %
        over the grid size $h = 1/N$ for $N = 8,\dots,75$.}
    \label{fig:errorSoftArea}
\end{figure}
\begin{figure}
    \begin{center}
        \includegraphics[width=0.31\textwidth]{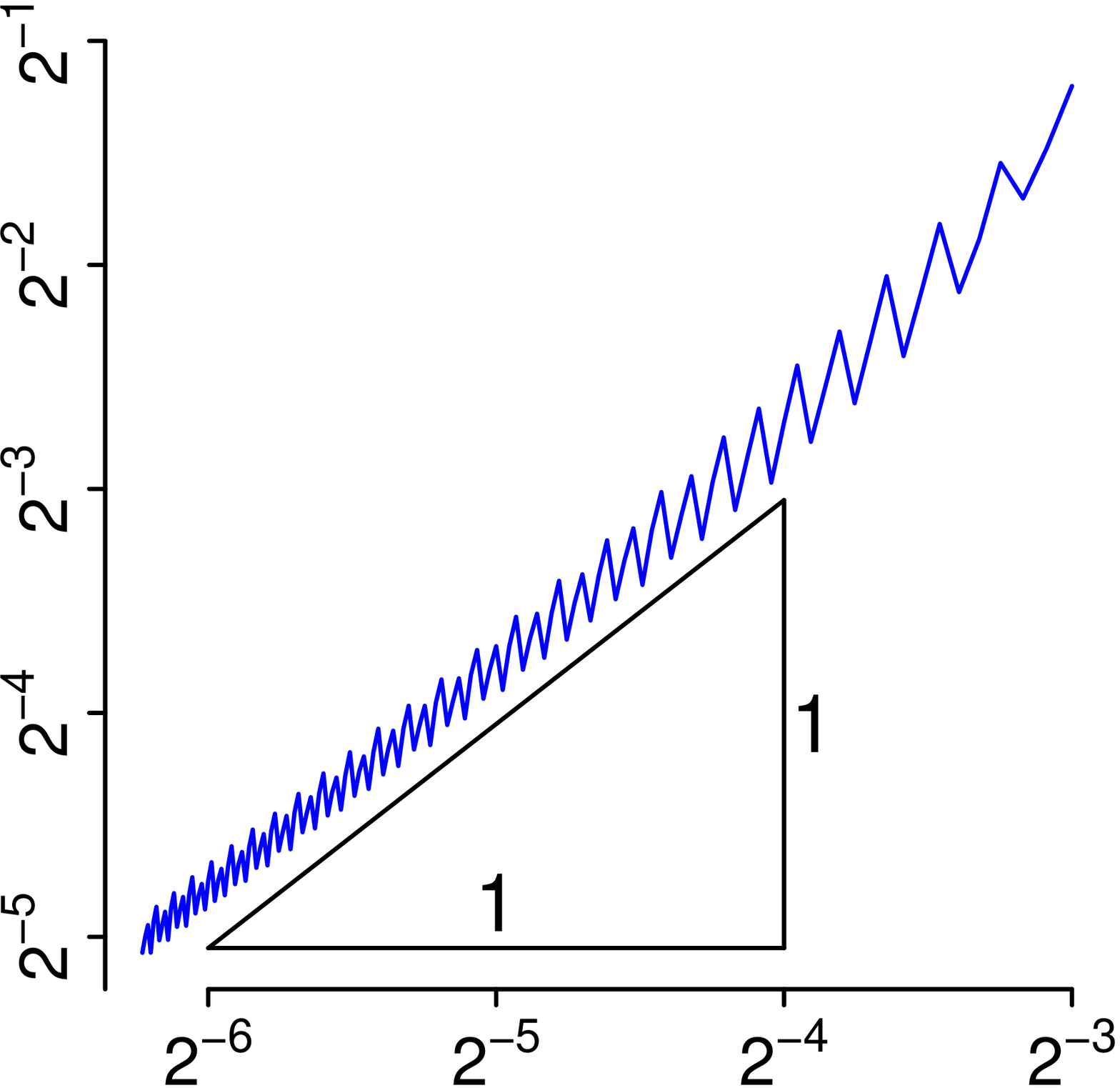}
        \hspace{0.7cm}
        \includegraphics[width=0.31\textwidth]{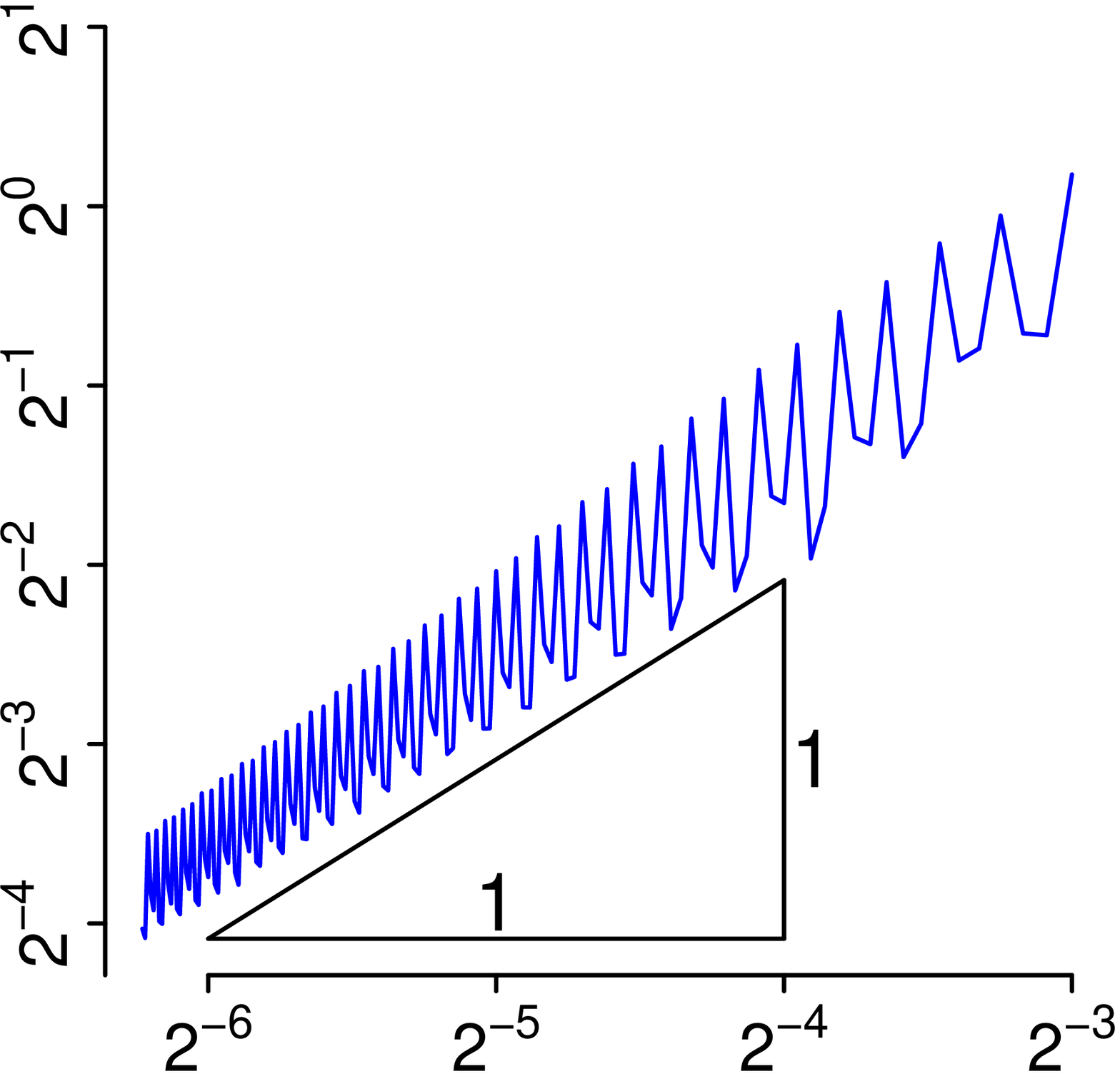}
    \end{center}
    \caption{%
        $H^1$-errors for the soft bulk formulation for %
        $s=0$ and $s=1$ (from left to right) %
        over the grid size $h = 1/N$ for $N = 8,\dots,75$.}
    \label{fig:H1errorSoftArea}
\end{figure}
\begin{figure}
    \begin{center}
        \includegraphics[width=0.31\textwidth]{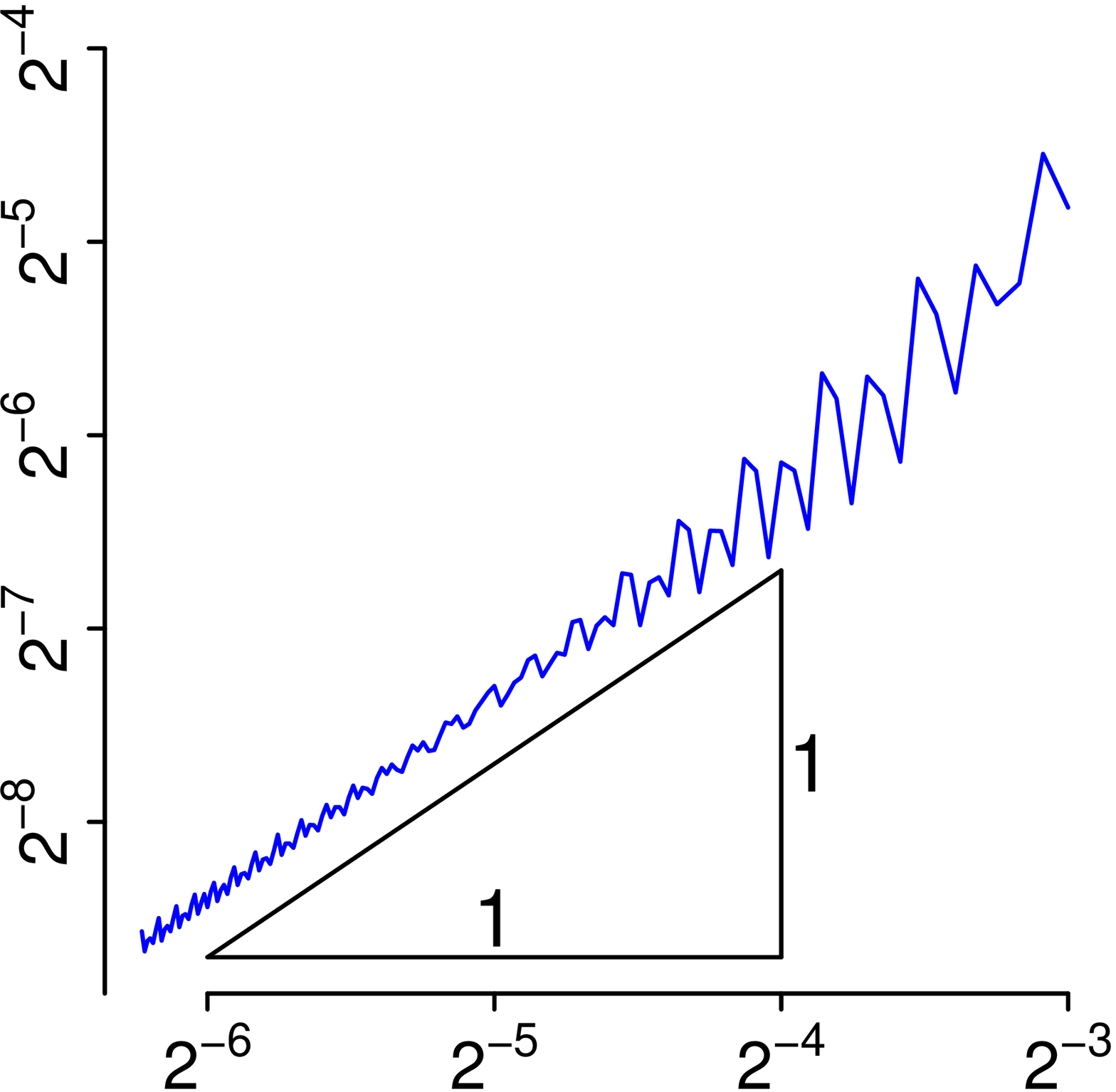}
        \hspace{0.7cm}
        \includegraphics[width=0.31\textwidth]{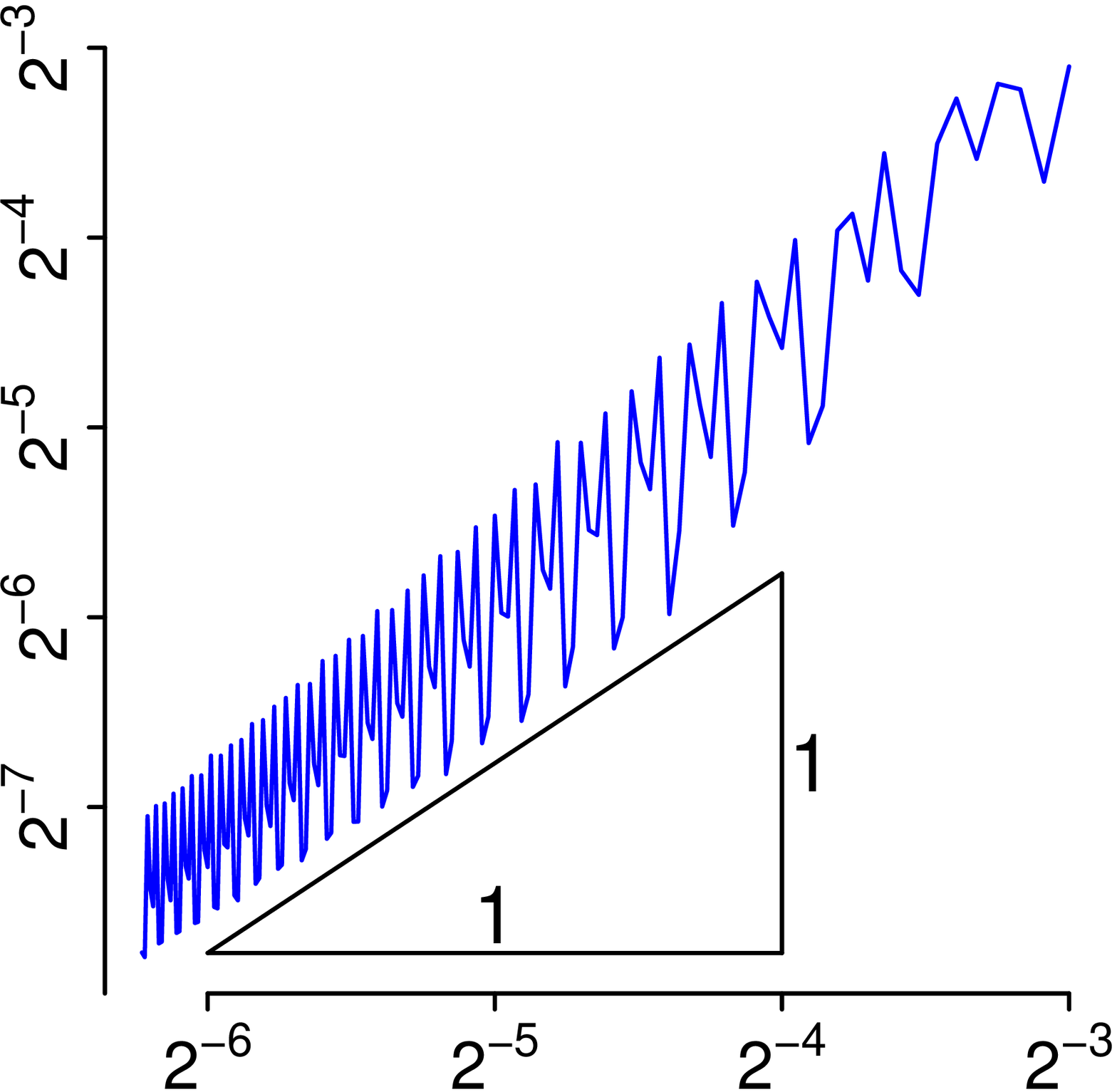}
    \end{center}
    \caption{%
        $L^2$-errors for the soft bulk formulation for %
        $s=0$ and $s=1$ (from left to right) %
        over the grid size $h = 1/N$ for $N = 8,\dots,75$.}
    \label{fig:L2errorSoftArea}
\end{figure}

\subsection{A non-symmetric example problem}
While the first example allowed to compute the exact solution due to its symmetry properties,
we will now assess the discretization for a non-symmetric example with several
particles but without known exact solution.

We consider the domain $\Omega = [-1,1]^2$ with homogeneous Dirichlet boundary
conditions on $\partial \Omega$ and we again select the parameters
\begin{align*}
    \kappa &= 1, &
    \sigma &= 0.
\end{align*}
This time we consider four elliptical particles embedded non-symmetrically into
the membrane domain $\Omega$.
The ellipses' major and minor axes have length $0.4$ and $0.2$, respectively.
Their positions and orientations are depicted in Figure~\ref{fig:genericSolution}.

Regarding the particle--membrane coupling conditions we assume that all particles
have a constant height profile $f^i_1 = 0$ and a constant slope $f^i_2$
varying for different particles. More precisely, we use
\begin{align*}
    f^1_2 &= 1, &
    f^2_2 &= 1, &
    f^3_2 &= -1, &
    f^4_2 &= -1.
\end{align*} 
The solution resulting from that setup is depicted in \cref{fig:genericSolution}.
\begin{figure}
    \begin{center}
        \includegraphics[width=0.45\textwidth]{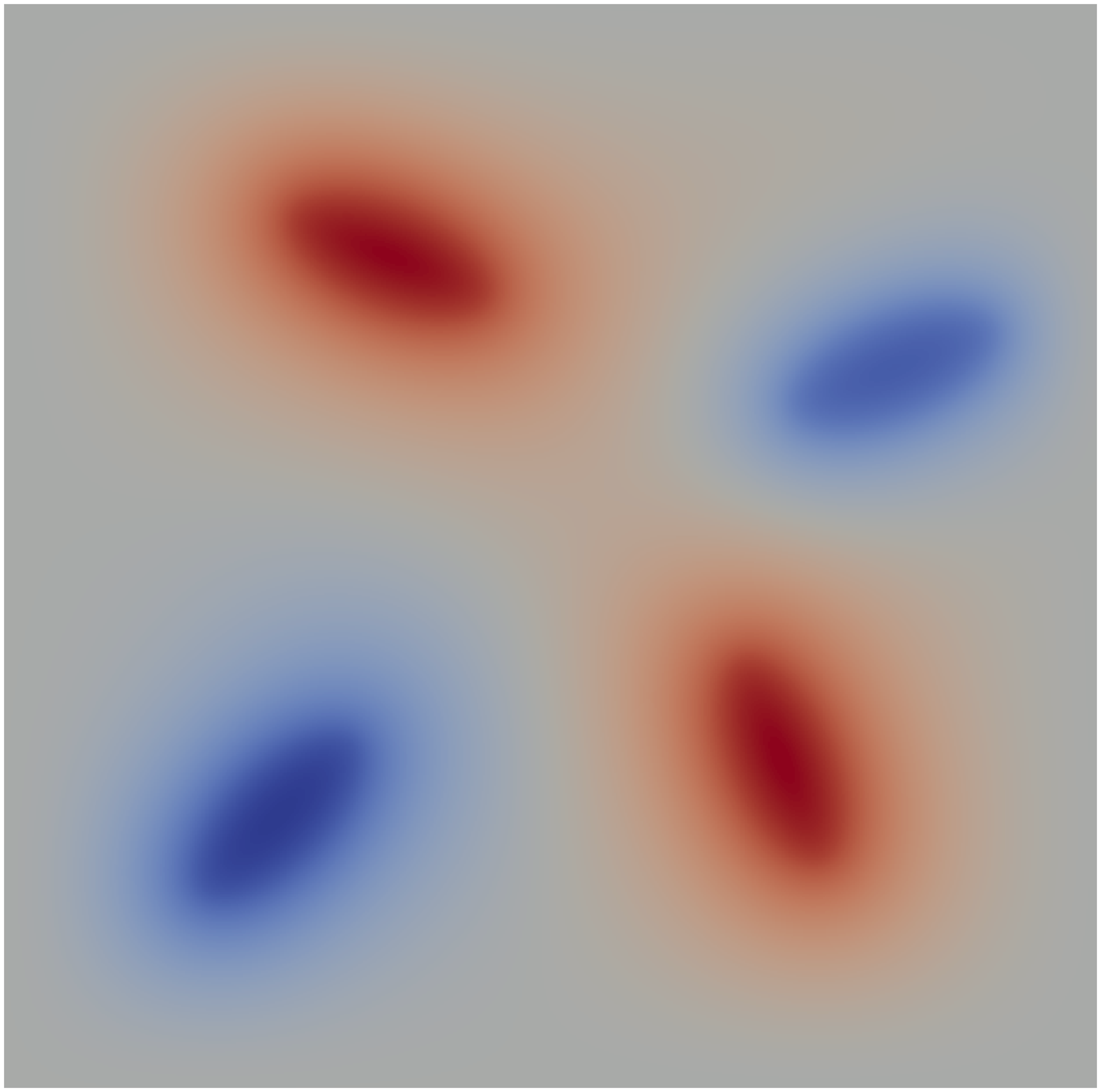}
        \includegraphics[width=0.45\textwidth]{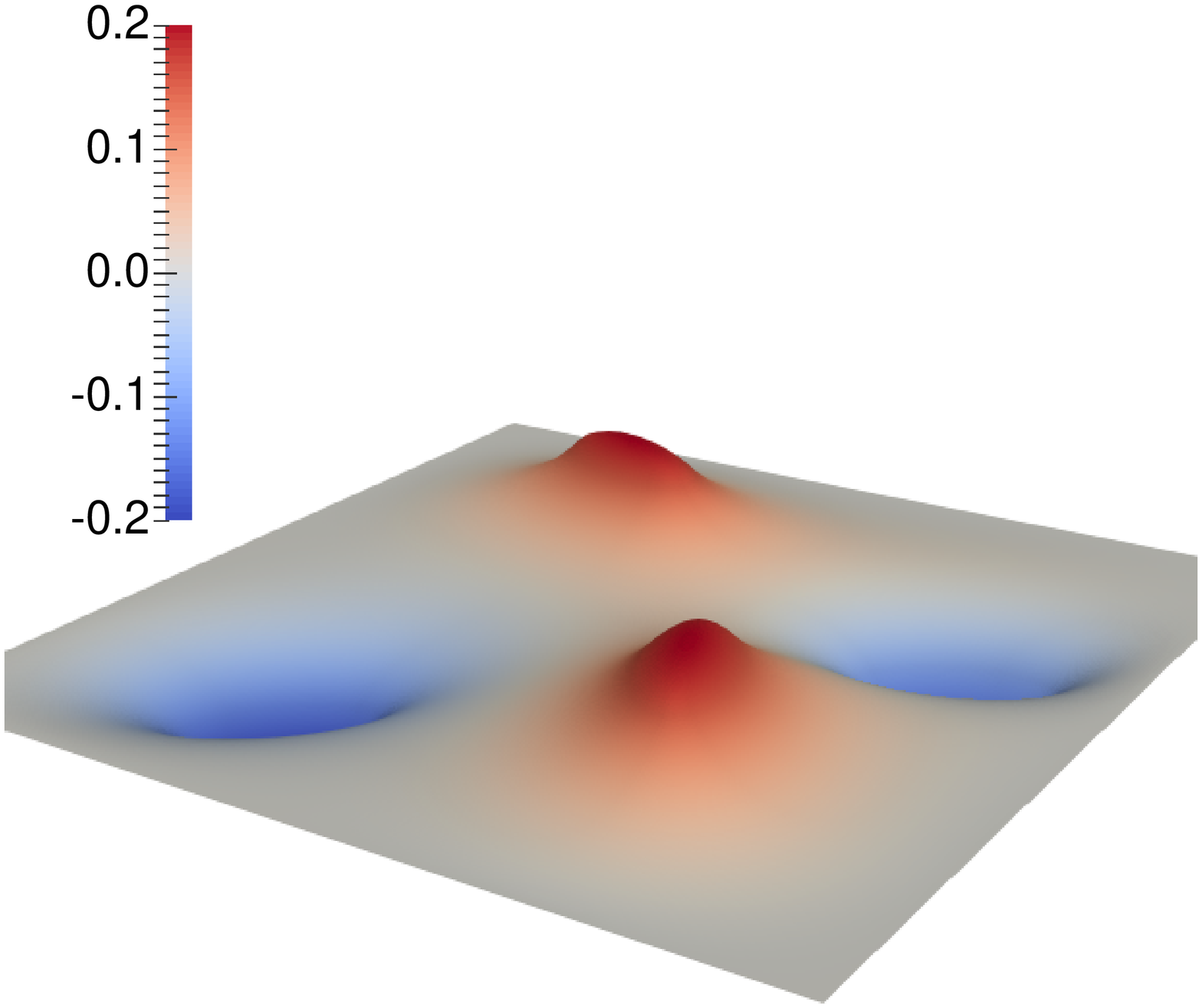}
    \end{center}
    \caption{Solution of the non-symmetric example problem. Left: Top view. Right: Rendered 3D view.}
    \label{fig:genericSolution}
\end{figure}

Since we do not know the exact solution for this problem we estimate the error
by comparing our discrete solutions to
a reference solution computed on a finer grid.
We computed the solutions for the different problem formulations on uniform
grids with mesh size $h=2^{-k}$
for $k= 3,\dots,6$ while the reference solution was computed on a uniform
grid with mesh size $h=2^{-8}$
using the soft curve formulation with $\eps = 10^{-3}(h^3,h)$.

The (approximate) discretization errors
for soft curve and soft bulk formulation are shown in
Figure~\ref{fig:errorSoftCurveGeneric} and Figure~\ref{fig:errorSoftAreaGeneric}, respectively.
As before we used $\eps = (ch^\lambda, ch)$ with $\lambda \in \{1,2,3\}$
for the soft curve formulation
and $\eps = ch^{4-2s}$ with  $s \in \{0,1\}$ for the soft bulk formulation.
In both cases we selected $c = 10^{-3}$.

For the soft bulk formulation the projection $P_{B}$ may lead to a
densely populated matrix for those degrees of freedom located inside
of the particles.
To avoid the resulting computational effort for $s=1$ we used the $H^1(B_i)$-norm in the form
\begin{align*}
    \|v\|_{H^1(B_i)}^2 = \|\nabla v\|_{L^2(B_i)}^2 + \Bigl(\int_{B_i} v \, dx\Bigr)^2.
\end{align*}
Then the penalty term that incorporates the $H^1(B_i)$--projections takes the form
\begin{align*}
    b_{B}(w,v) = (\nabla w, \nabla v)_{L^2(B)},
\end{align*}
which results in a sparse matrix again.

For both formulations, soft curve and soft bulk, we observe a rate that is slightly better
than the expected order of convergence, namely $O(h^{1/2})$.
Note that we restricted the mesh size $h=2^{-k}$ to powers of two in this example
to simplify comparison with the reference solution.
Furthermore, in the case of soft bulk constraints with $s=0$
we were not able to carry out the computations for $h=2^{-6}$ due to hardware restrictions.
This explains the much smaller number of data points in the plots.
\begin{figure}
    \begin{center}
        \includegraphics[width=0.31\textwidth]{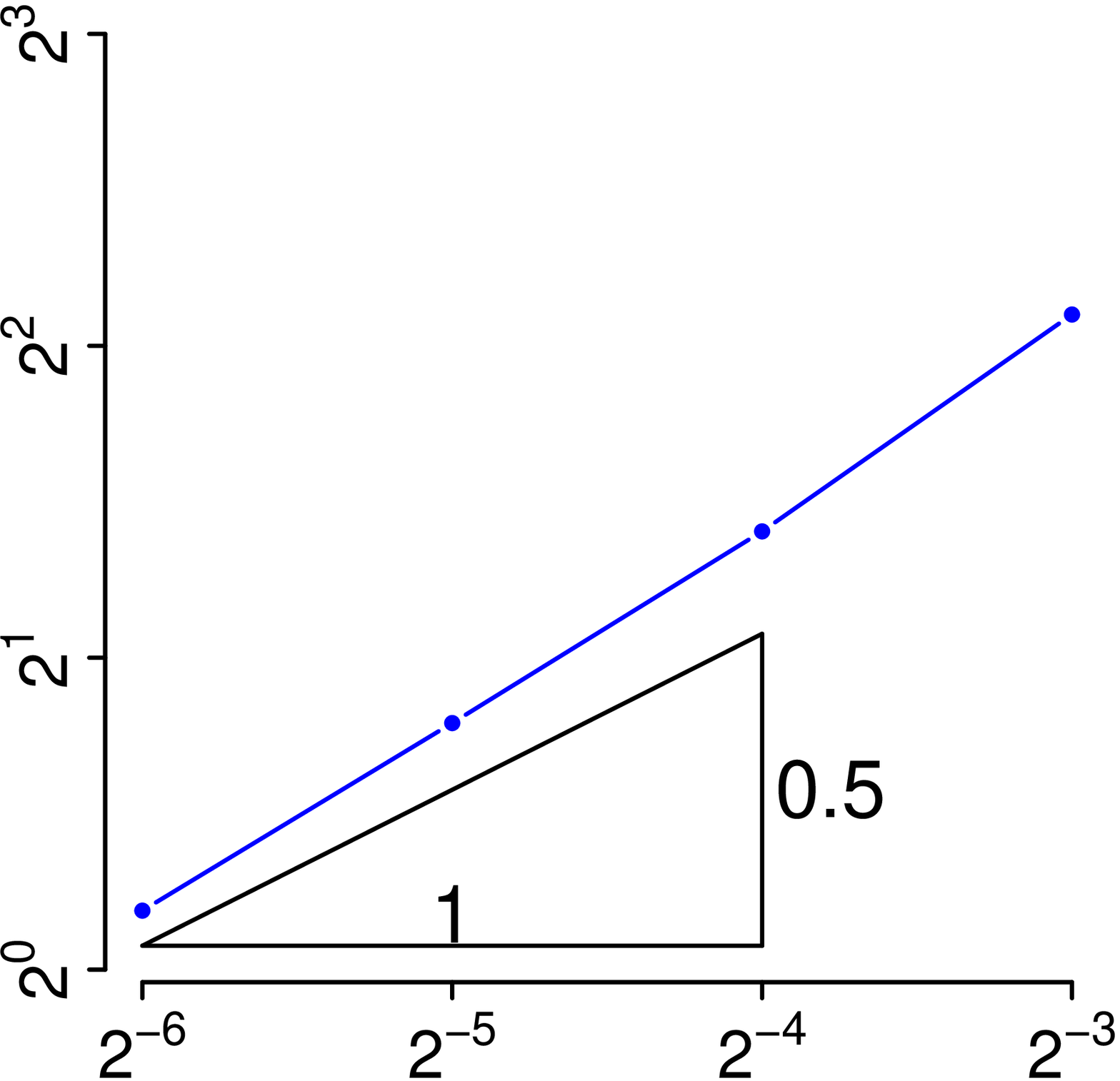}
        \includegraphics[width=0.31\textwidth]{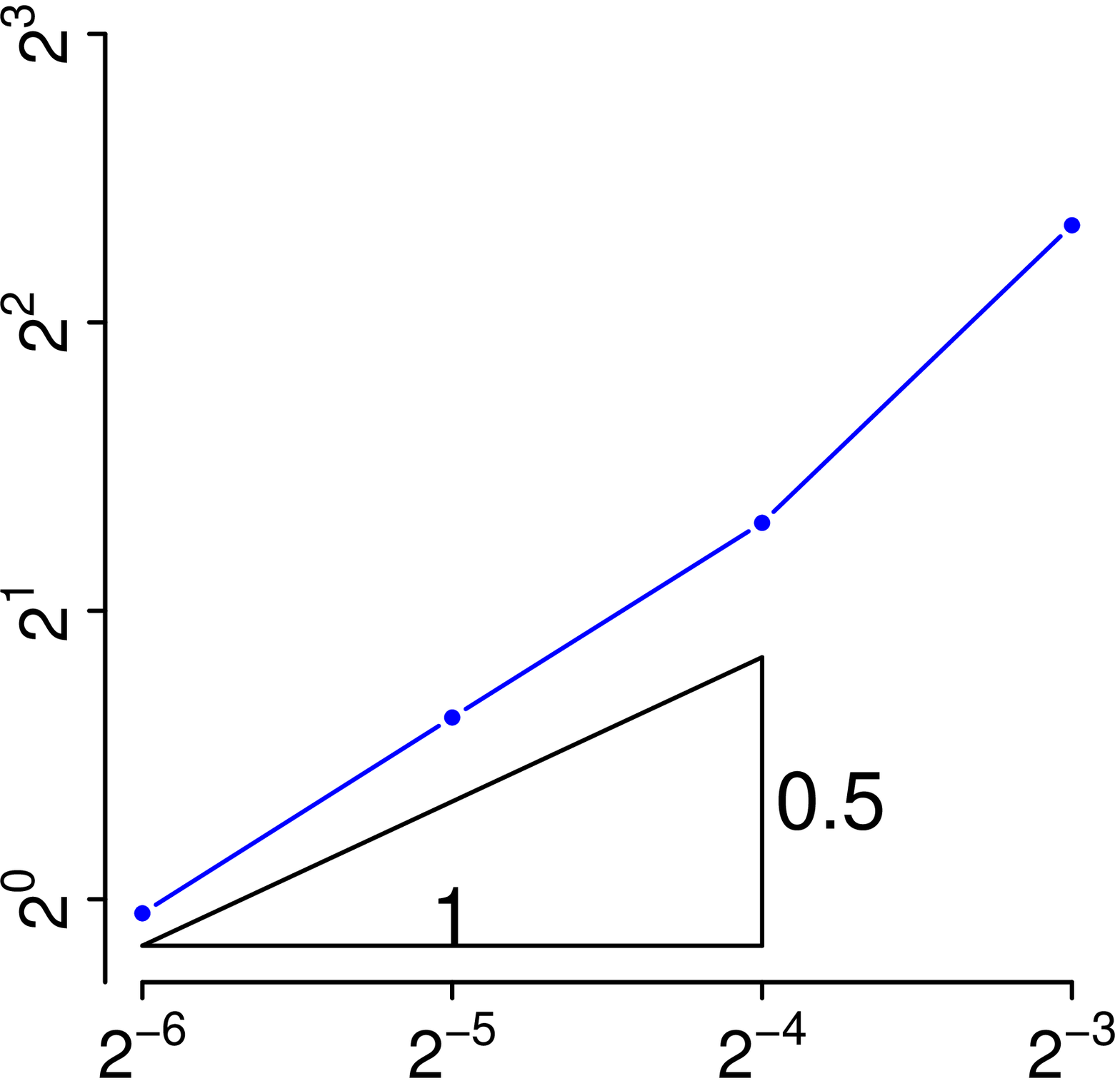}
        \includegraphics[width=0.31\textwidth]{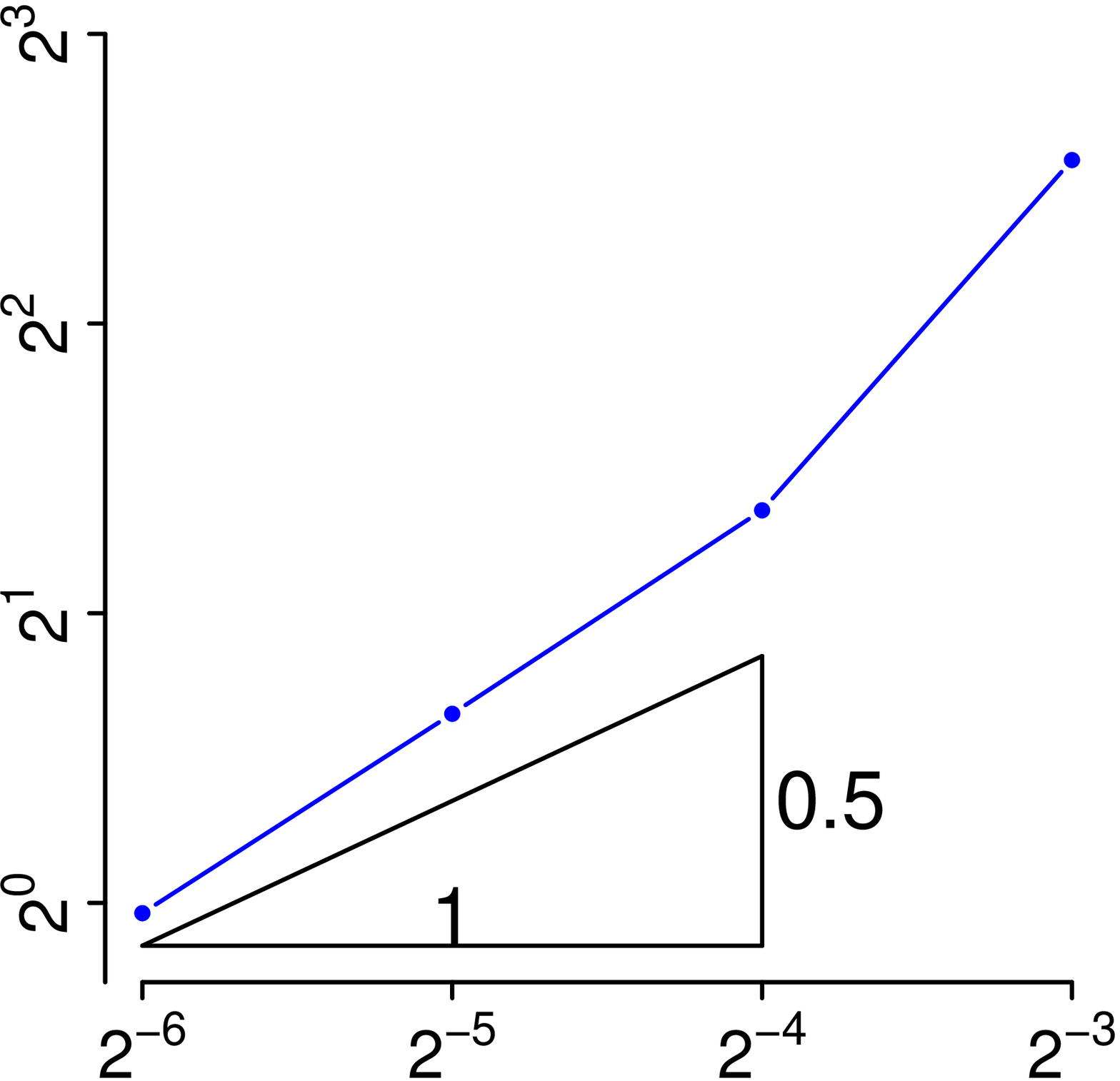}
    \end{center}
    \caption{%
        $H^2$-errors for the soft curve formulation for %
        $\lambda=1$, $\lambda=2$, and $\lambda=3$ (from left to right) %
        over the grid sizes $h=2^{-k}$, $k= 3,\dots,6$.}
    \label{fig:errorSoftCurveGeneric}
\end{figure}
\begin{figure}
    \begin{center}
        \includegraphics[width=0.31\textwidth]{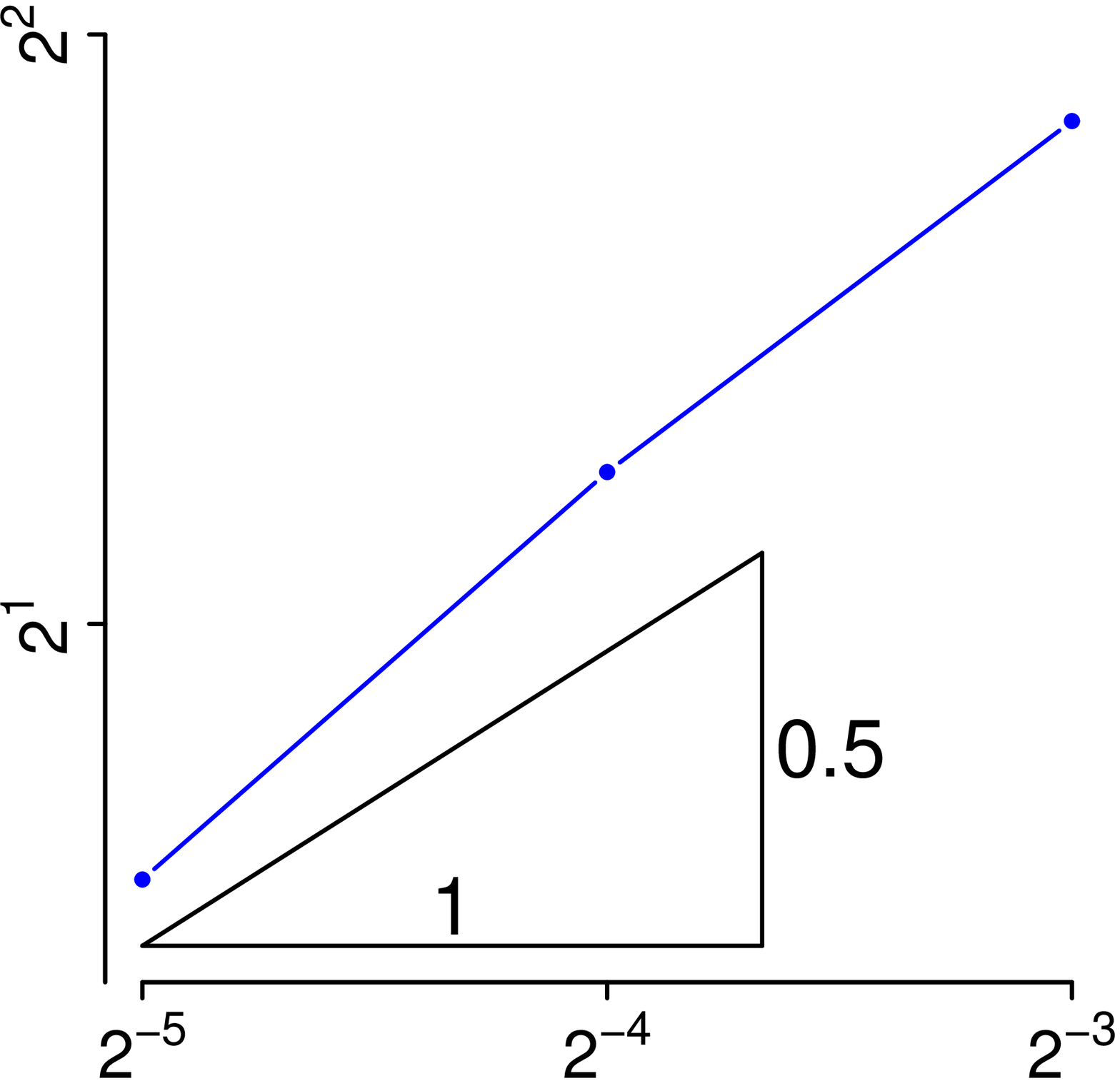}
        \hspace{0.7cm}
        \includegraphics[width=0.31\textwidth]{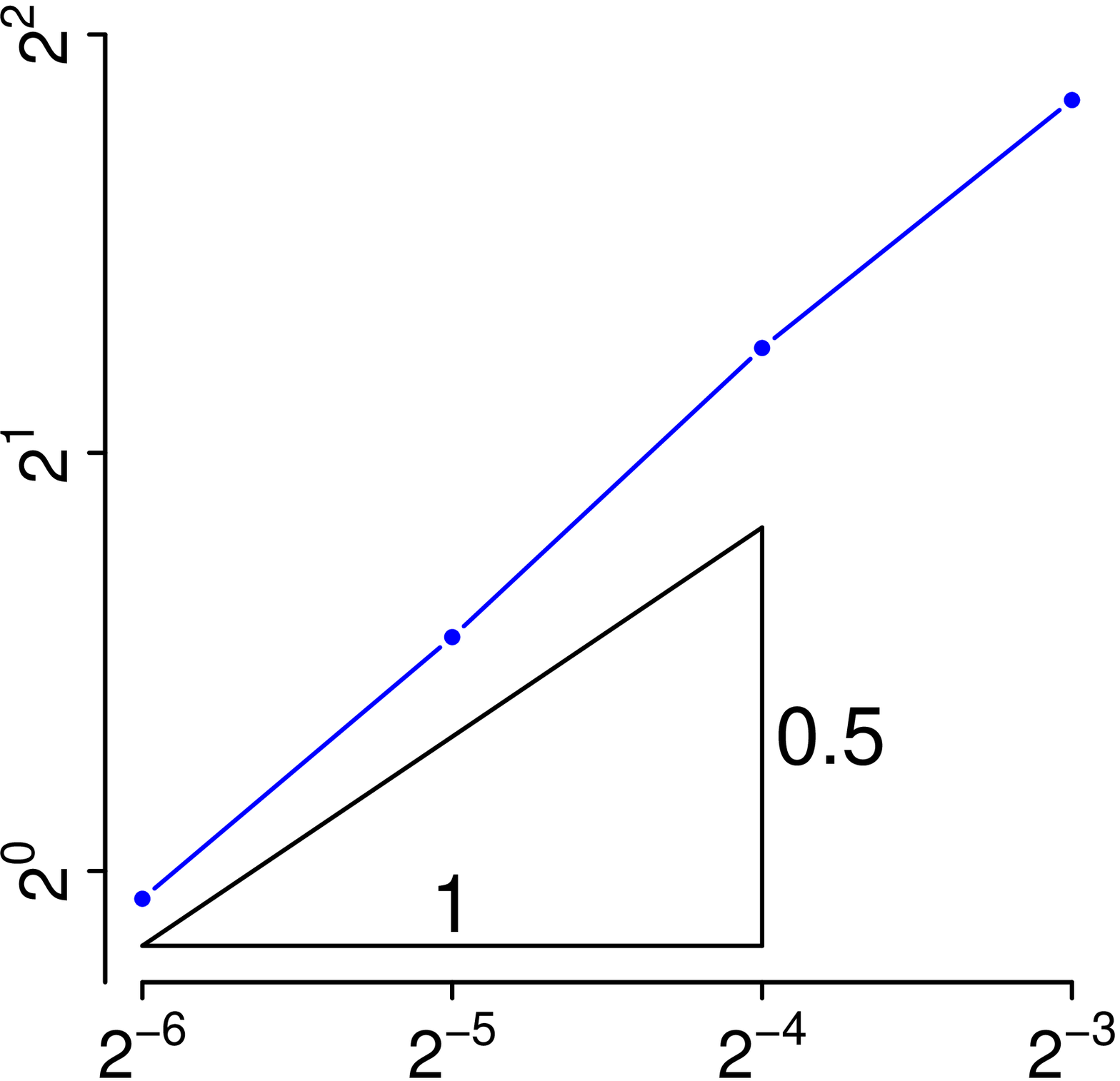}
    \end{center}
    \caption{%
        $H^2$-errors for the soft bulk formulation for %
        $s=0$ over the grid sizes $h=2^{-k}$, $k= 3,4,5$ and $s=1$ over the grid sizes $h=2^{-k}$, $k= 3,\dots,6$ (from left to right) %
        .}
    \label{fig:errorSoftAreaGeneric}
\end{figure}

As before we also computed the
$H^1$- and $L^2$-errors for the same set of example problems.
For the soft curve formulation we again observe that the order of the
$H^1$- and $L^2$-error depicted in
Figure~\ref{fig:H1errorSoftCurveGeneric} and Figure~\ref{fig:L2errorSoftCurveGeneric}, respectively,
is essentially squared in comparison with the $H^2$-error
leading to a convergence order which is approximately $O(h)$.
The situation is different for the soft bulk formulation.
Here, the $H^1$-error depicted in Figure~\ref{fig:H1errorSoftAreaGeneric}
is of order $O(h^{3/2})$ and the $L^2$-error depicted in
Figure~\ref{fig:L2errorSoftAreaGeneric} is of order $O(h^{5/2})$.
While the improved $H^1$- and $L^2$-order is not covered by the presented
theory anyway, we also cannot explain the
surprising difference in the observed order for the symmetric and non-symmetric
example problems.
\begin{figure}[H]
    \begin{center}
        \includegraphics[width=0.31\textwidth]{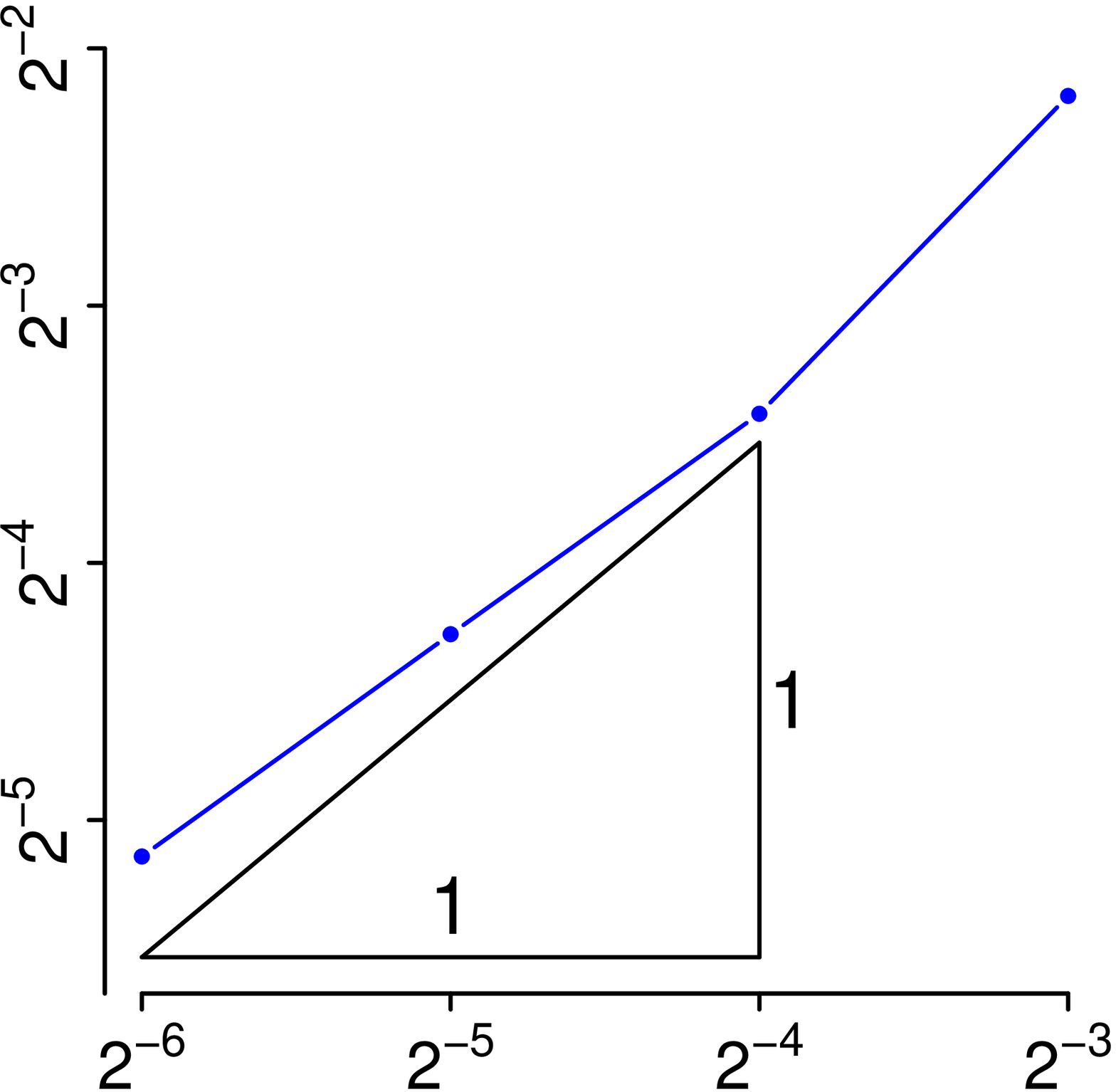}
        \includegraphics[width=0.31\textwidth]{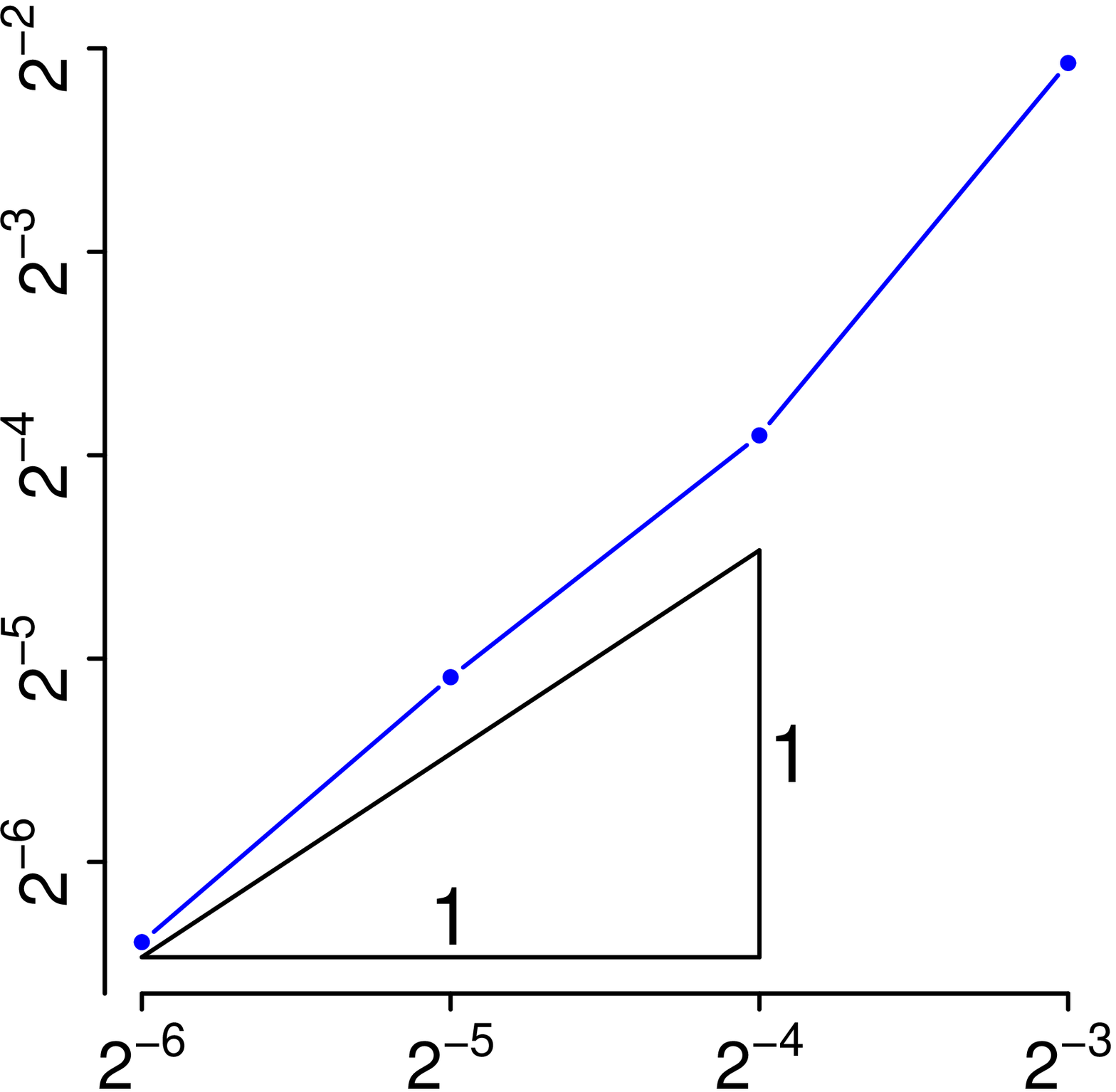}
        \includegraphics[width=0.31\textwidth]{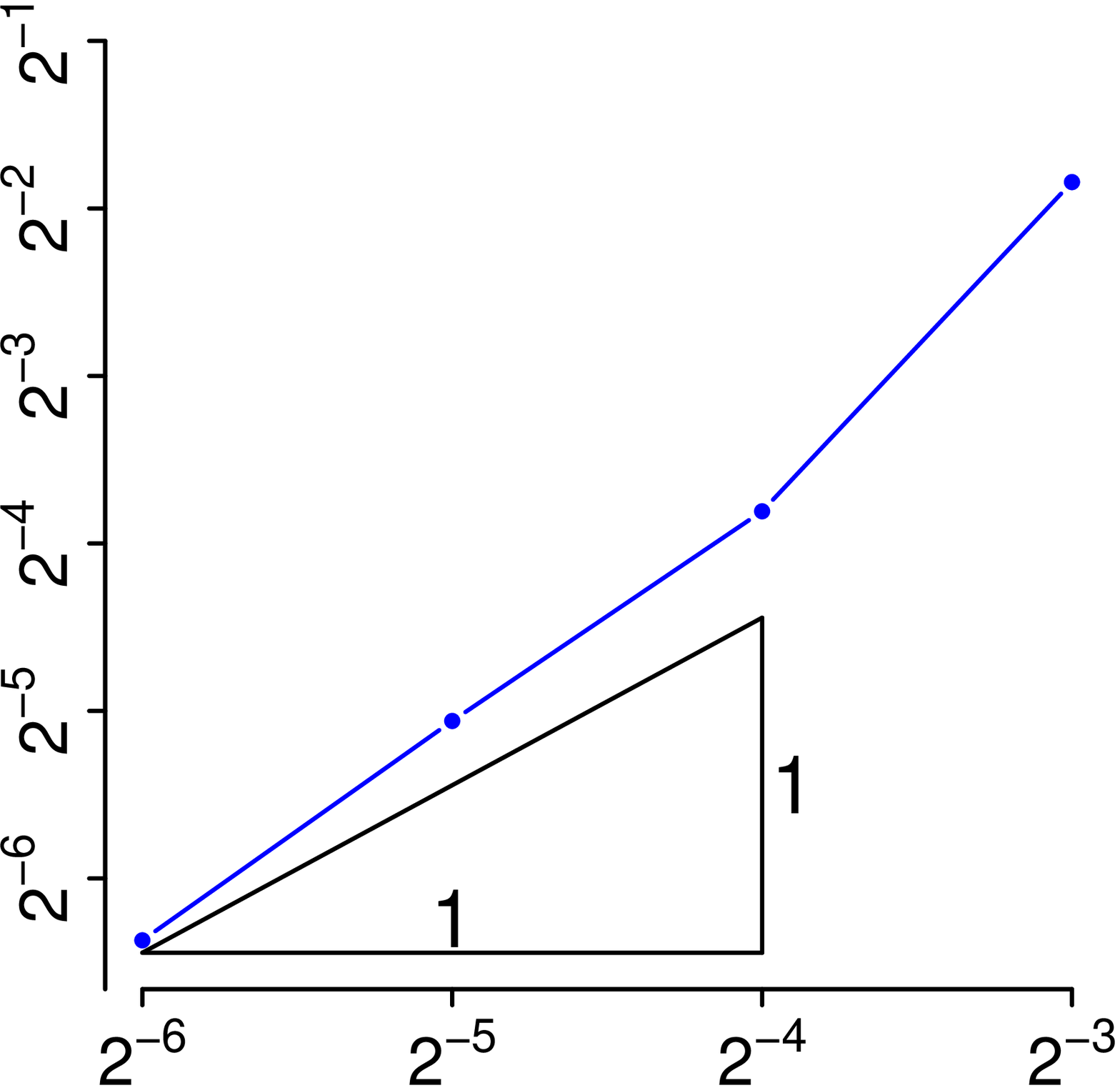}
    \end{center}
    \caption{%
        $H^1$-errors for the soft curve formulation for %
        $\lambda=1$, $\lambda=2$, and $\lambda=3$ (from left to right) %
        over the grid sizes $h=2^{-k}$, $k= 3,\dots,6$.}
    \label{fig:H1errorSoftCurveGeneric}
\end{figure}
\begin{figure}[H]
    \begin{center}
        \includegraphics[width=0.31\textwidth]{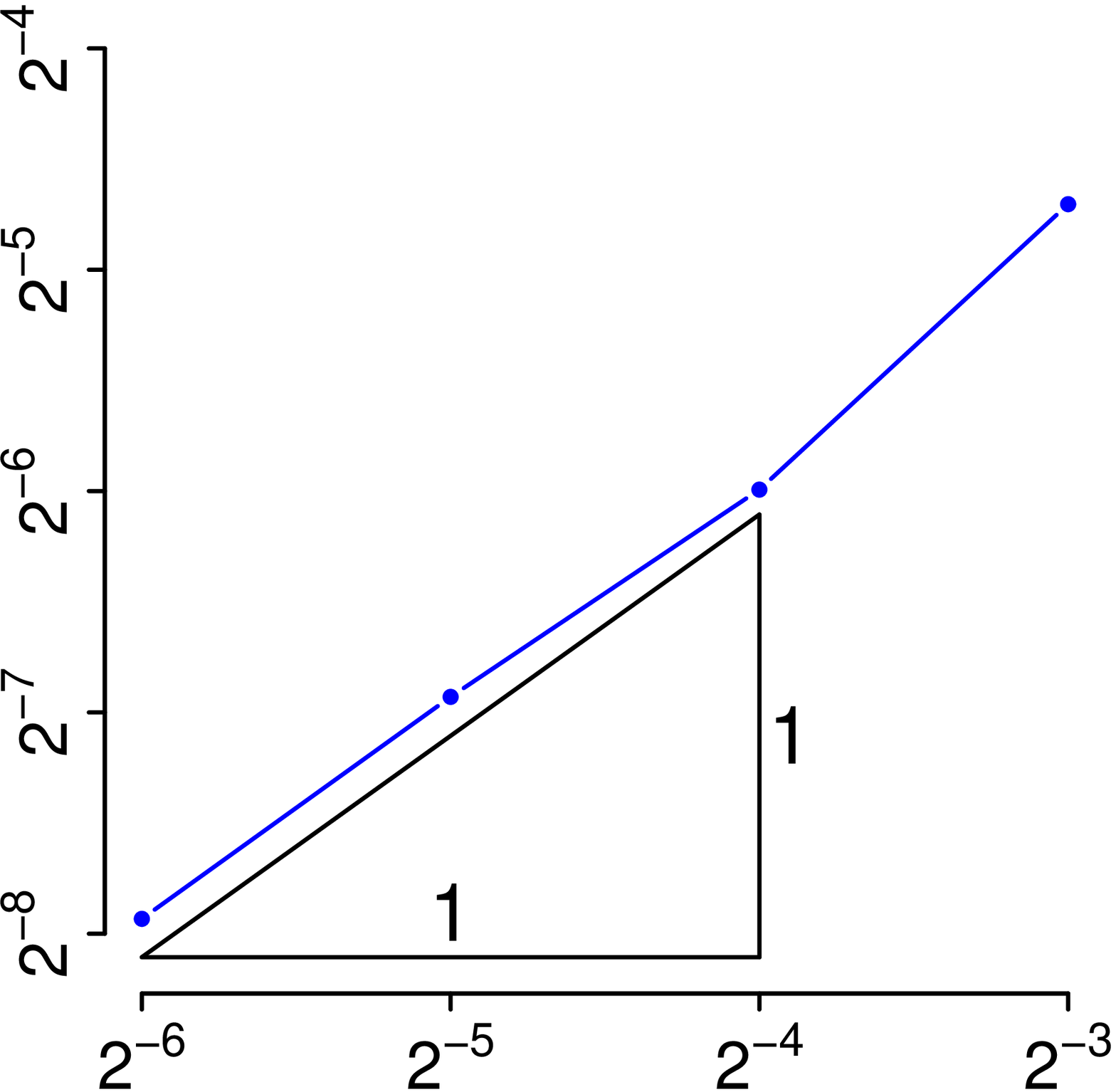}
        \includegraphics[width=0.31\textwidth]{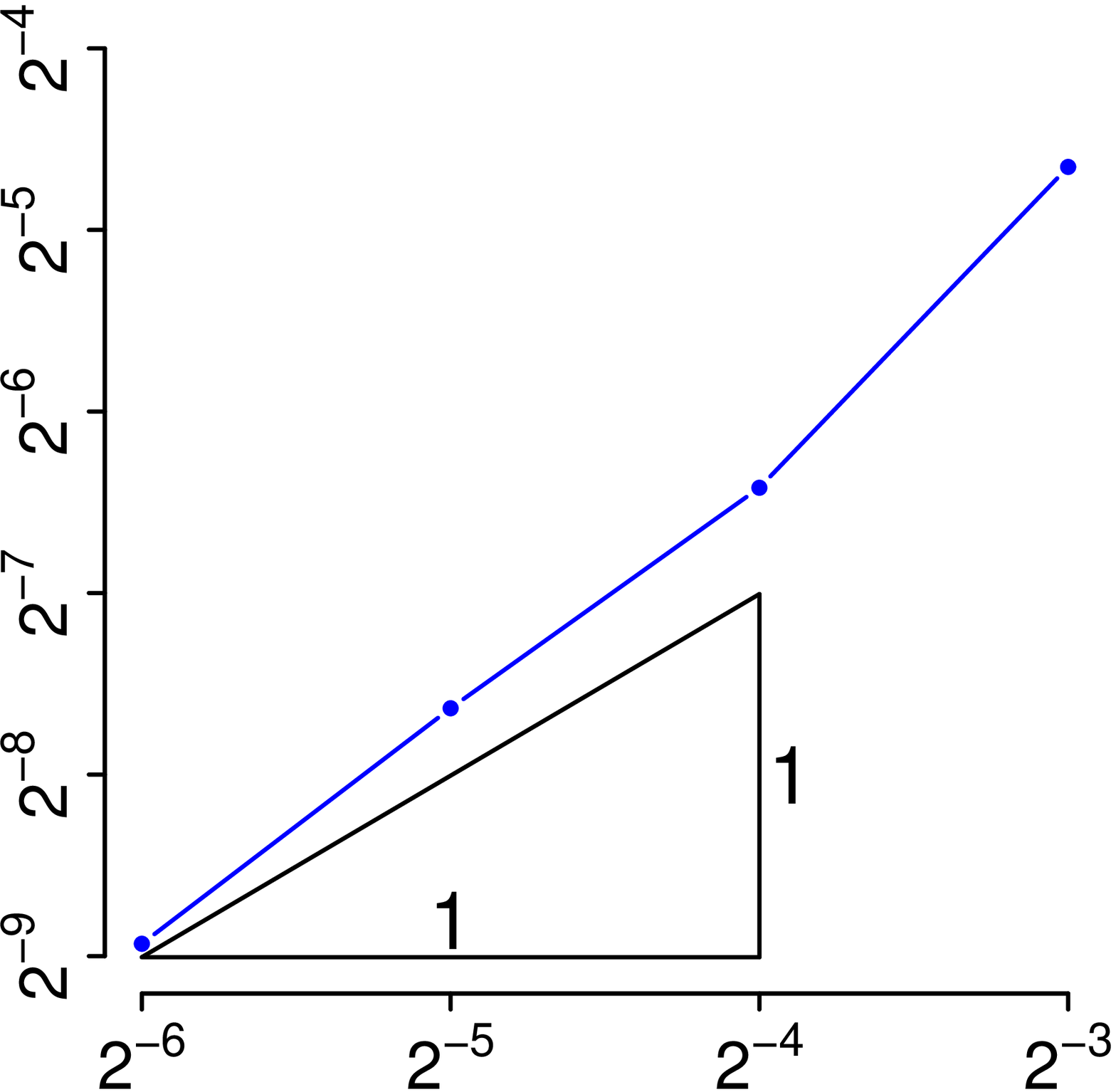}
        \includegraphics[width=0.31\textwidth]{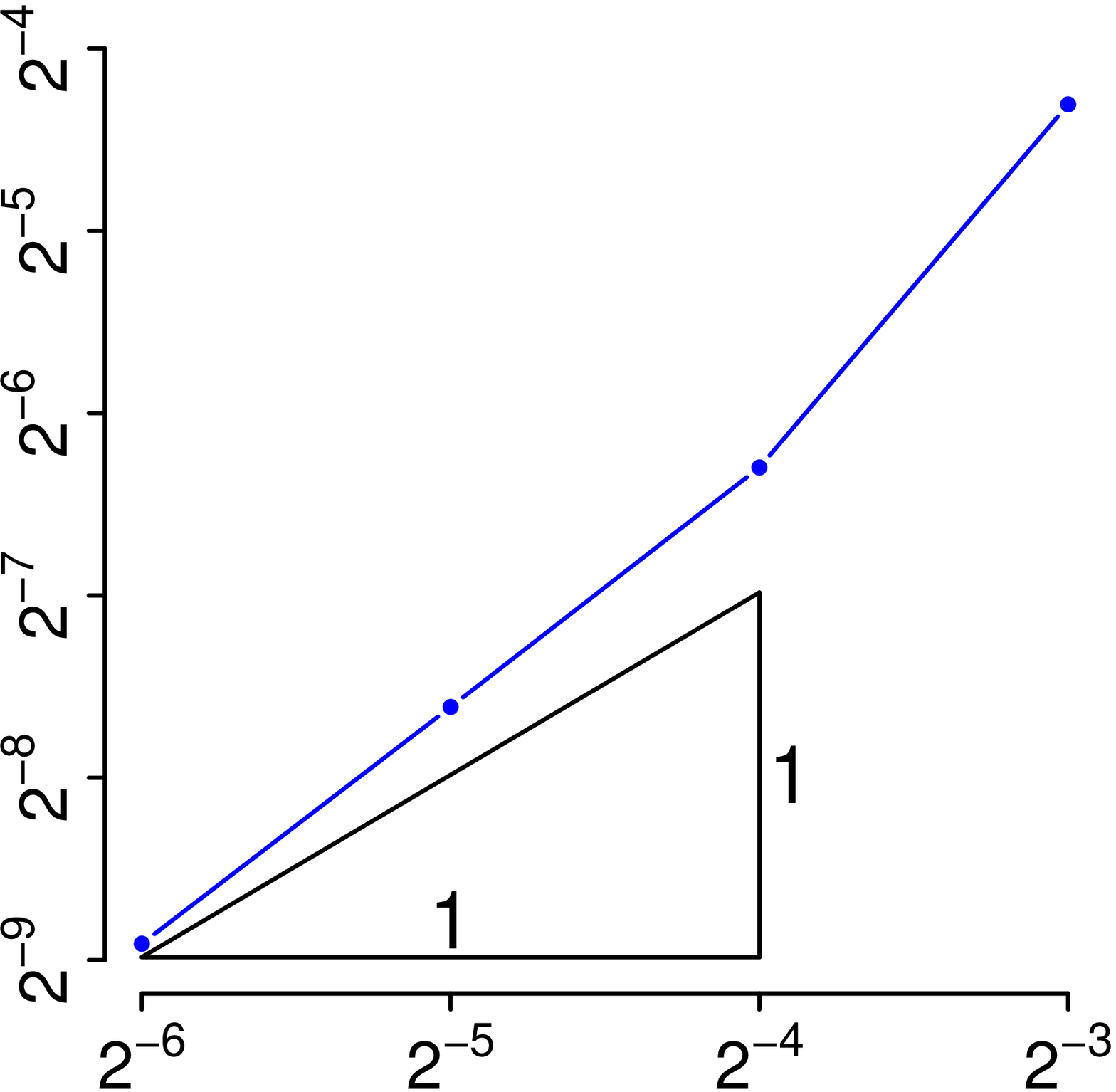}
    \end{center}
    \caption{%
        $L^2$-errors for the soft curve formulation for %
        $\lambda=1$, $\lambda=2$, and $\lambda=3$ (from left to right) %
        over the grid sizes $h=2^{-k}$, $k= 3,\dots,6$.}
    \label{fig:L2errorSoftCurveGeneric}
\end{figure}
\begin{figure}[H]
    \begin{center}
        \includegraphics[width=0.31\textwidth]{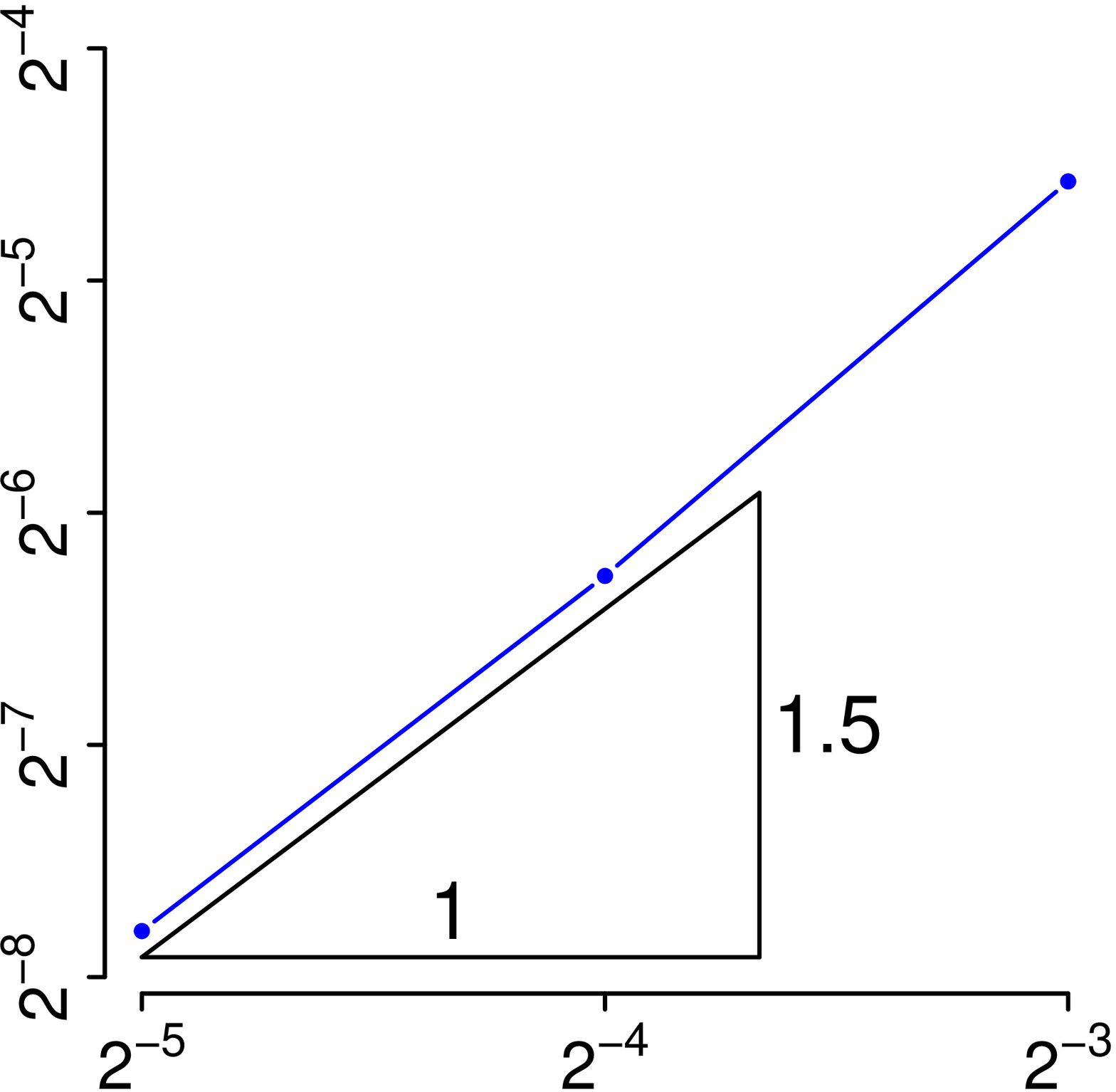}
        \hspace{0.7cm}
        \includegraphics[width=0.31\textwidth]{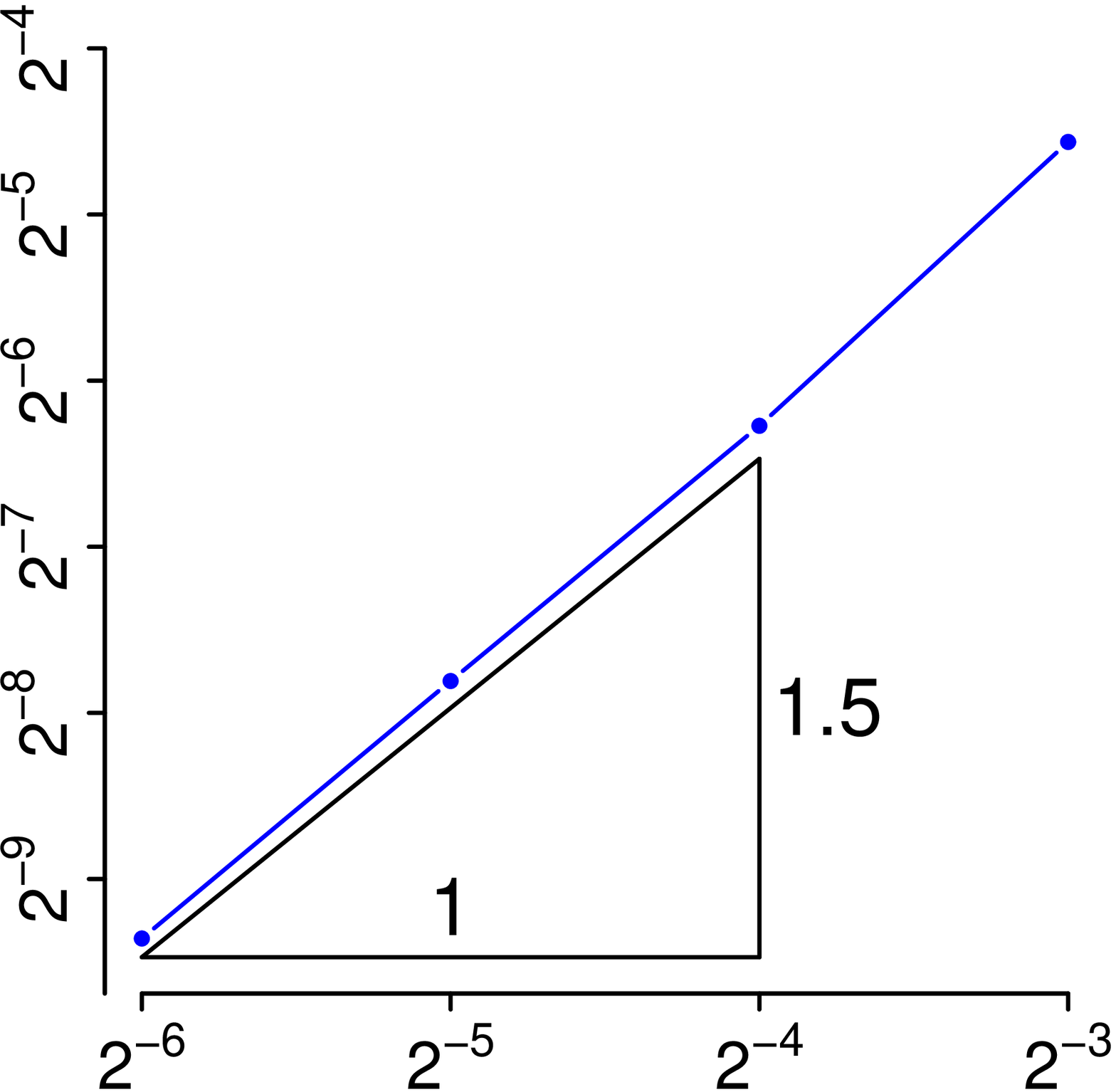}
    \end{center}
    \caption{%
        $H^1$-errors for the soft bulk formulation for %
        $s=0$ over the grid sizes $h=2^{-k}$, $k= 3,4,5$ and $s=1$ over the grid sizes $h=2^{-k}$, $k= 3,\dots,6$ (from left to right) %
        .}
    \label{fig:H1errorSoftAreaGeneric}
\end{figure}
\begin{figure}[H]
    \begin{center}
        \includegraphics[width=0.31\textwidth]{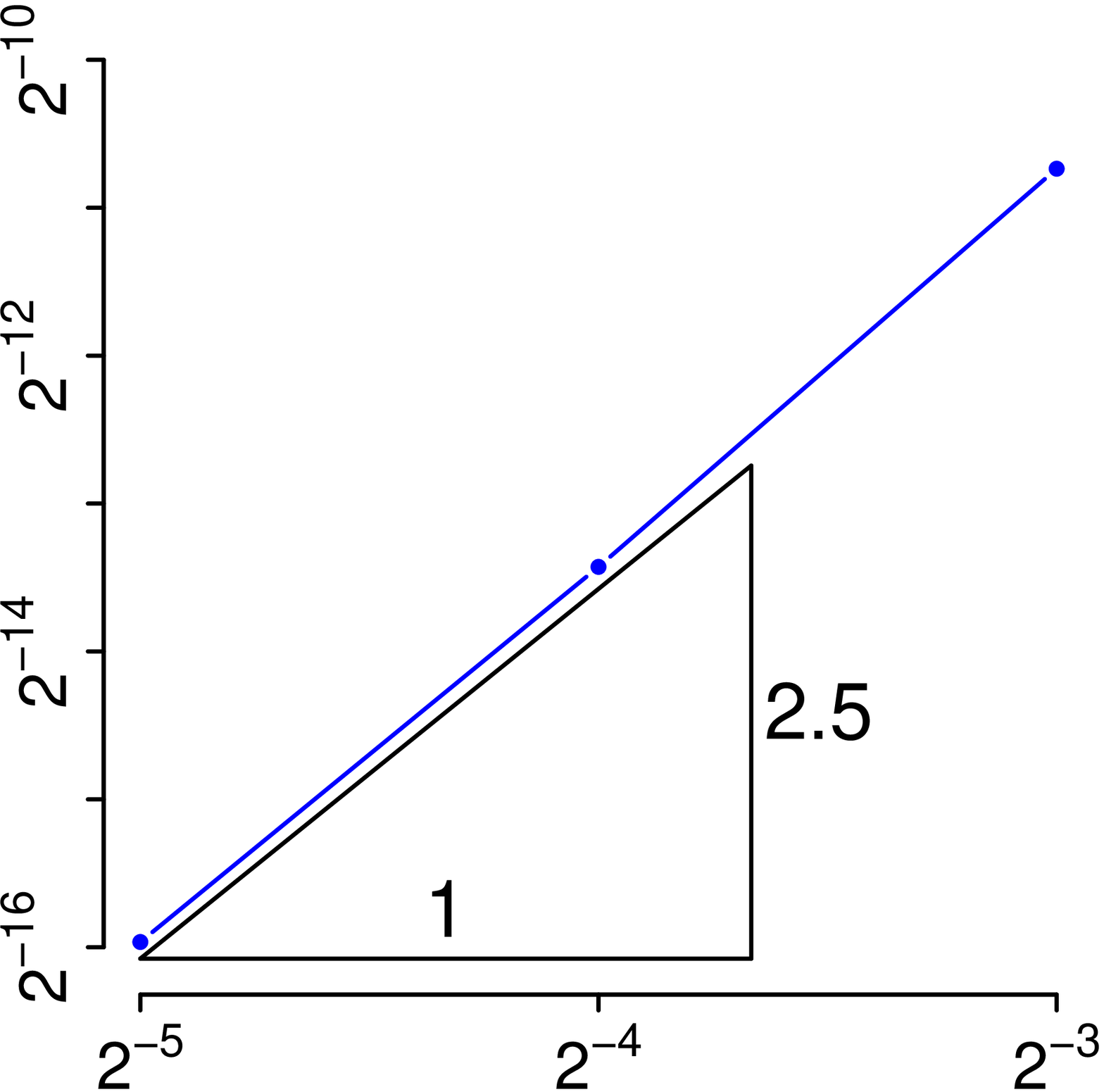}
        \hspace{0.7cm}
        \includegraphics[width=0.31\textwidth]{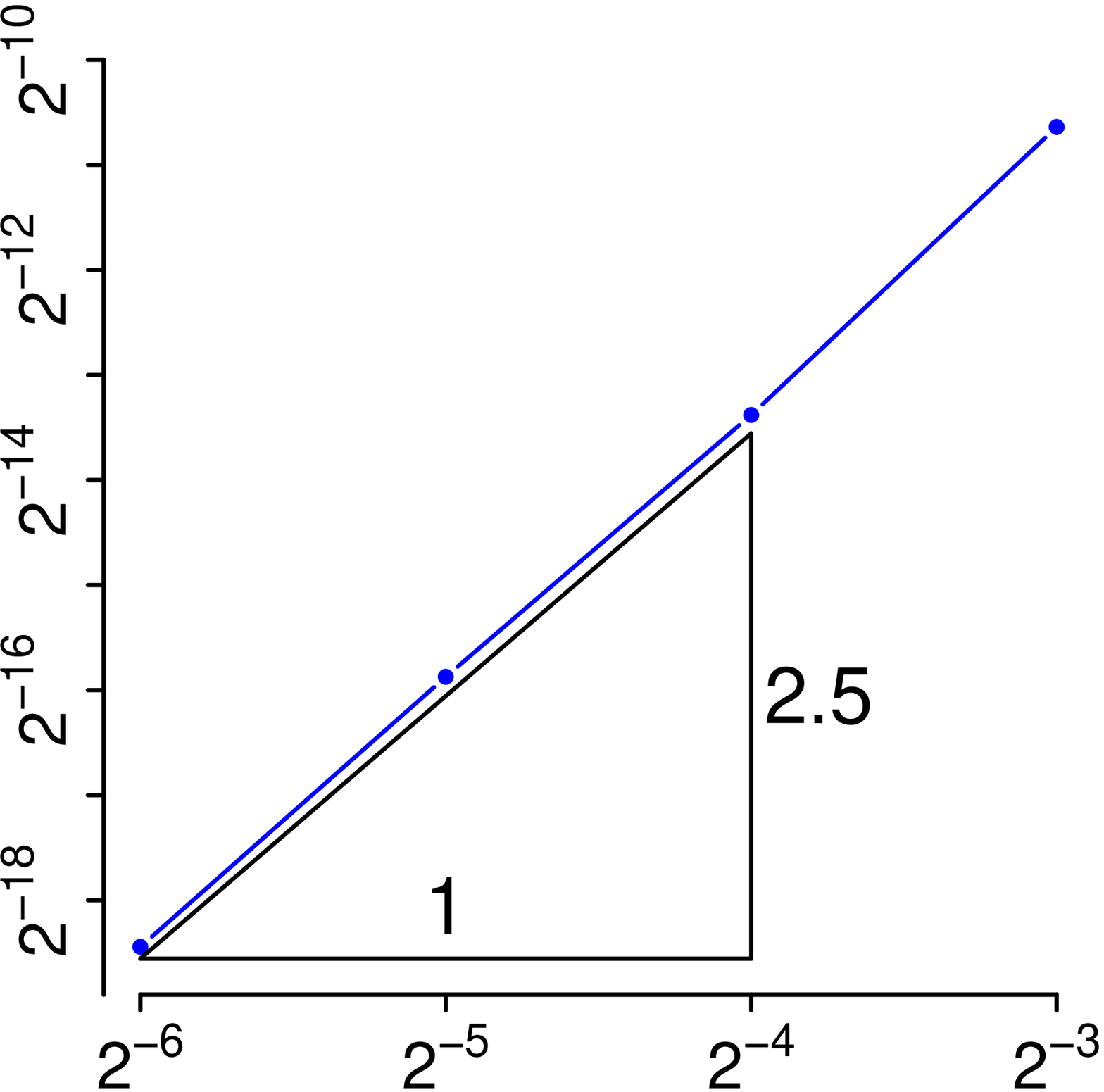}
    \end{center}
    \caption{%
        $L^2$-errors for the soft bulk formulation for %
        $s=0$ over the grid sizes $h=2^{-k}$, $k= 3,4,5$ and $s=1$ over the grid sizes $h=2^{-k}$, $k= 3,\dots,6$ (from left to right) %
        .}
    \label{fig:L2errorSoftAreaGeneric}
\end{figure}


\section{Acknowledgements}


This research has been funded by Freie Universit\"at Berlin and was carried
out in the associated project AP1 "Particles in lipid bilayers" of the
CRC 1114 "Scaling Cascades in Complex Systems".

\section{Appendix}


\begin{proof}[Proof of Lemma~\ref{lem:gluing}]
    Let $u_1 := u|_{\Omega_B}$ and $u_2 := u|_B$ as well as $\Omega_1 := \Omega_B$ and $\Omega_2 := B$.
    Define $\varphi_{\alpha}(x,y) := \frac{\vert D^\alpha u(x) - D^\alpha u(y) \vert^2}{\Vert x-y\Vert^{2s+2}}$.
    For $s \in (0,\frac{1}{2})$ the Sobolev--Slobodeckij norm of $u$ is given by
    \begin{align*}
        \Vert u\Vert_{H^{2+s}(\Omega)}^2
        & = \Vert u\Vert_{H^2(\Omega)}^2 + \max_{\vert\alpha\vert = 2} \Vert D^\alpha u\Vert_{H^s(\Omega)}^2
        \\ & = 
            \Vert u\Vert_{H^2(\Omega)}^2 + \max_{\vert\alpha\vert = 2} \int_{\Omega \times \Omega} \varphi_{\alpha}(x,y) \, \mathrm{d}(x,y)\text{.}
    \end{align*}
    Splitting the domain of integration yields
    \begin{align*}
        \int_{\Omega \times \Omega} \varphi_{\alpha}(x,y) \, \mathrm{d}(x,y)
        & = \Vert D^\alpha u_1 \Vert_{H^s(\Omega_1)}^2 + \Vert D^\alpha u_2\Vert_{H^s(\Omega_2)}^2
            + 2\int_{\Omega_1 \times \Omega_2} \varphi_{\alpha}(x,y) \, \mathrm{d}(x,y)\text{.}
    \end{align*}
    Using the layercake principle we write
    \begin{align*}
        \int_{\Omega_1 \times \Omega_2} \varphi_{\alpha}(x,y)
        = \int_0^\infty \left \vert \left\{ (x,y) \in \Omega_1\times \Omega_2 \mid \varphi_{\alpha}(x,y) \geq t  \right\} \right\vert \, \mathrm{d}t\text{.}
    \end{align*}
    By the assumptions from \cref{sec:notationAndProblems} it holds that $\Omega_1$ and $\Omega_2$ are Lipschitz domains.
    By virtue of the Sobolev embedding theorems we conclude $u_i \in C^2(\overline{\Omega_i})$, see for example \cite[Theorem 5.4]{Adams75}.
    Consequently, the constant $C := \max_{\vert\alpha\vert = 2} \sup_{\Omega_1 \times \Omega_2} \vert D^\alpha u_1(x) - D^\alpha u_2(y)\vert^2 \in \mathbb{R}_{>0}$ exists and thus the implication
    \begin{align*}
        \frac{\vert D^\alpha u(x) - D^\alpha u(y)\vert^2}{\Vert x-y\Vert^{2s+2}} \geq t
        \ \Longrightarrow \ 
        \Vert x-y\Vert \leq C^{\frac{1}{2s+2}} t^{-\frac{1}{2s+2}} =: \widetilde{C} t^{-\frac{1}{2s+2}}
    \end{align*}
    holds for all $(x,y) \in \Omega_1 \times \Omega_2$.
    Let $\Gamma := \partial\Omega_1 \cap \partial \Omega_2$.
    From
    \begin{align*}
        \left\vert \left\{ (x,y) \in \Omega_1\times \Omega_2 \mid \Vert x-y\Vert \leq \widetilde{C} t^{-\frac{1}{2s+2}} \right\} \right\vert
        \leq 4\vert\Gamma\vert (\widetilde{C} t^{-\frac{1}{2s+2}})(\widetilde{C} t^{-\frac{1}{2s+2}})^2
        =: \hat{C} t^{\frac{-3}{2s+2}}
    \end{align*}
    we infer
    \begin{align*}
        & \int_0^\infty \left \vert \left\{ (x,y) \in \Omega_1\times \Omega_2 \mid \varphi_{\alpha}(x,y) \geq t  \right\} \right\vert \, \mathrm{d}t
        \\ & \leq \vert \Omega \vert^2 + \int_1^\infty \left \vert \left\{ (x,y) \in \Omega_1\times \Omega_2 \mid \varphi_{\alpha}(x,y) \geq t  \right\} \right\vert \, \mathrm{d}t
        \\ & \leq \vert \Omega \vert^2 + \int_1^\infty \hat{C} t^{\frac{-3}{2s+2}} \, \mathrm{d}t\text{.}
    \end{align*}
    This expression is finite for $s < \frac{1}{2}$.
    This implies $u \in H^{2+\frac{1}{2}-\delta}(\Omega)$ for all $\delta >0$
    as was to be shown.
\end{proof}

\begin{proof}[Proof of Proposition~\ref{prop:strang}]
    For brevity we define $\wt{u} := \wt{u}^X_\eps$
    and assume without loss of generality that $\eps = (\eps_i)_{i=1,\dots,m} \leq 1$ holds component-wise.
    Let $v \in X$ and define $\wt{w}:= v -\wt{u}$.
    We have by \eqref{eq:helper1611061548}
    \begin{align*}
        \norm{\wt{w}}_{a_\eps}^2
        & = \norm{\wt{w}}_a^2 + \sum_{i=1}^m \frac{1}{\eps_i} \norm{\wt{w}}_{b_i}^2
        \\ & \leq \norm{\wt{w}}_a^2 + \sum_{i=1}^m \frac{1}{\eps_i} \norm{\wt{w}}_{\wt{b}_i}^2 + \sum_{i=1}^m \frac{c_i}{\eps_i} \norm{\wt{w}}^2\text{.}
    \end{align*}
    Using the continuity of $a$, the coercivity of $\wt{a}_\eps$ and $\eps \leq 1$ this leads to the inequality
    \begin{align}\label{eq:helper1611061601}
        \norm{\wt{w}}_{a_\eps}^2
        & \leq \left(\frac{\norm{a}}{\wt{\alpha}} + 1 + \sum_{i=1}^m \frac{c_i}{\eps_i \wt{\alpha}} \right) \norm{\wt{w}}_{\wt{a}_\eps}^2\text{.}
    \end{align}
    On the other hand we have the equality
    \begin{align*}
        \norm{\wt{w}}_{\wt{a}_\eps}^2
        & = \wt{a}_\eps(v,\wt{w}) - \wt{a}_\eps(\wt{u},\wt{w})
            + \ell_\eps(\wt{w}) - a_\eps(u,\wt{w})
            + a_\eps(v,\wt{w}) - a_\eps(v,\wt{w})
    \end{align*}
    which by application of \eqref{eq:abstractInexactPenaltyVariation},
    Cauchy--Schwarz inequality, coercivity of $a_\eps$, definition of
    $\wt{a}_\eps$ and $\wt{\ell}_\eps$ and $\eps \leq 1$ yields
    \begin{align}\label{eq:helper1611061616}
        \begin{aligned}
            \norm{\wt{w}}_{\wt{a}_\eps}^2
            & = a_\eps(v-u,\wt{w}) + (\wt{a}_\eps-a_\eps)(v,\wt{w}) + (\ell_\eps-\wt{\ell}_\eps)(\wt{w})
            \\ &
                \leq \norm{v-u}_{a_\eps} \norm{\wt{w}}_{a_\eps} 
             \\ & 
                \quad + \frac{\abs{(\wt{a}-a)(v,\wt{w})} + \sum_{i=1}^m \frac{1}{\eps_i} \abs{(\wt{b}_i-b_i)(u-v,\wt{w})} + \abs{(\ell-\wt{\ell})(\wt{w})}}{\sqrt{\alpha} \norm{\wt{w}}} \norm{\wt{w}}_{a_\eps}
        \end{aligned}
    \end{align}
    Then from the triangle inequality and \eqref{eq:helper1611061616} inserted into \eqref{eq:helper1611061601} we get
    \begin{align*}
        \norm{u - \wt{u}}_{a_\eps}
        & \leq \norm{u - v}_{a_\eps} + \norm{v - \wt{u}}_{a_\eps}
        \\ & \leq \norm{u - v}_{a_\eps} + \left(\frac{\norm{a}}{\wt{\alpha}} + 1 + \sum_{i=1}^m \frac{c_i}{\eps_i \wt{\alpha}} \right) \cdot \left( \norm{v-u}_{a_\eps} + \phantom{\frac{\sum_{i=1}^m \frac{1}{\eps_i}(\wt{b}}{\norm{\wt{w}}}}\right.
        \\ & \qquad + \left. \frac{\abs{(\wt{a}-a)(v,\wt{w})} + \sum_{i=1}^m \frac{1}{\eps_i} \abs{(\wt{b}_i-b_i)(u-v,\wt{w})} + \abs{(\ell-\wt{\ell})(\wt{w})}}{\sqrt{\alpha} \norm{\wt{w}}} \right)
    \end{align*}
    which proves the statement by taking the infimum over all $v \in X$ and replacing $\wt{w} = v - \wt{u}$ by the supremum over all $w \in X$.
\end{proof}

\bibliography{paper}

\ifimanumstyle
    \bibliographystyle{IMANUM-BIB}
\else
    \bibliographystyle{plain}
\fi

\end{document}